\documentclass[reqno,11pt]{amsart}
\newif\ifdraft
\draftfalse % replace with \drafttrue to see detailed proofs
\usepackage{amscd,amssymb,amsmath,amsthm}
\usepackage{graphicx}
\usepackage{color}
\usepackage[all,cmtip]{xy}
\usepackage{amsfonts}
\usepackage{verbatim}
\usepackage{enumerate}

\usepackage{graphicx}
\numberwithin{equation}{section}

\makeatletter
\newcommand\footnoteref[1]{\protected@xdef\@thefnmark{\ref{#1}}\@footnotemark}
\makeatother
\newcounter{mycounter}

\theoremstyle{plain}
\newtheorem{proposition}{Proposition}
\newtheorem{theorem}[proposition]{Theorem}

\newtheorem*{conjecture*}{Conjecture}
\newtheorem{definition}[proposition]{Definition}
\newtheorem{corollary}[proposition]{Corollary}
\newtheorem{lemma}[proposition]{Lemma}
\newtheorem{remark}[proposition]{Remark}

\newtheorem{example}[proposition]{Example}

\newtheorem{proposition-definition}[proposition]{Proposition/Definition}
\newtheorem*{proposition*}{Proposition}
\newtheorem*{theorem*}{Theorem}
\newtheorem*{maintheorem*}{Main Theorem}
\newtheorem*{maincorollary*}{Main Corollary}
\newtheorem*{corollary*}{Corollary}
\newtheorem*{lemma*}{Lemma}
\newtheorem*{remark*}{Remark}
\newtheorem*{remarks*}{Remarks}
\newtheorem*{definition*}{Definition}
\newtheorem*{example*}{Example}
\newtheorem*{examples*}{Examples}
\newtheorem*{criterion*}{Generation Criterion}
\newtheorem*{explanation*}{Explanation}

\numberwithin{proposition}{section}
\numberwithin{equation}{section}

\def\co{\colon\thinspace}

\newcommand{\N}{\mathbb{N}}
\newcommand{\Z}{\mathbb{Z}}
\newcommand{\Q}{\mathbb{Q}}
\newcommand{\R}{\mathbb{R}}
\newcommand{\C}{\mathbb{C}}
\newcommand{\K}{\mathbb{K}}

\newcommand{\coker}{\mathrm{coker}\,}

\renewcommand{\P}{\mathbb{P}}

\newcommand{\inn}{\mathrm{in}}
\newcommand{\out}{\mathrm{out}}

\newcommand{\OO}{\mathcal O}

\usepackage{ulem}
\def\red#1{{\textcolor{red}{#1}}}
\newcommand{\T}{\mathbb T}

\usepackage{longtable}
\usepackage[all,cmtip]{xy}
\usepackage{upref}
\usepackage{hyperref}

\renewcommand{\red}{}

\begin{document}

\title[Invariance of symplectic cohomology]
{Invariance of symplectic cohomology and\\ twisted cotangent bundles over surfaces}

\author{Gabriele Benedetti and Alexander F. Ritter}
\address{Gabriele Benedetti, Mathematisches Institut, Universit\"at Heidelberg, Germany.\newline \indent Alexander F. Ritter, Mathematical Institute, University of Oxford, England.}
\email{gbenedetti@mathi.uni-heidelberg.de \qquad and \qquad ritter@maths.ox.ac.uk}
\date{version: \today}

%%%%%%%%%%%%%%%%%%%%%%%%%%%%%%%%%%%%%%%%%%%%%%%%%%%%%%%%%%%
%%%%%%%%%%%%%%%%%%%%%%%%%%%%%%%%%%%%%%%%%%%%%%%%%%%%%%%%%%%
%%%%%%%%%%%%%%%%%%%%%%%%%%%%%%%%%%%
%%%%%%%%%%%%%%%%%%%%%%%%%%%%%%%%%%%
\begin{abstract}
We prove that symplectic cohomology for open convex symplectic manifolds is invariant when the symplectic form undergoes deformations which may be non-exact and non-compactly supported, provided one uses the correct local system of coefficients in Floer theory. As a sample application beyond the Liouville setup, we describe in detail the symplectic cohomology for disc bundles in the twisted cotangent bundle of surfaces, and we deduce existence results for periodic magnetic geodesics on surfaces. In particular, we show the existence of geometrically distinct orbits by exploiting properties of the BV-operator on symplectic cohomology.
\end{abstract}
\maketitle
%\setcounter{tocdepth}{2}
%\tableofcontents
%
%
%\\[1mm] $\bullet$
%
%
%
%
%
%
%%%%%%%%%%%%%%%%%%%%%%%%%%%%%%%%%%%
\section{Introduction}
\label{Section Introduction}
%%%%%%%%%%%%%%%%%%%%%%%%%%%%%%%%%%%
%
%
Symplectic cohomology is an invariant of non-compact symplectic manifolds, whose importance both in dynamical applications and in homological mirror symmetry has become increasingly clear in recent literature. This invariant is constructed using Hamiltonian Floer cohomology, which is surveyed in Salamon's lecture notes \cite{Salamon} for closed symplectic manifolds: in that case, the invariant recovers the quantum cohomology, whilst in the non-compact setup the invariant is much richer due to the Hamiltonian dynamics at infinity. The surveys by Oancea \cite{OanceaEnsaios}, Seidel \cite{SeidelBiased} and Abouzaid \cite{Abouzaid-survey} review many of the developments relating to symplectic cohomology.

Its origins in the work of Cieliebak-Floer-Hofer-Wysocki \cite{CFHW} and Viterbo \cite{Viterbo} were motivated especially by dynamical applications, specifically existence theorems for closed Hamiltonian orbits. Later on, the relationship between this invariant and the study of fillings of contact manifolds was explored, starting from the groundbreaking work of Bourgeois-Oancea \cite{Bourgeois-Oancea-Inventiones,Bourgeois-Oancea-S1,Kwon-Koert,Gutt}.

We will postpone to Section \ref{ss:applications} the discussion of cotangent bundles, in which case there is an abundance of literature on how symplectic cohomology has been used to prove the existence of closed geodesics and magnetic geodesics.

Symplectic cohomology has also been used effectively to obtain obstructions on the existence of exact Lagrangian submanifolds, a very difficult problem in symplectic topology, via Viterbo's functoriality theorem \cite{Viterbo}. In homological mirror symmetry, the crucial role played by symplectic cohomology goes back to Seidel's ICM talk \cite{SeidelICM}, in particular the open-closed string map which relates the Hochschild homology of the wrapped Fukaya category to the symplectic cohomology has become a crucial tool to prove theorems about generators for Fukaya categories due to the work of Abouzaid \cite{Abouzaid}, which was extended also to the non-exact setup by Ritter-Smith \cite{RitterSmith}.

Much of the symplectic literature on non-compact symplectic manifolds is focused on the case where the symplectic form is globally exact. This is because it simplifies the Floer theory considerably, and cotangent bundles $(T^*N,d\theta)$ were a driving motivating example. Interest in the non-exact setting \red{stems} not only from twisted cotangent bundles \red{$(T^*N,d\theta+\pi^*\sigma)$} and magnetic geodesics \red{(see Section \ref{ss:applications})}, but also from the fact that non-compact K\"{a}hler manifolds arising in algebraic geometry are very rarely exact as this would force closed holomorphic curves to be constant. We also wish to avoid the weaker assumption, often encountered in Floer theory, that $\omega$ is \textit{aspherical}, meaning $\omega$ vanishes on $\pi_2(M)$, as this would rule out any K\"{a}hler manifold that contains a non-trivial holomorphic sphere. Non-exactness allows Gromov-Witten theory to play an interesting role in Floer theory \cite{Ritter2,Ritter4,RitterSmith} \red{and} one interpretation of symplectic cohomology is as a generalisation of the quantum cohomology $QH^*(M,\omega)$ to non-compact settings \cite{Ritter6}.

Even in situations where the non-compact symplectic manifold $(M,d\theta)$ is exact, one can obtain substantial applications in symplectic topology by considering how the invariants change upon deforming the symplectic form $d\theta$. For instance, in the work of the second author \cite{Ritter1,Ritter2} such a deformation gave rise to new obstructions to the existence of exact Lagrangian submanifolds in cotangent bundles and in ALE spaces. In situations where $(M,\omega)$ is non-exact, it can also be beneficial to deform $\omega$. For instance, for non-compact Fano varieties in \cite{Ritter6} a deformation of the monotone toric K\"{a}hler form forced symplectic cohomology to become a semi-simple algebra, and this combined with the use of the open-closed string map gave rise to generation theorems for the wrapped Fukaya category.

To avoid making the introduction too technical, Section \ref{Section Summary of results} will be a summary of the precise definitions and deformation theorems which we now summarise in looser terms.
Our paper is concerned with the setup of (typically non-exact) symplectic manifolds which are exact at infinity.
We will say $(M,\omega,\theta)$ is a \textit{convex manifold} (Definition \ref{d:convex}) to mean that $(M,\omega)$ is an open symplectic manifold admitting an exhausting \red{(namely proper and bounded below)} function $h:M\to \R$, where $\theta$ is a $1$-form defined on $M^{\out}:=\{h\geq 0\}$ with
\begin{equation}\label{Equation Intro convex condition}
\omega=d\theta\quad\text{and} \quad \theta(X_h)>0 \quad \text{holds on } M^{\out}.
\end{equation}
Here $X_h$ is the Hamiltonian vector field, $\omega(\cdot,X_h)=dh$. An \textit{isomorphism of convex manifolds} is a symplectomorphism which preserves the $1$-form at infinity. On $M^{\out}$ there is a Liouville vector field $Z$ via $\theta=\iota_Z \omega$, and positivity in \eqref{Equation Intro convex condition} is equivalent to \red{$Z$ pointing out of the domains $\{h\leq a\}$ for all $a\geq0$. Below we use the notation
\[
\Sigma:=\{h=0\},\qquad M^{\inn}:=\{h\leq 0\}.	
\]}
We do not impose that $Z$ is positively integrable \red{since} one can always embed \red{$(M,\omega,\theta)$ via the $Z$-flow into the \textit{completion} $(\hat M,\hat\omega,\hat\theta)$, which at infinity is identifiable with $(\Sigma \times [0,\infty),d(e^r\alpha),e^r\alpha)$ for some contact form $\alpha$ on $\Sigma$ and $r$-coordinate on $[0,\infty)$. }

The symplectic cohomology of a convex manifold is the direct limit 
\begin{equation}\label{Equation Intro SH is direct lim}
SH^*(M,\omega,\theta)=\varinjlim HF^*(H)
\end{equation}
via Floer continuation maps of the Floer cohomologies computed for Hamiltonians $H:M\to \R$ which at infinity are ``linear'' of larger and larger slopes. Linearity refers to a radial coordinate $R=e^r$ determined by the choice of $h$ ($r$ is the time flown in direction $Z$ starting from $\Sigma$). \red{The coefficients of the Floer chain complexes are taken in the Novikov field $\Lambda$ involving series in a formal variable $t$, see \eqref{Equation Novikov ring}. For any cohomology class $\zeta\in H^1(\mathcal LM)$ on the free loop space $\mathcal LM$ of $M$, one can use $\zeta$ to twist the differential in the Floer chain complexes to define the so-called twisted symplectic cohomology $SH^*(M,\omega,\theta)_\zeta$. However the twisted differential, and thus $SH^*(M,\omega,\theta)_\zeta$, may not always be well-defined due to a lack of convergence in the Novikov field $\Lambda$ (Remark \ref{Remark Twisting variable s}).}

\red{It is crucial to understand invariance of symplectic cohomology under certain natural operations.} For example, as shown by the second author  
(following the proof in the Liouville case \cite{SeidelBiased}) symplectic cohomology is invariant under isomorphism.
\begin{theorem}[\cite{Ritter2}, Theorem 8\footnote{
		There is an erratum in the proof \cite[Lemma 7]{Ritter2} (analogously to
		\cite[(3c)(3.21)]{SeidelBiased}). The correct inequality is $|d(\partial_s f_s)\cdot \partial_s u| \leq \rho^{-1/2} C |\partial_s u|$.
		% (so $\rho^{-1/2}$ rather than $\rho$). 
		By Cauchy-Schwarz, $\rho^{-1/2}C|\partial_s u| \leq C(\varepsilon^{-1} \rho + \varepsilon|\partial_s u|^2)$ (for any $\varepsilon>0$).
		The next line is: $\Delta \rho + $(first order terms) 
		%$\geq |\partial_s u|^2 - \rho \partial_s h_s' - \rho(h_s'C+C) 
		%- C(\varepsilon \rho + \varepsilon^{-1}|\partial_s u|^2)$
		% since only the -C|\partial_s u| term needs to be changed with the new bound above
		%
		$
		%\geq (1-C\varepsilon)|\partial_s u|^2 - \rho \partial_s h_s' - \rho(h_s'C+C) 
		%- C\varepsilon^{-1} \rho
		%
		\geq (1-C\varepsilon)|\partial_s u|^2
		- \rho (\partial_s h_s' + h_s'C+C+C\varepsilon^{-1})
		$. To make the first term non-negative, we pick $\varepsilon=1/C$. The rest of the proof holds as written.
	}]\label{Theorem invariance under iso of SH}
	Any isomorphism $\varphi:(M_0,\omega_0,\theta_0)\to(M_1,\omega_1,\theta_1)$ naturally induces an isomorphism $\varphi_*\co SH^*(M_0,\omega_0,\theta_0) \to SH^*(M_1,\omega_1,\theta_1)$. The same holds for twisted symplectic cohomology if it is well-defined.
\end{theorem}
\begin{remark}
We always tacitly assume that 
$(M,\omega)$ is \textit{weakly monotone} (see Section \ref{Section Summary of results}), which ensures $HF^*(H)$ is well-defined by the methods of Hofer-Salamon \cite{Hofer-Salamon} rather than having to appeal to more advanced machinery, such as Kuranishi structures or Polyfolds.
\end{remark}

We emphasize that the function $h$ is not fixed in the definition of convex manifold, but it enters crucially in the construction of symplectic cohomology, since it determines the class of Hamiltonians. An implicit consequence of Theorem \ref{Theorem invariance under iso of SH} is that the symplectic cohomology is independent of $h$. Such statements will be familiar to experts from the exact setup, but some care is required in the non-exact setup as the surprisingly strong invariance result of the exact case \cite{Abouzaid-Seidel-Altering} was only possible due to the fact that $Z$ was globally defined and that the Floer action functional was single-valued, both of which fail in the non-exact setting.

By ``exact setting'', in which case $M$ is called a \textit{Liouville manifold}, we mean that $\theta$ additionally extends to a global primitive of $\omega$ on the whole $M$. This is stronger than asking that $\omega$ is globally exact, as $\theta$ may fail to extend. When this fails, $M$ is called \textit{Quasi-Liouville}; such examples arise from $T^*\T^2$ endowed with certain twisted symplectic form (see Theorem \ref{t:cmp}). We show in Lemma \ref{Lemma Liouville condition} that the obstruction for a convex manifold to be Liouville is the \textit{relative class}
\begin{equation}\label{Equation intro relative class}
[\omega,\theta]\in H^2(M,M^{\out})\cong H^2_c(M).
\end{equation}
\red{That isomorphism is canonical by Equation \eqref{e:isocoh} which implies that the class is in fact independent of the choice of exhausting function $h$ (and thus the choice of $M^{\out}$).}

\red{As we will see next, this class plays a special role, when we consider a \textit{deformation} of a convex manifold, which we now define. By deformation we mean a familiy $s\mapsto (M,\omega_s,\theta_s)$, $s\in[0,1]$ of convex manifolds such that the corresponding family of exhausting functions $h_s$ is given by $h_s:=H(s,\cdot)$ for some exhausting function $H:[0,1]\times M\to\R$. We will denote by $M^\out_s$ and $M^\inn_s$ the outer and inner part with respect to $h_s$ and by $(\hat M_s,\hat\omega_s,\hat\theta_s)$ the completion of $(M,\omega_s,\theta_s)$.}
\red{By \eqref{Equation intro relative class}, we can identify the cohomology classes $[\omega_s,\theta_s]$ as elements of the same vector space $H^2_c(M)$. The next result shows that if, under this identification, the relative class is constant along the deformation, all convex manifolds in the deformation (and hence their symplectic cohomologies) are isomorphic. This fact is used, for example, in the applications of Theorem \ref{t:nonexact}.}
\begin{theorem}\label{t:invariance1}
Let $(M,\omega_s,\theta_s)$ be a deformation \red{such that $[\omega_s,\theta_s]\in H^2_c(M)$ is independent of $s$. Then there is an isomorphism
\[
\varphi:(\hat M_0,\hat \omega_0,\hat\theta_0) \to (\hat M_1, \hat\omega_1,\hat\theta_1).
\]}
In particular, $SH^*(M,\omega_0,\theta_0) \cong SH^*(M,\omega_1,\theta_1)$\red{, and }the same holds for twisted symplectic cohomology if it is well-defined.
\end{theorem}
\red{The proof is in two steps. In the first step, we assume that the deformation is $C^0$-small and we construct the map in the statement as a composition $\varphi=\varphi_2\circ\varphi_1$. To define $\varphi_2:\hat M_0\to \hat M_1$, we use a trick from Seidel-Smith \cite{Seidel-Smith}: on $\hat M^{\mathrm{out}}_0$, $\varphi_2$ is obtained by applying Gray stability to the deformation of contact manifolds $(\Sigma,\theta_s|_{\Sigma})$ and then extending it in a canonical way to $M^{\mathrm{in}}$. As a result, $\varphi^*_2\hat\theta_1=\hat\theta_0$ and $\varphi^*_2\hat\omega_1=\hat\omega_0+\beta$, where $\beta$ is a closed 2-form with compact support. The class $[\beta]\in H^2_c(\hat M_0)$ vanishes if and only if the relative class $[\omega_s,\theta_s]$ is constant. In this case, we build an isomorphism $\varphi_1:(\hat M_0,\hat \omega_0,\hat\theta_0) \to (\hat M_0, \hat\omega_0+\beta,\hat\theta_0)$ as the time-one map of a flow obtained by a Moser argument.}
	
\red{In the second step, we break the interval $[0,1]$ into subintervals on which the deformation is $C^0$-small, and we obtain the map $\varphi$ as a composition of the maps obtained from the first step. The subdivision of the deformation into small ones (so-called ``adiabatic steps") is an idea which is frequently used in establishing isomorphisms of symplectic cohomology along deformations and we learned it from \cite[Section 2.4]{Bae-Frauenfelder}}.

\red{We now want to go beyond Theorem \ref{t:invariance1} by considering deformations in which the relative class varies. In this case, the existence of an isomorphism $\varphi_1$ isotopic to the identity is obstructed (see Lemma \ref{Lemma varphit preserves relative class}). If the transgression $\tau(\omega_s)\in H^1(\mathcal LM)$ also varies (where $\mathcal{L}M$ is the free loop space of $M$), one cannot even expect $SH^*(M,\omega_0,\theta_0)$ to be isomorphic to $SH^*(M,\omega_1,\theta_1)$ as they involve different Novikov fields. To off-set that, one must use twisted coefficients induced by $\tau(\omega_1-\omega_0)\in H^1(\mathcal{L}M)$. 
However, as observed before, the corresponding twisted symplectic cohomology may not be well-defined. The crucial point that we use to address this problem is to show that when $\beta$ is a closed 2-form on $M$ with sufficiently small norm on $M^{\inn}$, such that $\beta=d\lambda$ is exact on $M^{\out}$, then twisted symplectic cohomology $SH^*(M,\omega,\theta)_{\tau(\beta)}$ is well-defined and satisfies}
\begin{equation}\label{e:SHiso}
SH^*(M,\omega,\theta)_{\tau(\beta)}\cong SH^*(M,\omega+\beta,\theta+\lambda).
\end{equation}
Moreover, the twisted symplectic cohomology is a unital $\Lambda$-algebra admitting a canonical unital $\Lambda$-algebra homomorphism
\begin{equation}\label{Equation c* map twisted}
c^*: QH^*(M,\omega)_{\beta} \to SH^*(M,\omega,\theta)_{\tau(\beta)}.
\end{equation}
Here, $QH^*(M,\omega)_{\beta}$ is the twisted quantum cohomology of $(M,\omega)$, so holomorphic spheres $u:\C P^1 \to M$ are counted with Novikov weight $t^k$ where $k=\int u^*\omega + \int u^*\beta$, and as a $\Lambda$-vector space $QH^*(M,\omega)_{\beta}=H^*(M;\Lambda)$.

The precise quantitative statement of the above claim is Theorem \ref{t:invariance2}, which is a mouthful, but it implies the following memorable result, which is our main theorem and which proves invariance under ``short deformations''.
\begin{theorem}\label{c:maint}
Let $(M,\omega_s,\theta_s)$ be a deformation of convex manifolds. Then for all sufficiently small $s\geq 0$, there is a unital $\Lambda$-algebra isomorphism
\begin{equation}\label{Eqn nonexact is twisted}
SH^*(M,\omega_s,\theta_s) \cong SH^*(M,\omega_0,\theta_0)_{\tau(\omega_s-\omega_0)},
\end{equation}
which commutes via the $c^*$-maps from \eqref{Equation c* map twisted} with the unital $\Lambda$-algebra isomorphism 
\begin{equation}\label{Eqn QH is twisted}
QH^*(M,\omega_s) \cong QH^*(M,\omega_0)_{\red{\omega_s-\omega_0}}.
\end{equation}
\end{theorem}

We remark that equation \eqref{Eqn QH is twisted} is much simpler than \eqref{Eqn nonexact is twisted}, because both vector spaces equal $H^*(M;\Lambda)$ and the moduli spaces defining the quantum product only depend on an almost complex structure $J$ which can be simultaneously tamed by both $\omega_s$ and $\omega_0$, when $\|\omega_s-\omega_0\|<1$. Thus the twist on the right in \eqref{Eqn QH is twisted} just ensures that $J$-holomorphic spheres are counted with the correct Novikov weight. Explicit examples of twisted quantum cohomology can be described for closed Fano toric manifolds (by Batyrev \cite{Batyrev} and Givental \cite{Givental,Givental2}) in terms of the Landau-Ginzburg superpotential suitably twisted, and similarly in non-compact settings \cite[Section 5]{Ritter6}.

The proof of Theorem \ref{c:maint} uses two key ideas.
The first is to use the map $\varphi_1$ mentioned under
Theorem \ref{t:invariance1} to reduce to the case where one modifies $\omega$ only on a compact subset. This approach bypasses the difficulty of proving a maximum principle for an $s$-dependent $\theta_s$. The second, is a new energy estimate \eqref{Equation estimate of mathcalSu in terms of energy} which allows us to run a continuation argument whilst varying the symplectic form on a compact subset. A new and unexpected feature compared to deforming Liouville manifolds \cite{Ritter2} is that even the formal twisting in Theorem \ref{c:maint} requires such an energy estimate to obtain convergence of counts of moduli spaces.
\begin{remark}\label{Remark Intro deformation literature}
The energy estimate \eqref{Equation estimate of mathcalSu in terms of energy} first appeared in the 2014 PhD thesis \cite{Benedetti}, which the first author used to prove Corollary \ref{t:invariance3} (these results were hitherto not published in a journal, which this paper rectifies). In retrospect (unknown to the author at the time) the idea involved is similar to estimates in Le-Ono \cite[Lemma 5.4]{Le-Ono}: they do not deform $\omega$, but \cite[Theorem 5.3]{Le-Ono}
% They used this to prove that the fixed point Floer cohomology of a non-exact symplectomorphism $\phi$, isotopic to the identity, on a closed symplectic manifold $M$ recovers the Novikov cohomology associated to the one-form given by the Calabi invariant of $\phi$. 
builds a continuation map arising from a deformation of the symplectic vector field, similar to the one we construct in Section \ref{Section Compact deformation invariance}. This energy estimate has since appeared independently in the work of Zhang on spectral invariants for aspherical closed symplectic manifolds \cite[Section 4]{Zhang}. 
\red{The key idea of the estimate is also used at the heart of the recent work of Groman-Merry on the symplectic cohomology of twisted cotangent bundles} \cite[Theorem 6.2]{Groman-Merry} (compare with Theorem \ref{Theorem Energy estimate 2 for Floer traj}).

In addition to how this energy estimate played a role in these papers, we should also illustrate the non-triviality of Theorem \ref{c:maint} by comparing it with invariance theorems in the existing literature. A simpler version of Theorem \ref{c:maint} was proved by the second author in \cite{Ritter2} for compactly supported deformations of Liouville manifolds\red{, where} $(M,\omega_0=d\theta_0)$ is Liouville and $\omega_s=d\theta$ at infinity. Even this simple case at the time required a complicated bifurcation argument and the use of the exact action functional $A_H$ to control energy, which would not generalise to non-exact settings. 
As another example, consider the invariance result \cite[Theorem 1.7]{Viterbo} stated in the seminal paper by Viterbo, where $\omega$ is only allowed to vary amongst aspherical symplectic forms. Upon closer inspection,  filling in the details of the proof of \cite[Theorem 1.7]{Viterbo} does not appear to be straightforward. Indeed notice that the proof does not address the non-trivial issue of obtaining a priori energy estimates needed for compactness of moduli spaces of continuation solutions, and proving an $s$-dependent version of the maximum principle. Bae-Frauenfelder \cite{Bae-Frauenfelder} explain such an invariance proof for closed aspherical symplectic manifolds $M$, provided one assumes in addition that the aspherical symplectic forms $\omega_s$ have primitives with at most linear growth on the universal cover of $M$. Our Corollary \ref{t:invariance3} (a consequence of Theorem \ref{c:maint}) yields a proof of \cite[Theorem 1.7]{Viterbo} that bypasses all of these concerns, and it also immediately implies \cite[Theorem 8]{Ritter2}.
\end{remark}
To prove Theorem \ref{c:maint} for ``long deformations'', so for all $s\in [0,1]$, we require an additional condition. \red{Given a deformation $(M,\omega_s,\theta_s)$ of convex manifolds together with a choice of twisting class $\zeta_0 \in H^1(\mathcal{L}M)$, we call it a \textit{transgression-invariant deformation} if
\begin{equation}\label{Intro Transgression invariant family}
\tau (\omega_s) \in \R_{\geq 0}\cdot(\tau (\omega_0)+\zeta_0) \subset H^1(\mathcal{L}M).
\end{equation}
This condition essentially ensures that the local system of Novikov coefficients is constant in $s$. We can then break up a long deformation into ``short'' pieces and apply Theorem \ref{c:maint} appropriately twisted.} 

\begin{corollary}\label{t:invariance3}
Let $(M,\omega_s,\theta_s)$ be a transgression-invariant deformation of convex manifolds for $\zeta_0 \in H^1(\mathcal{L}M)$. \red{Then the following twisted symplectic cohomologies are well-defined for all $s\in [0,1]$ and they are related by a unital $\Lambda$-algebra isomorphism:}
\begin{equation}\label{Equation transgression invariant iso on SH}
SH^*(M,\omega_s,\theta_s)_{{\zeta}_s} \cong SH^*(M,\omega_0,\theta_0)_{\zeta_0},
\end{equation}
where $\zeta_s=\tau(\omega_0-\omega_s)+\zeta_0 \in H^1(\mathcal{L}M)$.

As a special case, if $(M,\omega_s,\theta_s)$ are convex and $\tau(\omega_s)\in H^1(\mathcal{L}M)$ is constant, then
\[
SH^*(M,\omega_1,\theta_1) \cong SH^*(M,\omega_0,\theta_0).
\]
\end{corollary}
\red{A simple application of Corollary \ref{t:invariance3} is the case of deformations of convex manifolds starting from a Liouville manifold $(M,d\theta_0,\theta_0)$ such that $\tau(\omega_s)=a(s)\cdot \tau(\omega_1)$ for some function $a:[0,1]\to[0,\infty)$ with $a(0)=0$ and $a(1)=1$. Taking $\zeta_0=\tau(\omega_1)$, we get
\[
SH^*(M,\omega_1,\theta_1) \cong SH^*(M,d\theta_0,\theta_0)_{\tau (\omega_1)},
\]
}which so far was known only for compactly supported deformations \cite{Ritter2}.

\red{One can restate the above results for \textit{convex domains}, namely compact symplectic manifolds with contact type boundary, instead of convex manifolds. Such domains have completions which are convex manifolds. We explain these details in Section \ref{Section Summary of results}.}
%%%%%%%%%%%%%%%%%%%%%%%%%%%%%%%%%%%
\subsection{Introduction: Applications}\label{ss:applications}
%%%%%%%%%%%%%%%%%%%%%%%%%%%%%%%%%%%

We decided to focus our applications on twisted cotangent bundles over surfaces, as these already display many interesting features. These are convex manifolds arising from non-compactly supported deformations of Liouville manifolds.
However, the more general setup of Theorem \ref{c:maint} is relevant in many applications which do not arise from deforming Liouville manifolds \red{such as
\begin{enumerate}[(i)]
\item negative complex line bundles, see \cite{Ritter4};
\item the non-compact Fano toric manifolds described in \cite{Ritter6}, in which the deformations of the canonical monotone toric K\"ahler form to generic nearby toric K\"{a}hler forms played a crucial role in mirror symmetry applications \cite[Section 5]{Ritter6};
\item the non-exact convex symplectic manifolds arising as crepant resolutions of isolated quotient singularities, described in the work on the McKay correspondence by McLean-Ritter \cite{McLean-Ritter}. \end{enumerate}
Within the last class of examples} there are the toric ones, which have been analysed, as far as the existence of multiple periodic Reeb orbits is concerned, in recent work of Abreu-Gutt-Kang-Macarini \cite{GuttKang}, where our invariance result played a role.

Magnetic geodesics on a closed manifold $N$ are solutions of a second-order ODE determined by a Riemannian metric $g$ and a closed 2-form $\sigma$ on $N$. The natural lifts of magnetic geodesics to the cotangent bundle $\pi:T^*N\to N$ of $N$ using the metric $g$ are the integral lines of the Hamiltonian flow on $T^*N$ for the symplectic form
\[
\omega_{\sigma}=d\theta+\pi^*\sigma 
\]
and the Hamiltonian $H(q,p)=\tfrac12g_q(p,p)$, where $\theta=p\, dq$ is the canonical 1-form.
We use this Hamiltonian description to study magnetic geodesics which are periodic. 

By now, there is a rich literature on such curves, inspired by work from the early 1980s by Arnol'd \cite{Arnold}, Novikov and Taimanov \cite{Novikov-original-paper,NovikovTaimanov}. For an extensive survey and references on this literature, we refer to Contreras-Macarini-Paternain \cite{CMP}, Ginzburg-G\"{u}rel \cite{GinzburgGurel} and Benedetti's 2014 PhD thesis \cite{Benedetti}. The current paper stems from the latter,
 namely Theorems \ref{t:nonexact} and \ref{t:cmp} on the existence of magnetic geodesics (we rectified and expanded the proofs of the existence of multiple orbits), and Corollary \ref{t:invariance3} on transgression-invariant deformations (which we strengthened to Theorem \ref{t:invariance2}). Since 2014 the field has moved on fast and it is now known in general that a periodic magnetic geodesic exists for almost all energy levels (see Asselle-Benedetti \cite{Asselle-Benedetti} and \red{Groman-Merry \cite{Groman-Merry}}).

Our note originated from trying to relate the existence of periodic magnetic geodesics in the free-homotopy class $\nu\in[S^1,N]$ with the symplectic cohomology $SH^*_\nu(D^*_rN,\omega_{\sigma})$, where $D^*_rN$ is the co-disc bundle of radius $r>0$. This symplectic invariant is well-defined if the co-sphere bundle $S^*_rN$ is of positive contact-type \cite{Ritter2} (see Remark \ref{Remark Intro Convex domain}), and is generated by the periodic magnetic geodesics of energy $\tfrac12r^2$ together with, if $\nu=0$, the cohomology of $N$. Therefore, the existence result would follow if this invariant turned out not to be zero, for $\nu\neq 0$, or not to coincide with the usual cohomology of $N$, for $\nu=0$. That such a line of argument holds for standard geodesics, where $\sigma=0$, for any closed manifold $N$, dates back to Viterbo \cite{Viterbo} and has become a standard tool in symplectic topology \cite{SeidelBiased}. 

When $\sigma$ is exact, it is a classical result that $S^*_rN$ is of positive contact-type for $\tfrac 12r^2>c_0(g,\sigma)$, where $c_0(g,\sigma)$ is the Ma\~n\'e critical value of the universal abelian cover of $N$. In this case, $(D^*_rN,\omega_{\sigma})$ is a Liouville domain (see Example \ref{ex:QL}) and the invariance of symplectic cohomology for this class of manifolds (see e.g.~\cite{SeidelBiased,Seidel-Smith}) yields
\begin{equation}\label{e:Viterbo}
SH^*_\nu(D^*_rN,\omega_{\sigma})\cong SH^*_\nu(T^*N,d\theta)\cong H_{n-*}(\mathcal L_\nu N)_{\tau(w_2(N))},
\end{equation}
where on the right one obtains the singular homology of the space of free loops in the class $\nu$ with coefficients twisted by the transgression of the second Stiefel-Whitney class. The latter is the Viterbo isomorphism \cite{Viterbo} (see Abouzaid \cite{Abouzaid-survey} for a survey). The twist can be ignored if $N$ is spin or if one works with coefficients in characteristic two.

We are therefore interested in the situation in which $(D^*_rN,\omega_{\sigma})$ is not a Liouville domain. However, if $\sigma$ is not exact and either $\dim N\geq 3$ or $N=\mathbb T^2$, then none of the $S^*_rN$ can be of contact-type since $\pi^*\sigma$ is not exact on $S^*_rN$. 

Our paper will only consider the case of surfaces $N$, as the purpose of the application is only to illustrate the deformation theorem for convex manifolds. These lead to results about the existence of closed magnetic geodesics in surfaces which are by now classical thanks to the work of Cristofaro-Gardiner and Hutchings \cite{CG-Hutchings}, which used embedded contact homology to prove the existence of two periodic magnetic geodesics, for $\dim N=2$ and $S^*_rN$ of contact-type, without non-degeneracy assumptions. 

\begin{theorem}\label{t:nonexact}
Let $N\neq\mathbb T^2$ be a closed orientable surface with a Riemannian metric $g$ and a non-exact 2-form $\sigma$. If $r>0$ is large, or for $N=S^2$ if $r>0$ is small and $\sigma$ is nowhere vanishing, then $S^*_rN$ is of positive contact-type. Under those assumptions,
\begin{equation}\label{Equation vanishing SH}
SH^*(D^*_rS^2,\omega_{\sigma})\cong SH^*(T^*S^2,d\theta)_{\tau(\pi^*\sigma)}=0,
\end{equation}
and there is a prime periodic magnetic geodesic of energy $\tfrac12r^2$. Unless one of the iterates of that orbit is transversally degenerate, there are at least two such geodesics. If $N$ has genus\;$\geq 2$,
\[
SH^*_\nu(D^*_rN,\omega_{\sigma})\cong SH^*_\nu(T^*N,d\theta)_{\tau(\pi^*\sigma)}=\left\{\begin{aligned}
H_{2-*}(N)&\quad \text{if }\nu=0,\\
H_{2-*}(S^1)&\quad \text{if }\nu\neq 0,
\end{aligned} \right.
\]
and there is at least one periodic magnetic geodesic in each free homotopy class $\nu\neq 0$.
\end{theorem}
\begin{remark}
A neighbourhood of the zero-section in $(T^*S^2,\omega_{\sigma})$ and a neighbourhood of the zero section in the line bundle $\mathcal O(-2)\to\C\P^1$ are symplectomorphic by classical results (see, for instance, \cite[Theorem 3.4.10]{McDuff-Salamon-IntroToSympl} or \cite[Section 2.4A]{EliPol}). In Appendix \ref{Subsection From the magnetic TS2 to the HK TCP1}, we construct an explicit global symplectomorphism between $(T^*S^2,\omega_{\sigma})$ and $\mathcal O(-2)\to\C\P^1$, using the round metric $g$ and the area form $\sigma$. Thus, the vanishing \eqref{Equation vanishing SH} is consistent with the fact that $SH^*(\mathcal{O}_{\C\P^1}(-2))=0$ by \cite{Ritter2,Ritter4}.
\end{remark}
Now suppose that $N=\T^2$, $\sigma$ is exact, $\tfrac12r^2\leq c_0(g,\sigma)$. Contreras, Macarini and Paternain showed in \cite{CMP} that $(D^*_r\mathbb T^2,\omega_{\sigma})$ cannot be a Liouville domain; however they also list a simple class of examples \red{(which we refer to as QL-magnetic-tori)} for which $S^*_r\mathbb T^2$ is of positive contact-type for $\tfrac12r^2$ close to $c_0(g,\sigma)$.
\begin{theorem}\label{t:cmp}
Let $(g,\sigma)$ be a \red{QL-magnetic-torus} as in Section \ref{Subsection An example of a Quasi-Liouville magnetic TT2}, where $\sigma$ is an exact 2-form on $\mathbb T^2$. Then there exists an $\epsilon>0$ such that $S^*_r\mathbb T^2$ is of positive contact-type for all $\tfrac12r^2\in(c_0(g,\sigma)-\epsilon,c_0(g,\sigma)]$ and
\[
SH^*_\nu(D^*_r\mathbb T^2,\omega_{\sigma})\cong SH^*_\nu(T^*\mathbb T^2,d\theta)\cong H_{2-*}(\mathbb T^2) \qquad \textrm{ for all }\nu\in [S^1,\mathbb T^2].
\]
In particular, there is at least one periodic magnetic geodesic of energy $\tfrac12r^2$ in every non-trivial free homotopy class (and two in the non-degenerate case). If the contact form on $S^*_r\mathbb T^2$ is non-degenerate, there are infinitely many contractible periodic magnetic geodesics with energy $\tfrac12r^2$.
\end{theorem}

To prove the passage from the existence of one to two (respectively infinitely many) closed orbits in Theorem \ref{t:nonexact} (respectively \ref{t:cmp}), we use a new  general scheme which is applicable in theory to many other situations. We exploit the properties of the BV-operator $\Delta: SH^*(M)\to SH^{*-1}(M)$ on symplectic cohomology (see Section \ref{Subsection BV}). To our knowledge, this approach has not appeared elsewhere in the literature.
The method more familiar to experts is to prove such results using the $S^1$-equivariant symplectic cohomology (we also sketch the proof using that method), for example see Kang \cite{Kang} and more recently \cite{GuttKang}.
However using $\Delta$ is more economical as the $S^1$-equivariant differential involves infinitely many correction terms to the ordinary differential, of which $\Delta$ is the first correction term. The BV-operator method is based on the interplay at the chain level arising from $
\partial \Delta+\Delta\partial=0$, between the degree $+1$ Floer differenial $\partial$ and the degree $-1$ BV-operator $\Delta$.
Together with some non-trivial filtration properties by McLean-Ritter \cite[Appendix D]{McLean-Ritter} and work of Zhao \cite[Equation (6.1)]{Zhao} (see Section \ref{Subsection period}), we then deduce the multiplicity results for magnetic geodesics stated above.
% constraints on cardinalities of cycles and boundaries, which enable us to improve existence results beyond just one periodic Hamiltonian orbit (and its iterates).

\begin{remark}
We mention four other Floer theories, which have been defined on twisted cotangent bundles and which are also generated by certain periodic magnetic geodesics.
The first is the Rabinowitz Floer Homology constructed by Merry \cite{Merry} when $\sigma$ has a bounded primitive on the universal cover of $N$. This ``RFH'' involves a combination of the homology and cohomology of the free loop space. Moreover, Bae and Frauenfelder \cite{Bae-Frauenfelder} established a continuation isomorphism between RFH of the twisted cotangent bundle and RFH of the ordinary cotangent bundle.
The second, developed by Frauenfelder, Merry and Paternain in \cite{Frauenfelder-Merry-Paternain,Frauenfelder-Merry-Paternain2} uses quadratic Hamiltonians satisfying the Abbondandolo-Schwarz growth condition and it is defined for forms $\sigma$ having at most linear growth on the universal cover of $N$. In this case, periodic orbits of a given period, instead of a given energy, are detected.
The third, due to Gong \cite{Gong}, assumes $\sigma$ admits a primitive of at most linear growth on the universal cover, but uses compactly supported Hamiltonians which are large enough over the zero section. Periodic magnetic geodesics for almost every level in some energy range are detected.

In all three theories, the assumptions imply that $\omega_{\sigma}$ is aspherical, meaning $\omega_{\sigma}$ integrates to zero on $\pi_2(T^*N)$. This in particular implies global exactness of $\omega_{\sigma}$ when $N$ is simply connected, so it does not apply to $T^*S^2$. More specifically, their assumptions ensure that the Floer action functional is single-valued and that no twisted coefficients appear, whereas our setup endeavours to overcome such restrictive conditions. 

\red{After the appearance of the third arXiv version of the present paper, a fourth theory due to Groman and Merry \cite{Groman-Merry} appeared. The theory works for higher dimensional $N$ and non-exact $\sigma$, in which case the twisted $T^*N$ is not convex but still geometrically bounded at infinity. By a careful choice of class of Hamiltonians, they proved that $SH_\nu^*(T*N,\omega_\sigma)$ is well-defined and that the first isomorphism in
\begin{equation}\label{Equation SH TN is H LN intro}
SH^*_\nu(T^*N,\omega_{\sigma})\cong SH^*_\nu(T^*N,d\theta)_{\tau(\pi^*\sigma)}\cong H_{n-*}(\mathcal{L}_\nu N)_{\tau(\sigma)\otimes \tau(w_2(N))}
\end{equation}
holds. Here the middle term and the second isomorphism were constructed by the second author \cite{Ritter1,Ritter3}, who also showed that
\[H_*(\mathcal{L} N)_{\tau(\beta)}=0\]
%
%
%That vanishing also holds for $\tau(\eta)\neq 0\in H^1(\mathcal{L}_0 N)$ when $N$ is just of finite type ($\pi_m(N)$ is finitely generated for each $m\geq 2$) \cite{Ritter3}.
if  $\tau(\beta)\neq 0\in H^1(\mathcal{L}_0N)$ and $\pi_m(N)$ is finitely generated for each $m\geq 2$ (e.g.\,if $N$ is simply connected and $\beta\neq 0\in H^2(N)$).
More refined conditions for vanishing were proved in \cite{Albers-Frauenfelder-Oancea}. 
The vanishing result then implies the existence of periodic magnetic geodesics in this setting.}
\end{remark}

\subsection{Structure of the paper}
In Section \ref{Section Background on symplectic manifolds exact at infinity} we prove foundational results about convex manifolds, and Theorem \ref{t:invariance1} in Subsection \ref{Subsection Constant relative class implies compactly supported deformation}. In Section \ref{Section SH} we construct (twisted) symplectic cohomology for convex manifolds, and prove Theorem \ref{t:invariance2}.(1) in Subsection \ref{Subsection Existence of twisted symplectic cohomology}. In Section \ref{Section Compact deformation invariance} we prove Theorem \ref{t:invariance2}.(2) and Corollary \ref{t:invariance3}. Section \ref{Section Application: Twisted cotangent bundles of surfaces} deals with twisted cotangent bundles and proves Theorems \ref{t:nonexact} and \ref{t:cmp}. Appendix \ref{Subsection From the magnetic TS2 to the HK TCP1} constructs an explicit symplectomorphism between the twisted $T^*S^2$ and $\mathcal O_{\C\P^1}(-2)$.
Appendix \ref{Appendix1} recalls iteration formulae for Conley-Zehnder indices in dim$\,=3$ used in Section \ref{Section Application: Twisted cotangent bundles of surfaces}.
\subsection{Acknowledgements}
G.B.~ wishes to thank his PhD advisor Gabriel Paternain for suggesting to compute the symplectic cohomology of twisted cotangent bundles of surfaces and for his valuable suggestions and insights on this problem. We are very grateful to the referee for the careful reading of the manuscript and for the valuable comments, which helped us improve the paper considerably.
%%%%%%%%%%%%%%%%%%%%%%%%%%%%%%%%%%%
%%%%%%%%%%%%%%%%%%%%%%%%%%%%%%%%%%%
\section{Convex manifolds and their deformations: precise definitions}
\label{Section Summary of results}
%%%%%%%%%%%
One often constructs symplectic cohomology as an invariant associated to a closed symplectic manifold $D$ with contact-type boundary $\Sigma=\partial D$ (see \cite{OanceaEnsaios,SeidelBiased} and Remark \ref{Remark Intro Convex domain}). One then builds a non-compact symplectic manifold $M$ by attaching a conical end $\Sigma \times [0,\infty)$. For example $D=D^*N$ completes to $M=T^*N$ with $\Sigma=S^*N$. When $\omega=d\theta$ is globally exact one blurs the distinction between $SH^*(D)$ and $SH^*(M)$ because a surprisingly strong invariance result applies \cite{Abouzaid-Seidel-Altering}. This relies on the existence of a global compressing Liouville flow and a single-valued action functional, which are not available in the non-exact setting. Our paper could be entirely phrased in terms of closed manifolds $D$ \cite{Benedetti} but we decided to instead start with a given non-compact symplectic manifold $(M,\omega)$, as this is increasingly the practical setup one encounters (in the exact setup, this point of view is discussed in Seidel-Smith \cite{Seidel-Smith}). 

\begin{definition}\label{d:convex}
\red{A \textit{convex manifold} is a triple $(M,\omega,\theta)$ where $(M,\omega)$ is an open symplectic manifold admitting some exhausting (namely proper and bounded below) function $h:M\to \R$ such that $\theta$ is a $1$-form defined on $M^{\out}:=\{h\geq 0\}$ satisfying $\omega=d\theta$ and $\theta(X_h)>0$.
We call $h$ a \textit{Liouville function} for $(M,\omega,\theta)$.
On $M^{\out}$ there exists a Liouville vector field $Z$ defined by $\theta=\iota_Z \omega$. Thus $\theta(X_h)=dh(Z)$. Therefore, $\theta(X_h)>0$ if and only if $Z$ points out of the domains $M_{\leq a}:=\{h\leq a\}$ for all $a\geq0$. We set
\[
\Sigma:=\{h=0\},\qquad M^{\inn}:=\{h\leq 0\}.
\]
We call $(M,\omega,\theta)$ \textit{complete} if the flow of $Z$ is positively integrable.}\footnote{
By flowing via $\frac{1}{dh(Z)}\cdot Z$ starting from the regular level set $\Sigma=\partial M^{\inn}=h^{-1}(0)$ we obtain a foliation $M^{\out}=\sqcup_{x\in [0,\infty)} h^{-1}(x) \cong \Sigma \times [0,\infty)$ 
by diffeomorphic regular level sets of positive contact type for $\theta|_{h^{-1}(x)}$, 
using $x=h$ as the second coordinate. 
%Although the regular level sets of $h$ are of positive contact type for the contact form $\theta|_{h^{-1}(x)}$, that foliation is not usually symplectically well-behaved like \ref{Equation Intro j} as the flow of $Z$ need not preserve the level sets. 
Using this, one finds that $Z$ is integrable for all positive time $\Leftrightarrow$ for some choice of $h$ we have $\int_{0}^\infty \frac{1}{\ell(x)}\, dx=\infty$ where $\ell(x):=\max_{p\in h^{-1}(x)} dh(Z)|_p$.
}
\end{definition}
\begin{definition}
\red{Given convex $(M_0,\omega_0,\theta_0)$, $(M_1,\omega_1,\theta_1)$, an \textit{isomorphism} is a symplectomorphism $\varphi:(M_0,\omega_0) \to (M_1,\omega_1)$ which at infinity satisfies $\varphi^*\theta_1=\theta_0$.}
\end{definition}
\begin{definition}
\red{A deformation of a convex manifold is a family $(M,\omega,\theta_s)$ with $s\in[0,1]$ together with a family of exhausting functions $h_s:M\to\R$ such that $H:[0,1]\times M\to\R$ given by $H(s,\cdot)=h_s$ is also exhausting.
For $0\leq a\leq b$, we will use the notation}
\[
\red{M_s^{\geq a}:=\{h_s\geq a\},\qquad M_s^{[a,b]}:=\{a\leq h_s\leq b\}.}
\]
\end{definition}
\red{For convex $(M,\omega,\theta)$, the form $\alpha = \theta|_{\Sigma}$ is a positive contact form on $\Sigma = h^{-1}(0)$, meaning $\alpha \wedge (d\alpha)^{\dim_{\C} M-1}>0$ with respect to the boundary orientation for $\partial M^{\inn}$. Moreover, there is a smooth function $\sigma:\Sigma\to (0,\infty]$ and a symplectomorphism $j$, called \textit{conical parametrisation},
\begin{equation}\label{Equation Intro j}
j: \big(\left\{(y,r)\in \Sigma\times [0,\infty): r< \sigma(y) \right\},\ d(r\alpha)\big)\; \longrightarrow \; (M^{\out},\omega)
\end{equation}
defined by the time $r$ flow of $Z$, so
\[
j(y,r) = \mathrm{Flow}^{Z}_{r}(y).
\]
In particular $j^*\theta=e^r\alpha$. Note that $(M,\omega,\theta)$ is complete if and only if we can pick $\sigma\equiv\infty$ in \eqref{Equation Intro j}.} 

\red{Any convex $(M,\omega,\theta)$ can always be embedded into the completion $(\hat M,\hat\omega,\hat\theta)$ of $M^{\inn}$ obtained by gluing $\Sigma\times[0,\infty)$ and $M^{\inn}$ via the map $j$ above. Since we will complete $M$ with respect to different pairs $(\omega,\theta)$, we will also sometimes use the more precise notation
\[
(M,\omega,\theta)^\wedge:=(\hat M,\hat\omega,\hat\theta).
\]
If $(M,\omega_s,\theta_s)$ is a deformation, we write $(\hat M_s,\hat\omega_s,\hat\theta_s)$ to indicate the corresponding completions. 
}
\begin{remark}A choice of Liouville function $h$ on $M$ determines a positive contact hypersurface
\begin{equation}\label{introsigma}
\Sigma=\{h=0\} = \partial M^{\out} \subset M, \textrm{ with contact form }\alpha:=\theta|_{\Sigma},
\end{equation}
%
%
% next sentence is meaningless, since haven't extended h to global function
% and h is not proper unless \sigma is constant
%
%Note that $h=r$ satisfies the Definition.
%
%
but $\Sigma$, $M^{\out}$, $h$ are not fixed in Definition \ref{d:convex}.
\red{For instance, in view of \eqref{Equation Intro j}, the function $r$ is one of many examples of a Liouville function on $\hat M$.
Changing the }choice of $h$ corresponds to modifying $\Sigma = \Sigma \times \{0\} \subset \Sigma \times \R$ to a ``graph'' $$\{(y,f(y)): y\in \Sigma\}$$ of a smooth function $f: \Sigma \to \R$ (see Section \ref{Subsection varying alpha in a family}). 
The condition $\varphi^*\theta_1=\theta_0$ in Definition \ref{d:convex} implies $\varphi_*Z_1=Z_0$ at infinity, therefore for large $r$ in the coordinates \eqref{Equation Intro j} we have
\begin{equation}\label{e:normalform}
\varphi(y,r) = (\psi(y),r-f(y))
\end{equation}
% The condition on Z says we are a graph
% then just consider eqn phi*(e^{r_1}\alpha_1) = e^{r_0}\alpha_0.
%
%
where $\psi: \Sigma_0 \to \Sigma_1$ is a contactomorphism, namely a diffeomorphism satisfying $\psi^*\alpha_1=e^{f}\alpha_0$ for a smooth function $f: \Sigma_0 \to \R$.
Thus the contactomorphism class of $\Sigma$ and the contact structure $\xi=\ker \alpha\subset T\Sigma$ are invariants under isomorphism, but the positive contact form $\alpha$ is not. We show in Section \ref{Subsection varying alpha in a family} how $\alpha$ can be varied arbitrarily subject to those invariants (Remark \ref{Remark can vary alpha arbitrarily}).
\end{remark}

\begin{example}\label{ex:QL}
As we do not require completeness, if $\theta(X_h)>0$ only holds near $\Sigma=h^{-1}(0)$, we still obtain a convex submanifold: $(h^{-1}(-\infty,\epsilon),\omega,\theta)$ for small $\epsilon>0$.

 Recall $M$ is \textit{Liouville} if in addition to \eqref{Equation Intro convex condition}, $\theta$ extends to a global primitive of $\omega$. However, there are convex $(M,\omega,\theta)$ for which $\omega$ is globally exact, but the given $\theta$ does not extend to a global primitive; we call these \textit{Quasi-Liouville}. 
By Lemma \ref{Lemma Liouville condition}, the obstruction for a convex manifold to be Liouville is the \textit{relative class} \eqref{Equation intro relative class}.
If $[\omega]=0\in H^2(M)$, this obstruction becomes $[j^*\lambda-e^r\alpha]\in H^1(M^{\out})$ where $\lambda$ is the given global primitive of $\omega$. The $T^*\T^2$ of Theorem \ref{t:cmp} yield Quasi-Liouville examples.
%
% In quasi-liouville case have omega=d\lambda, but don't use theta' on Mout,
% you use theta.
% So compare [d\lambda,\lambda]=0 versus [d\lambda,Ralpha].
% So difference is [0,\lambda-Ralpha]
% This is zero relative class iff \lambda-Ralpha is exact on Mout.
% Indeef if exact, then can extend Ralpha to a global primitive of omega, so Liouville.

The relative class \eqref{Equation intro relative class} is preserved under isomorphisms, meaning $\varphi^*[\omega_1,\theta_1]=[\omega_0,\theta_0]$ viewed as classes in\footnote{It will hold in the relative de Rham cohomology $H^2(M_0,M_0^{\out})$ if we choose $M_0^{\out}$ and $M_1^{\out}=\varphi(M_0^{\out})$ so that $\varphi^*\theta_1=\theta_0$ holds on $M_0^{\out}$, see Lemma \ref{Lemma varphit preserves relative class}.} $H^2_c(M)$. This is an obstruction to the existence of an isomorphism (Example \ref{Example Harris}) which did not appear for Liouville manifolds as $[\omega,\theta]=0$ in that case.
\end{example}

The symplectic cohomology of a convex manifold is defined by the direct limit \eqref{Equation Intro SH is direct lim} over Floer continuation maps of the Floer cohomologies computed for Hamiltonians $H:M\to \R$ which at infinity are linear in $R=e^r$ of larger and larger slopes. This class of Hamiltonians depends on $j$, $h$ and $\Sigma$ in \eqref{Equation Intro j} and \eqref{introsigma}, as they determine the radial coordinate $R$, and these choices are not unique given $(M,\omega,\theta)$. A different choice corresponds to changing $R$ to $e^{f(y)}R$ for a smooth function $f: \Sigma \to \R$. The proof of Theorem \ref{Theorem invariance under iso of SH} constructs an isomorphism between the two symplectic cohomologies computed for Hamiltonians that are linear for the respective radial coordinate.

We will always tacitly assume that 
$(M,\omega)$ is \textit{weakly monotone}\footnote{At least one of the following holds:
(i) $c_1$ vanishes on $\pi_2(M)$, (ii) $\omega$ vanishes on $\pi_2(M)$, (iii) $(M,\omega)$ is monotone (so $c_1|_{\pi_2(M)}=\lambda \omega|_{\pi_2(M)}$  for some $\lambda>0$), or (iv) the minimal Chern number $|N|\geq \dim_{\C} M - 2$, where $\langle c_1(TM), \pi_2(M)\rangle = N \Z$. This is equivalent to requiring that for each $A\in H_2(M,\Z)$, if
$3-\dim_{\C}M \leq c_1(A) < 0$ then $\omega(A)\leq 0.$
} as that ensures the Floer cohomology groups are well-defined \cite{Hofer-Salamon} without appealing to advanced machinery, such as Kuranishi structures or Polyfolds. Magnetic $(T^*N,\omega_{\sigma})$ are always weakly monotone as their first Chern class vanishes. Following \cite{Ritter4,Ritter6} we work over coefficients in the \textit{Novikov field} $\Lambda$ involving `series' in a formal variable $t$ (see Section \ref{Subsection Novikov field, Action $1$-form, Energy}).
%, which depends only on the transgression  $\tau(\omega)\in H^1(\mathcal LM)$.

\begin{remark}\label{Remark Intro Convex domain}
A \textit{convex domain} $(D,\omega,\alpha)$ is a compact symplectic manifold $(D,\omega)$ such that $\alpha$ is a positive contact form on the boundary $\Sigma=\partial D$ with $d\alpha=\omega|_{T\Sigma}$ (see Lemma \ref{Lemma Trick to get convexity}).  Given $(D,\omega)$, the possible such choices of $\alpha$ determine a convex set. 
%
% so consider alpha_s = s alpha_0 + (1-s) alpha_1.
% easiest if, after local extensions to theta_0,theta_1,
% you consider the Liouville v.f. Z_0,Z_1
% then s Z_0 + (1-s) Z_1 is Liouville and outward pointing
% so contracted with omega and restricted gives contact form, so alpha_s contact
%
% EASIER
% dalpha_0 = dalpha_1
% so can just do explicit computation to check it is contact.
%
Convex domains arise as ``sublevel sets'' of convex manifolds:
\[
D=M^{\inn}\cup \big\{(y,r)\in \Sigma \times [0,\infty): y\in \Sigma,\ r\leq f(y)\big\}\subset \hat{M}
\]
for any smooth $f:\Sigma \to [0,\infty)$.
By convention, $SH^*(D):=SH^*(M)$ for any completion $M=\hat{D}$ of $D$ (this is well-defined by Theorem \ref{Theorem invariance under iso of SH}, Remark \ref{Remark choice of primitive for convex domain}).
%(by Theorem \ref{Theorem invariance under iso of SH}, $SH^*(M)$ is invariant under the choice of primitive for $\omega$ on a collar of $\Sigma$ which induces the completion).
% where $M=\hat D$ is a completion of $M^{\inn}=D$.
% there is a choice of primitive theta near the collar, not a big deal
We abusively speak of isomorphisms of such $D$ when we mean isomorphisms of their completions. Also, $SH^*(D)$ is invariant under deformations of the contact form by Corollary \ref{t:invariance3}. 
\end{remark}

We can now state the quantitative version of Theorem \ref{c:maint}.

\begin{theorem}\label{t:invariance2} 
Let $(M,\omega,\theta)$ be convex. For $a>0$, consider the convex manifold $M^{<a}:=\{h<a\}\subset M$, and let $M^{(0,a)}:=\{0<h<a\}$. There are $\epsilon,\epsilon'>0$ depending on $a$, such that for all closed two-forms $\beta$ on $M^{<a}$ exact on $M^{(0,a)}$ with $\Vert\beta\Vert_{C^1(M^{<a})}<\epsilon$,

\begin{enumerate}
\item The twisted symplectic cohomology $SH^*(M,\omega,\theta)_{\tau(\beta)}$ can be constructed using a suitable cofinal subfamily of radial Hamiltonians (see Section \ref{Subsection Existence of twisted symplectic cohomology}). It is a unital $\Lambda$-algebra admitting a canonical unital $\Lambda$-algebra homomorphism
\[
c^*: QH^*(M,\omega)_{\beta} \to SH^*(M,\omega,\theta)_{\tau(\beta)}.
\]
\item For any representative $[\mu,\lambda]\in H^2(M^{<a},M^{(0,a)})$ of the class $[\beta]\in H^2(M^{<a})$, with $\|(\mu,\lambda)\|_{C^1(M^{<a},M^{(0,a)})}<\epsilon'$, the triple
$(M^{<a},\omega+\mu,\theta+\lambda)$
is convex and admits a unital $\Lambda$-algebra isomorphism
\begin{equation}\label{Eqn nonexact is twisted precise}
SH^*(M^{<a},\omega+\mu,\theta+\lambda)\cong SH^*(M,\omega,\theta)_{\tau(\beta)}
\end{equation}
commuting with the canonical $c^*$ maps from $QH^*(M,\omega+\mu)\cong QH^*(M,\omega)_{\beta}$. Thus, the group on the right in \eqref{Eqn nonexact is twisted precise} is independent of the choice of the cofinal family of Hamiltonians.
\end{enumerate}
\end{theorem}

% The difficulty of proving invariance outside of the aspherical setting is apparent upon reading \cite{Ritter2}, which starts with a Liouville manifold $(M,d\theta,\theta)$ and shows in \cite[Theorem 8]{Ritter2} that \eqref{Eqn nonexact is twisted} holds for $\omega_s=d\theta+s\beta$ and all $s\in [0,1]$, where $\beta$ is a closed compactly supported $2$-form with $\Vert\beta\Vert<1$. One of the key ingredients \cite[Section 6.5]{Ritter2} was the energy estimate $E(u)\leq k (A_H(x)-A_H(y))$, where $k>1$, for Floer solutions $u$ converging to Hamiltonian orbits $x,y$ obtained for the non-exact form $d\theta+\beta$, where one exploits the action functional $A_H$ determined by the exact form $d\theta$. By contrast, our Theorem \ref{t:invariance2} does not presume that $M$ admits an exact symplectic form, and in fact we do not have at our disposal a well-defined action functional in that general setting.
% Despite the rather general setup, Theorem \ref{t:invariance2} (more precisely, Theorem \ref{c:maint}) immediately implies \cite[Theorem 8]{Ritter2} as a simple application.

As mentioned in the Introduction, to obtain \eqref{Eqn nonexact is twisted} for all $s$, one breaks down the ``long'' deformation into ``short'' pieces and uses a twisted version of Theorem \ref{c:maint}. Unfortunately convergence issues in the Novikov field prevent such twistings in general, so we introduce a good notion of local systems on $\mathcal{L}M$ which work.

\begin{definition}\label{Definition Intro transgression compatible twist}
	For convex $(M,\omega,\theta)$, call ${\zeta} \in H^1(\mathcal{L}M)$ \textit{trangression-compatible} if \[{\zeta} \in \R_{>-1}\cdot [\tau(\omega)] \;\;\textrm{ or }\;\; [\tau(\omega)]=0 \in H^1(\mathcal{L}M)\] Equivalently, choosing representative $1$-forms $\tau_{\omega}$ and $\eta$ in the classes $\tau(\omega)$ and ${\zeta}$, for some $c\geq 0$ there is a function $\mathcal{K}: \mathcal{L}M\to \R$, with
	\begin{equation}\label{Equation Intro transgression compatible}
	\tau_\omega= c(\tau_\omega+{\eta}) + d\mathcal{K}.
	\end{equation}
	A family $(M,\omega_s,\theta_s;{\zeta}_s)$ is \textit{transgression-compatible} if $s\mapsto\tau(\omega_s)+\zeta_s \in H^1(\mathcal{L}M)$ is constant and for each $s\in [0,1]$, $\zeta_s$ is transgression-compatible with $\omega_s$. This implies that, after twisting by $\zeta_s$, the system of Novikov coefficients is constant in $s$.
\end{definition}

We prove that symplectic cohomology is always defined for transgression-compatible twists, and we prove that Theorem \ref{c:maint} implies Theorem \ref{t:invariance3}. A simple example is if $\tau(\omega)=0\in H^1(\mathcal{L}M)$: in that case any twisting $\zeta\in H^1(\mathcal{L}M)$ is allowed.

Finally, we remark that deformations of convex domains $(D,\omega_s)$ reduce to the problem of deformations of convex manifolds. Observe that the completions of $(D,\omega_s)$, as $s$ varies, are all diffeomorphic to the completion $M=D\cup (\partial D \times [0,\infty))$ of $(D,\omega_0)$. Pulling back the data via this identification we obtain a corresponding deformation $(M,\omega_s,\theta_s)$ of complete convex manifolds. Thus Theorems \ref{t:invariance1} and \ref{t:invariance2} apply.

%%%%%%%%%%%%%%%%%%%%%%%%%%%%%%%%%%%
%%%%%%%%%%%%%%%%%%%%%%%%%%%%%%%%%%%
\section{Convex manifolds and their deformations: the proofs}
\label{Section The role of the relative class}
\label{Section Background on symplectic manifolds exact at infinity}
%%%%%%%%%%%%%%%%%%%%%%%%%%%%%%%%%%%%%%%%%%%%%%%%%%%%%%%%%%%%%%%%%%%%%%%%%
%%%%%%%%%%%%%%%%%%%%%%%%%%%%%%%%%%%%%%%%%%%%%%%%%%%%%%%%%%%%%%%%%%%%%%%%%
%%%%%%%%%%%%%%%%%%%%%%%%%%%%%%%%%%%%%%%%%%%%%%%%%%%%%%%%%%%%%%%%%%%%%%%%%
\subsection{Relative cohomology}
\label{Subsection Remarks about relative cohomology}
%%%%%%%%%%%%%%%%%%%%%%%%%%%%%%%%%%%%%%%%%%%%%%%%%%%%%%%%%%%%%%%%%%%%%%%%%
%%%%%%%%%%%%%%%%%%%%%%%%%%%%%%%%%%%%%%%%%%%%%%%%%%%%%%%%%%%%%%%%%%%%%%%%%
Let $(M,\omega,\theta)$ be convex (Definition \ref{d:convex}). Following Bott-Tu \cite[Section I.7,p.78]{Bott-Tu}, we define the relative de Rham cohomology $H^*(M,M^{\out})$ via the mapping cone $C^*(M,M^{\out})=\Omega^*(M)\oplus \Omega^{*-1}(M^{\out})$ with differential
\[
D(x,y):=(dx,j^*x-dy),
\]
where $j^*x$ is the pull-back of $x$ to $M^{\out}$ via \eqref{Equation Intro j}.
So $[\omega,\theta] \in H^2(M,M^{\out})$.
\begin{lemma}\label{Lemma Liouville condition}
A convex manifold $(M,\omega,\theta)$ is Liouville if and only if the relative class $[\omega,\theta] \in H^2(M,M^{\out})$ vanishes. 
\end{lemma}
\begin{proof}
For Liouville $(M,d\theta,\theta)$, $[d\theta,\theta]=[D(\theta,0)] = 0 \in H^2(M,M^{\out})$.
Conversely, assume $[\omega,\theta] = 0$. Then there is a $1$-form $\lambda\in\Omega^1(M)$ and a function $f\in\Omega^0(M^\out)$ such that $\omega=d\lambda$ on $M$, and $\theta-\lambda=df$ on $M^\out$. After extending $f$ to a smooth function on $M$, we obtain an extension $\theta:=\lambda+d f\in \Omega^1(M)$ with $\omega=d\theta$ on $M$.
\end{proof}
Relative cohomology is isomorphic to compactly supported cohomology $H^*_c(M)$. Indeed, we have maps $H^*_c(M\setminus M^\out)\to H^*(M,M^\out)\to H^*_c(M)$ given by
\begin{equation}\label{e:isocoh}
[x]\mapsto [x,0],\qquad [x,y]\mapsto [x-d(\rho y)],
\end{equation}
where $\rho:M\to [0,1]$ is any function with $\rho=1$ at infinity and $\rho=0$ near $M^{\inn}$. Up to the isomorphism $H^*_c(M\setminus M^\out)\cong H^*_c(M)$ induced by the inclusion $M\setminus M^\out\subset M$, the maps in \eqref{e:isocoh} are isomorphisms that are inverse to each another.

\begin{lemma}\label{Lemma varphit preserves relative class}
If $\varphi_t$ is an isotopy of $M$, $\varphi_0=\mathrm{id}$, 
%
% normally would say $\varphi_t(M^{\out})=M^{\out}$ from construction
% but since have iso with H_c*(M) don't need to worry about that anymore. 
% 
then $[\varphi_t^*x]=[x]$ in $H^*_c(M)$ for any $[x]\in H^*_c(M)$.
If in addition $\varphi_t(M^{\out})=M^{\out}$, 
%
% without this, don't even get functoriality
% and not just because ill-defined.
% Have situation later where (omega_0+beta_s,theta_0) --> (omega_s,theta_s)
% latter is constant in H^2 rel, so pull-back, would imply that
% (beta_s,0) is constant in H^2 rel. But in general it is not 
% even in H^2 rel because not closed, as get j^*beta_s on collar, possibly nonzero.
% Would need to use phi_t^{-1}(M^out) to get functoriality.
%
then
$[\varphi_t^*x,\varphi_t^*y]=[x,y]$ in $H^*(M,M^{\out})$ for any $[x,y]\in H^*(M,M^{\out})$. 
\ifdraft 
\else 
%\begin{proof}
%This follows by the relative analogue of \cite[Cor.4.1.2,p.35]{Bott-Tu}.
%\end{proof}
\fi
\end{lemma}
\begin{proof}
This follows from the functorial properties of $H^*_c$ \cite[p.~26]{Bott-Tu} (with respect to proper maps). The second claim is the relative analogue of \cite[Cor.~4.1.2, p.~35]{Bott-Tu}. 
\end{proof}
\ifdraft
\begin{proof} This is the relative analogue of \cite[Cor.4.1.2,p.35]{Bott-Tu}. It holds because the projection $\pi^*: H^*(M,M^{\out})\cong H^*(M\times \R,M^{\out}\times \R)$ is inverse to any of the sections $s_t^*$. In particular, $\mathrm{id}-\pi^*s_t^*=kD+Dk$ at the chain level, by taking $k=(-1)^{*-1}K$ for the $K$ described by Bott-Tu \cite[p.35]{Bott-Tu} and using the general convention that an operator $k$ on $\Omega^*(M),\Omega^*(M^{\out})$ also acts on $C^*(M,M^{\out})$ by $k(x,y)=(kx,(-1)^{\mathrm{deg}(k)}ky)$.
\end{proof}
\else 
\fi
\red{Given a convex manifold $(M,\omega,\theta)$, two different choices of exhausting function give rise to different conical ends, thus two choices of decomposition
\[
M=M^{\inn}_1\cup M^{\out}_1=M^{\inn}_2\cup M^{\out}_2.
\]
However, there is an isotopy $\varphi_t:M\to M$ supported near the conical end such that $\varphi_t(M_1^{\out})=M_2^{\out}$ (by simply isotopying the radial coordinate, i.e. moving along flowlines of $Z$). The class $[\omega,\theta]$ viewed in $H^2_c(M)$ via \eqref{e:isocoh} therefore does not depend on the choice of exhausting function, nor the decomposition $M=M^{\inn}\cup M^{\out}$.
\\
\indent
Given a family of convex manifolds $(M,\omega_s,\theta_s)_{s\in [0,1]}$, observe that we may redefine any given exhausting function for $(M,\omega_s,\theta_s)$ so that $M^{\out}\subset M_s^{\out}$, where $M^{\out},M_s^{\out}$ are respectively conical ends for $(M,\omega_0,\theta_0),(M,\omega_s,\theta_s)$. We can thus view the family of classes $[\omega_s,\theta_s]\in H^2(M,M_s^{\out})$ as elements in $H^2(M,M^{\out})$ via pull-back, and it is meaningful to consider the case when this is a ``constant'' class. Equivalently:
}

\begin{definition}
\red{For a family of convex manifolds $(M,\omega_s,\theta_s)$ we say that the family of relative classes $[\omega_s,\theta_s]$ is constant if their images in $H^2_c(M)$ are constant via the isomorphism  $H^2(M,M_s^{\out})\cong H^2_c(M)$ of \eqref{e:isocoh}.}
\end{definition}
\begin{lemma}\label{Lemma relative class constant implies compactly supported primitive}
\red{Let $(M,\omega_s,\theta_s)$ be a family of convex manifolds with constant class $[\omega_s,\theta_s]$.}
Let $\varphi_s: M \to M$ be an isotopy, $\varphi_0=\mathrm{id}$, and $\varphi_s^*\theta_s=\theta_0$ at infinity. Then, for a smooth family of compactly supported $1$-forms $\lambda_s$,
\begin{equation}\label{Eqn difference forms is d cpt sup}
\varphi_s^*\omega_s-\omega_0 = d\lambda_s.
\end{equation}
\end{lemma}
\begin{proof}
By \eqref{e:isocoh}, $[\omega_s-d(\rho \theta_s)]\in H^*_c(M)$ is constant. By Lemma \ref{Lemma varphit preserves relative class}, $[\varphi_s^*\omega_s - d(\varphi_s^*(\rho \theta_s))]$ is constant in $H^*_c(M)$. It follows that
$\varphi_s^*\omega_s-\omega_0=d(\varphi_s^*(\rho \theta_s)-\rho \theta_0) + d\lambda_s'$
 for some compactly supported one-forms\footnote{That $\lambda_s'$ can be chosen smoothly in $s$ follows because a smooth path $[0,1]\to \Omega^2_{exact}(M,M^{\out})$ can be lifted to $[0,1]\to \Omega^1(M,M^{\out})$ via the smooth surjective linear map $d: \Omega^1(M,M^{\out})\to \Omega^2_{exact}(M,M^{\out}) \subset \Omega^2(M,M^{\out})$. Locally this corresponds to choosing a smooth family of orthogonal complements to $\ker d$, which can be achieved by taking orthogonal complements with respect to a choice of Riemannian metric.} $\lambda_s'$. By construction $\varphi_s^*(\rho \theta_s)-\rho \theta_0$ vanishes at infinity (since $\varphi_s^*\theta_s=\theta_0$ there), so we may take $\lambda_s=\lambda_s'+\varphi_s^*(\rho \theta_s)-\rho \theta_0$ in \eqref{Eqn difference forms is d cpt sup}.
%
% PROOF BELOW WOULD REQUIRE THAT phi_s preserves Mout
% 
%%Since $[\omega_s,\theta_s]=[\omega_0,\theta_0]$, 
%By Lemma \ref{Lemma varphit preserves relative class}, 
%%implies
%%
%%
%$[\varphi_s^*\omega_s-\omega_0,\varphi_s^*\theta_s-\theta_0]
%= 0
%\in H^2(M,M^{\out}).
%$
%%
%%
%As $\varphi_s^*\theta_s-\theta_0=0$ for large $r$, pick $\rho$ as above with $\rho(\varphi_s^*\theta_s-\theta_0)=0$ on $M$. Then that vanishing relative class maps to $[\varphi_s^*\omega_s-\omega_0]$ via 
%%
%%
%$
%H^2(M,M^{\out})\cong H^2_c(M).
%%,\; [\varphi_s^*\omega_s-\omega_0,\varphi_s^*\theta_s-\theta_0]\mapsto [\varphi_s^*\omega_s-\omega_0]
%$
%%
%%
%So \eqref{Eqn difference forms is d cpt sup} follows.
\end{proof}
\begin{lemma}\label{Lemma from iso get betas are exact}
\red{Let $(M,\omega_s,\theta_s)$ be a family of convex manifolds with constant class $[\omega_s,\theta_s]$. Let $(\hat{M}_s,\hat{\omega}_s,\hat{\theta}_s)$ denote the completion of $(M,\omega_s,\theta_s)$. Suppose we are given an isomorphism $\varphi_s:(\hat M_0,\hat\omega_0+\beta_s,\hat\theta_0)\to (\hat M_s,\hat\omega_s,\hat\theta_s)$ for compactly supported $2$-forms $\beta_s$ on $\hat M_0$, satisfying $\varphi_s|_{M^{\inn}_0} = \mathrm{id}_{M_0^{\inn}}$ and $\varphi_0=\mathrm{id}_{\hat M_0}.$
Then there is a family of compactly supported $1$-forms $\lambda_s$ on $\hat{M}_0$, with $\lambda_0=0$, such that $\beta_s=d\lambda_s$.
}
\end{lemma}
\begin{proof}
\red{The manifolds $\hat{M}_s$ are all diffeomorphic to $M$ through a family of proper diffeomorphisms $\psi_s$, with $\psi_0=\mathrm{id}$, which are equal to the identity on $M_0^{\inn}$ (indeed the conical ends attached to $M$ in the completion construction can be radially isotoped towards $M$). Composing yields an isotopy $\psi_s\circ \varphi_s: \hat{M}_0 \to \hat{M}_0$ which equals the identity on $M_0^{\inn}$. Lemma \ref{Lemma relative class constant implies compactly supported primitive} therefore implies that $\beta_s=(\hat{\omega}_0+\beta_s)-\hat{\omega}_0=d\lambda_s$ for some compactly supported $1$-forms $\lambda_s$ on $\hat{M}_0$, with $\lambda_0=0$.
}\end{proof}

\begin{remark}
If $H^1(\Sigma)=0$ over $\R$, $[\omega_s,\theta_s]\in H^2(M,M^{\out})$ is constant in $s$ precisely if $[\omega_s]\in H^2(M)$ is.
This also holds if $H^1(M)\to H^1(M^{\out})\cong H^1(\Sigma)$ is surjective, by the long exact sequence $H^1(M)\to H^1(M^{\out})\to H^2(M,M^{\out})\to H^2(M)$.
\end{remark}
\subsection{Deformations at infinity:\,varying the contact form}
\label{Subsection varying alpha in a family}
%%%%%%%%%%%%%%%%%%%%%%%%%%%%%%%%%%%%%%%%%%%%%%%%%%%%%%%%%%%%%%%%%%%%%%%%%
%%%%%%%%%%%%%%%%%%%%%%%%%%%%%%%%%%%%%%%%%%%%%%%%%%%%%%%%%%%%%%%%%%%%%%%%%
%
%
Let $(M,\omega,\theta)$ be convex \red{and complete} (Definition \ref{d:convex}). 
We often blur the distinction between the domain and image of $j$ in \eqref{Equation Intro j} so $\Sigma\equiv \partial M^{\inn}$ and $\alpha=\theta|_{\Sigma}$.
Before proving Theorem \ref{t:invariance1}, we will prove a technical result which shows that any deformation of contact forms $(\alpha_s)_{0\leq s\leq 1}$ on $\Sigma$ with $\alpha_0=\alpha$ can be recovered by a deformation $i_s$ of the conical parametrisation \eqref{Equation Intro j}. Applying Gray's stability theorem to the family $\alpha_s$ yields a smooth family of functions $f_s: \Sigma\to \R$ and a contact isotopy $\psi_s:\Sigma\to \Sigma$, $\psi_0=\mathrm{id}$, with 
$\psi_s^*\alpha_s=e^{f_s}\alpha_0.$
\red{Given $\epsilon>0$, for any constant $c$ with $c>\epsilon$ and $c> \epsilon + \max f_s$, we obtain the subset %
\[
M_{\alpha_s}^{\out}=\big\{(\psi_s(y),r): y\in \Sigma,\ r\geq c -f_s(y)\big\}\subset \Sigma\times (\epsilon,\infty)
\]
having as boundary 
\[
\Sigma_{\alpha_s} =\big\{(\psi_s(y),c -f_s(y)): y\in \Sigma\big\}.
\]}
%
%
\begin{comment}
\begin{center} \input{image01d.pdf_t} \end{center}
\end{comment}
%By changing $\Sigma$ to the ``graph'' $\Sigma_s$ of $f_s$, as in the picture, we obtain another parametrization $j_s$ of the end using the contact manifold $(\Sigma,\alpha_s)$ instead of $(\Sigma,\alpha)$. 

\begin{lemma}\label{Lemma family of alpha gives parametrization}
The contact forms $\alpha_s$ determine conical parametrisations $i_s$ of the end $M_{\alpha_s}^{\out}$ of $M$, for the contact manifold $(\Sigma,\alpha_s)$ (see the figure below). Namely,
\begin{center} \begin{picture}(0,0)%
\includegraphics{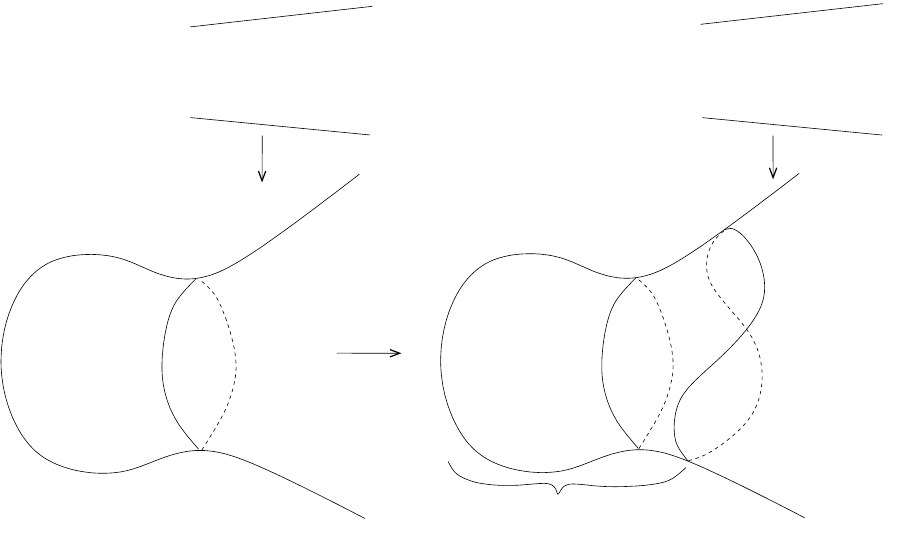}%
\end{picture}%
\setlength{\unitlength}{1533sp}%
\begingroup\makeatletter\ifx\SetFigFont\undefined%
\gdef\SetFigFont#1#2#3#4#5{%
  \reset@font\fontsize{#1}{#2pt}%
  \fontfamily{#3}\fontseries{#4}\fontshape{#5}%
  \selectfont}%
\fi\endgroup%
\begin{picture}(11174,6575)(362,-2758)
\put(2833,2623){\makebox(0,0)[lb]{\smash{{\SetFigFont{8}{9.6}{\rmdefault}{\mddefault}{\updefault}{\color[rgb]{0,0,0}$d(e^r \alpha)$}%
}}}}
\put(2826,3088){\makebox(0,0)[lb]{\smash{{\SetFigFont{8}{9.6}{\rmdefault}{\mddefault}{\updefault}{\color[rgb]{0,0,0}$\Sigma\times [0,\infty)$}%
}}}}
\put(909,-1406){\makebox(0,0)[lb]{\smash{{\SetFigFont{8}{9.6}{\rmdefault}{\mddefault}{\updefault}{\color[rgb]{0,0,0}$\omega$}%
}}}}
\put(9496,1798){\makebox(0,0)[lb]{\smash{{\SetFigFont{8}{9.6}{\rmdefault}{\mddefault}{\updefault}{\color[rgb]{0,0,0}$i_s$}%
}}}}
\put(2472,-560){\makebox(0,0)[lb]{\smash{{\SetFigFont{8}{9.6}{\rmdefault}{\mddefault}{\updefault}{\color[rgb]{0,0,0}$\Sigma_0$}%
}}}}
\put(909,-981){\makebox(0,0)[lb]{\smash{{\SetFigFont{8}{9.6}{\rmdefault}{\mddefault}{\updefault}{\color[rgb]{0,0,0}$M^{in}$}%
}}}}
\put(9136,3089){\makebox(0,0)[lb]{\smash{{\SetFigFont{8}{9.6}{\rmdefault}{\mddefault}{\updefault}{\color[rgb]{0,0,0}$\Sigma\times [c,\infty)$}%
}}}}
\put(9136,2639){\makebox(0,0)[lb]{\smash{{\SetFigFont{8}{9.6}{\rmdefault}{\mddefault}{\updefault}{\color[rgb]{0,0,0}$d(e^r \alpha_s) \textrm{ where}$}%
}}}}
\put(11521,2639){\makebox(0,0)[lb]{\smash{{\SetFigFont{8}{9.6}{\rmdefault}{\mddefault}{\updefault}{\color[rgb]{0,0,0}$\alpha_s=e^{f_s}\alpha$}%
}}}}
\put(3106,1808){\makebox(0,0)[lb]{\smash{{\SetFigFont{8}{9.6}{\rmdefault}{\mddefault}{\updefault}{\color[rgb]{0,0,0}$i_0$}%
}}}}
\put(6340,-1402){\makebox(0,0)[lb]{\smash{{\SetFigFont{8}{9.6}{\rmdefault}{\mddefault}{\updefault}{\color[rgb]{0,0,0}$\omega$}%
}}}}
\put(6352,-982){\makebox(0,0)[lb]{\smash{{\SetFigFont{8}{9.6}{\rmdefault}{\mddefault}{\updefault}{\color[rgb]{0,0,0}$M^{in}$}%
}}}}
\put(4653,-266){\makebox(0,0)[lb]{\smash{{\SetFigFont{8}{9.6}{\rmdefault}{\mddefault}{\updefault}{\color[rgb]{0,0,0}$\mathrm{id}$}%
}}}}
\put(9946,170){\makebox(0,0)[lb]{\smash{{\SetFigFont{8}{9.6}{\rmdefault}{\mddefault}{\updefault}{\color[rgb]{0,0,0}$\Sigma_{\alpha_s}=\textrm{``graph"}(f_s)$}%
}}}}
\put(6880,-2635){\makebox(0,0)[lb]{\smash{{\SetFigFont{8}{9.6}{\rmdefault}{\mddefault}{\updefault}{\color[rgb]{0,0,0}$M_{\alpha_s}^{in}$}%
}}}}
\put(10525,-1389){\makebox(0,0)[lb]{\smash{{\SetFigFont{8}{9.6}{\rmdefault}{\mddefault}{\updefault}{\color[rgb]{0,0,0}$M_{\alpha_s}^{out}$}%
}}}}
\put(10542,-1863){\makebox(0,0)[lb]{\smash{{\SetFigFont{8}{9.6}{\rmdefault}{\mddefault}{\updefault}{\color[rgb]{0,0,0}$d\theta$}%
}}}}
\put(3689,-1851){\makebox(0,0)[lb]{\smash{{\SetFigFont{8}{9.6}{\rmdefault}{\mddefault}{\updefault}{\color[rgb]{0,0,0}$d\theta$}%
}}}}
\put(3684,-1389){\makebox(0,0)[lb]{\smash{{\SetFigFont{8}{9.6}{\rmdefault}{\mddefault}{\updefault}{\color[rgb]{0,0,0}$M^{out}$}%
}}}}
\end{picture}%
 \end{center}
%By changing $\Sigma$ to the ``graph'' $\Sigma_s$ of $f_s$, as in the picture, we obtain another parametrization $j_s$ of the end using the contact manifold $(\Sigma,\alpha_s)$ instead of $(\Sigma,\alpha)$. 
%
%
\begin{align*}
i_s&:\big(\Sigma\times [c,\infty),e^{r_s}\alpha_s,d(e^{r_s}\alpha_s)\big)\to (M_{\alpha_s}^{\out},\omega|_{M_{\alpha_s}^{\out}},\theta),\\  i_s&(y,r_s)=\mathrm{Flow}^Z_{r_s-c}\big(\psi_s(y),c-f_s(y)\big)=
\big(\psi_s(y),r_s - f_s(y)\big).
\end{align*}
In particular $i_s^*\theta=i_s^*(e^r\alpha_0)=e^{r_s}\alpha_s$ and $i_s(\Sigma\times\{c\})=\Sigma_{\alpha_s}$. The new data induced by $i_s$ is $Z_s=Z$ with radial coordinate $r_s = r+f_s(y)\in [c,\infty)$.
\end{lemma}
\begin{proof} 
We readily compute\[
i_s^*(e^r \alpha_0)|_{(y,r_s)} = e^{r_s - f_s(y)} \psi_s^*\alpha_0 = e^{r_s - f_s(y)} e^{f_s(y)}\alpha_s= e^{r_s}\alpha_s.\qedhere
\]
\end{proof}
\begin{remark}\label{Remark can vary alpha arbitrarily}
\red{The family $i_s$ yields a family $M^{\inn}_{\alpha_s}:=\overline{M\setminus M^{\out}_{\alpha_s}}\subset M$ of diffeomorphic manifolds.} This yields an isotopy $\varphi_s: M \to M$, $\varphi_0=\mathrm{id}$, such that $\varphi_s(M^{\inn})=M^{\inn}_{\alpha_s}$, \red{$\varphi_s(\Sigma\times\{c\})=\Sigma_{\alpha_s}$}. Now $(M,\varphi_s^*\omega,\varphi_s^*\theta)$ is a deformation of $(M,\omega,\theta)$ admitting the conical parametrisation $\varphi_s^{-1}\circ i_s$ on $M^{\out}$ modeled on $(\Sigma,\alpha_s)$. Thus, up to isomorphisms of convex manifolds, we can arbitrarily vary the contact form $\alpha_s$ on $\Sigma$ subject to fixing the contact structure $\xi=\ker \alpha$ (using Lemma \ref{Lemma varphit preserves relative class}, Theorem \ref{t:invariance1}).
\end{remark}

%%%%%%%%%%%%%%%%%%%%%%%%%%%%%%%%%%%%%%%%%%%%%%%%%%%%%%%%%%%%%%%%%%%%%%%%%
%%%%%%%%%%%%%%%%%%%%%%%%%%%%%%%%%%%%%%%%%%%%%%%%%%%%%%%%%%%%%%%%%%%%%%%%%
  
%
% Note Z_s is time-independent v.f. so flowlines will not intersect.
% Also flow will transport Z_s, so if it was outward along Sigma=Sigma_{s,0} then it will
% be outward along Sigma_{s,t}.
% So flow keeps pushing Sigma_{s,t} outwards.
% The flow can't slow down, otherwise it converges to a zero of Z, but 
% it has no zeros since theta_s are non-zero.
%
%
% Could alternatively ask for an exhausting function h: M^{\out} \to \R such that
% dh(Z)>0. 
%

%%%%%%%%%%%%%%%%%%%%%%%%%%%%%%%%%%%%%%%%%%%%%%%%%%%%%%%%%%%%%%%%%%%%%%%%%
%%%%%%%%%%%%%%%%%%%%%%%%%%%%%%%%%%%%%%%%%%%%%%%%%%%%%%%%%%%%%%%%%%%%%%%%%
\subsection{Deformations are compactly supported up to isomorphisms}
\label{Subsection Comparing the two types of variations}
%%%%%%%%%%%%%%%%%%%%%%%%%%%%%%%%%%%%%%%%%%%%%%%%%%%%%%%%%%%%%%%%%%%%%%%%%
%%%%%%%%%%%%%%%%%%%%%%%%%%%%%%%%%%%%%%%%%%%%%%%%%%%%%%%%%%%%%%%%%%%%%%%%%

\begin{proposition}\label{Lemma varphis isotopy on M}
Let $(M,\omega_s,\theta_s)$ be a family of convex manifolds with completions $(\hat M_s,\hat\omega_s,\hat\theta_s)$. There is a family of compactly supported closed $2$-forms $\beta_s$ on $\hat M_0$ with $\beta_0=0$ such that $(\hat M_0,\hat\omega_0+\beta_s,\hat\theta_0)$ is convex and admits a family of isomorphisms 
\begin{equation*}
\varphi_s:(\hat M_0,\hat\omega_0+\beta_s,\hat\theta_0)\to (\hat M_s,\hat\omega_s,\hat\theta_s),\qquad \varphi_s|_{M^{\inn}_0} = \mathrm{id}_{M_0^{\inn}},\qquad \varphi_0=\mathrm{id}_{\hat M_0}.
% \tag*{$\qed$}
\end{equation*}
Note that $\beta_s = \varphi_s^*\hat\omega_s - \hat\omega_0$.
\end{proposition}
\begin{proof}
\red{Let us fix some $b>0$ such that $M_s^{\geq b}\subset M_0^{\out}$ for all $s\in[0,1]$. As the statement of the proposition does not depend on the choice of exhausting function used for $(M,\omega_s,\theta_s)$, we may assume $b=0$, so $M_s^{\out}\subset M_0^{\out}$.
Let $s'\in[0,1]$ be arbitrary. We claim that there exists $\epsilon_{s'}>0$ such that if $s''\in (s'-\epsilon,s'+\epsilon)$, then there exists a closed 2-form $\beta_{s',s''}$ on $\hat M_{s'}$ with compact support and an isomorphism $\varphi_{s',s''}:(\hat M_{s'},\hat\omega_{s'}+\beta_{s',s''},\theta_{s'})\to(\hat M_{s''},\hat\omega_{s''},\hat\theta_{s''})$ such that $\varphi_{s',s''}|_{M_0^{\inn}}=\mathrm{id}_{M_0^{\inn}}$. 
If the claim holds, then we can extract a finite subcover of the compact interval $[0,1]$ from the cover $\{(s'-\epsilon_{s'},s'+\epsilon_{s'})\}_{s'\in[0,1]}$. The isomorphism $\varphi_s$ in the statement of the present proposition is then obtained by a finite composition of isomorphisms of the type $\varphi_{s',s''}$. We now prove the claim. For ease of notation we just do the case $s'=0$.\\
\indent
As $M_s^{\out}\subset M_0^{\out}$,} we may assume that $Z_s$ is defined on $M^{\out}$, in particular it is non-vanishing there. Since $Z$ is transverse to $\Sigma=\partial M^{\out}$ and outward pointing, the same will hold for $Z_s$ for small $s$. By positively integrating $Z_s$ starting from $\Sigma$ we obtain a family of
conical parametrisations $j_s:\{(y,r):0\leq r<\sigma_s(y)\}\to M$ 
for $(M,\omega_s,\theta_s)$ (compare \eqref{Equation Intro j}).
%
% above implies a posteriori that one can find a *family* of Liouville functions h_s
% (using radial coord), with fixed level set Sigma at 0
%
Furthermore, one can construct an isotopy $F_s:M\to M$ with $F_0=\mathrm{id}$ such that $F_s\circ j_s|_{\Sigma\times[0,\epsilon]}=j_0|_{\Sigma\times[0,\epsilon]}$ for some $\epsilon>0$. 
As $F_s$ can be chosen to be compactly supported near $\Sigma$, we can ensure that $F_s=\mathrm{id}$ on the original $M^{\inn}$ before enlarging and $F_s=\mathrm{id}$ at infinity.
Observe that
if we can prove the proposition for $(M,F_s^*\omega_s,F_s^*\theta_s)$ then it will also follow for $(M,\omega_s,\theta_s)$ by conjugating the isomorphism by $F_s$.
% by suitably composing with $F_s$, commutative diagram.
Thus we may now assume that
\begin{equation}\label{e:jsj0}
j_s|_{\Sigma\times[0,\epsilon]}=j_0|_{\Sigma\times[0,\epsilon]}.
\end{equation}
%
% this way do not worry about large deformations,
% first fix j and h for M_0,
% then consider j_s and h_s for M_s, WLOG land in M_0^{out}
% for small s, have theta_s and theta_0 close in fixed compact region,
% so Z_s transverse to some translate Sigma_0 x {const} of Sigma_0.
% Then replace j_s by flowing via Z_s with Sigma_0 x {const} instead.
% 
% transversality is an open condition, so can straighten out the Z fields
%
The family $j_s$ determines contact forms $\alpha_s$ on $\Sigma$, with $j_s^*(\theta_s)=e^{\rho}\alpha_s$ where $\rho\in [0,\infty)$ plays the same role as $r$ in \eqref{Equation Intro j}.
By Lemma \ref{Lemma family of alpha gives parametrization}, we obtain a new family of conical parametrisations $i_s: (\Sigma \times [c,\infty)) \to M_{\alpha_s}^{\out}\subset \hat{M}$, with $i_s^*\theta_0 = e^{\rho}\alpha_s$ where $\rho\in [c,\infty)$. By construction, $i_s^{-1}(M_{\alpha_s}^{\out})\subset \hat{M}$ does not intersect $M^{\inn}$. 
% 
% and $r+f_s(\psi^{-1}(y))\geq c\geq \epsilon$.
%
The composition
\[
\varphi_s:=j_s\circ i_s^{-1}: M_{\alpha_s}^{\out}\to \hat{M}
\]
satisfies $\varphi_s^*\theta_s = \theta_0$. We will now extend $\varphi_s$ to $\hat{M}\setminus M^{\inn}$ so that $\varphi_s$ equals the identity near $\partial M^{\inn}$ (the proposition will then follow by further extending via $\varphi_s|_{M^{\inn}}=\mathrm{id}$).

\red{By construction, $M_{\alpha_s}^{\out}\subset \Sigma \times (\epsilon,\infty)\subset \hat{M}$ (where we have identified the domain and the image of $j_0$ to simplify the notation). By the definition of $i_s$ in Lemma \ref{Lemma family of alpha gives parametrization}, $\varphi_s$ satisfies the equation 
\begin{equation}\label{Equation varphi js and j0}
\varphi_s(y,r) = j_s\Big(\psi_s^{-1}(y),r+f_s\big(\psi^{-1}_s(y)\big)\Big)\subset \hat{M}_s
\end{equation}
for $(y,r)\in M_{\alpha_s}^{\out} \subset \Sigma \times (\epsilon,\infty)$. 
By construction those values satisfy the inequalities $r+f_s\big(\psi^{-1}_s(y))\geq c>\epsilon$. Thus we can extend $\varphi_s$ to $\Sigma \times [0,\infty)$ by
\[
\varphi_s(y,r)=j_s(\psi_{a_s(r)}^{-1}(y),b_s(y,r))
\] where we now explain the interpolations $a_s:[0,\infty)\to [0,1]$, $b_s: \Sigma\times [0,\infty) \to [0,\infty)$. We want $a_s(r)=s$ if $(r,y)\in M_{\alpha_s}^{\out}$ for some $y$, $a_s(r)=0$ for $r\in [0,\epsilon/2]$. We want $b_s(y,r)$ to smoothly interpolate between the value required on $M_{\alpha_s}^{\out}$ by \eqref{Equation varphi js and j0} and the function $b(y,r)=r$ for $r\in [0,\epsilon/2]$ (e.g. smoothen a linear interpolation). This ensures by \eqref{e:jsj0} that for $r\in [0,\epsilon/2]$ we have $\varphi_s(y,r)=(y,r)$, the identity map.}
\end{proof} 
\begin{remark} If the $(M,\omega_s,\theta_s)$ are complete, then the proof simplifies as one can work on a fixed manifold $M^{\inn}\cup (\Sigma \times [0,\infty))$ throughout, instead of using $\hat{M}_s$.
\end{remark}
%%%%%%%%%%%%%%%%%%%%%%%%%%%%%%%%%
\subsection{Convex domains}
%%%%%%%%%%%%%%%%%%%%%%%%%%%%%%%%%%

Convex domains $(D,\omega,\alpha)$ are defined in Rmk.\ref{Remark Intro Convex domain}. Observe $\omega$ is exact in a collar neighbourhood $C$ of $\partial D$. A choice of primitive $\theta$ on $C$ with $\theta|_{\Sigma}=\alpha$ yields a Liouville vector field $Z$ by $\iota_Z\omega=\theta$. Integrating $Z$ backwards in time from $\Sigma$, we may assume $C$ is identified with $\Sigma\times [-\epsilon,0]$. The completion $\hat D$ extends this via \eqref{Equation Intro j} to a conical end $\Sigma \times (-\epsilon,\infty)$, yielding a complete convex manifold.

\begin{remark}\label{Remark choice of primitive for convex domain}
Note that $M=D\setminus \partial D$ is a convex manifold $(M,\omega,\theta)$, as we obtained a (non-complete) conical parametrisation by $\Sigma \times [-\epsilon,0)$ above. Suppose $\theta_0,\theta_1$ are two choices of primitive as above. As $C,\Sigma$ are homotopy equivalent, and the closed form $\theta_1-\theta_0$ pulls back to zero on $\Sigma$, we have $\theta_1-\theta_0=df$ for some $f:C\to \R$. % vanishing along $\Sigma.$ 
%Thus $(\omega,\theta_0)-(\omega,\theta_1)=D(0,-f)$ in $C^*(M,M^{\out})$.
Thus $[\omega,\theta_0 + d(sf)]_{0\leq s \leq 1}$ is a constant family in $H^2(M,M^{\out})$. So Theorem \ref{t:invariance1} implies that the isomorphism class of $(D,\omega,\alpha)$ does not depend on the chosen primitive on the collar, as there is an isomorphism of the completions $(M,\omega,\theta_0)^{\wedge}\cong (M,\omega,\theta_1)^{\wedge}$.
\end{remark}

\begin{lemma}\label{Lemma Trick to get convexity}
Let $(D,\omega)$ be a closed symplectic manifold with boundary $\Sigma:=\partial D$.
Suppose that $h: D\to \R$ is a smooth function such that $\Sigma$ is a regular level set, $h$ strictly increases in the outward normal direction, and $\omega|_{T\Sigma}=d\alpha$ for some $\alpha\in \Omega^1(\Sigma)$. Then
$(D,\omega,\alpha)$ is a convex domain if and only if $\alpha (X_h)>0$.
\end{lemma}
\begin{proof}
By definition, $X_h\in T\Sigma$ so that $\omega|_{T\Sigma}(\cdot,X_h)=dh|_{\Sigma}=0$. As $\omega$ is symplectic, $\dim \ker \omega|_{T\Sigma}=1$ and $\ker \omega|_{T\Sigma}=\R\, X_h$. This readily implies that the condition $\alpha\wedge (d\alpha)^{\dim_{\C} D-1}>0$ is equivalent to $\alpha(X_h)>0$.
\end{proof}
Arguing as in Proposition \ref{Lemma varphis isotopy on M}, we obtain the following.

\begin{proposition}\label{l:deformationcompact}
Let $(D,\omega_s,\alpha_s)$ be a deformation of convex domains. After choosing a family of primitives $\theta_s$ for $\omega_s$ near $\partial D$ with $\theta_s|_{\Sigma}=\alpha_s$, by completion we obtain a family of complete convex manifolds $(M_s,\omega_s,\theta_s)$. Then there is a family of convex manifolds $(M_0,\omega_0+\beta_s,\theta_0)$ where the forms $\beta_s$ are compactly supported, $\beta_0=0$, together with a family of isomorphisms 
\begin{equation*}
\varphi_s:(M_0,\omega_0+\beta_s,\theta_0)\to (M_s,\omega_s,\theta_s),\qquad \varphi_s|_D = \mathrm{id},\qquad \varphi_0=\mathrm{id}.\tag*{$\qed$}
\end{equation*}
\end{proposition}

For small perturbations of a convex manifold which is the interior of some convex domain, Proposition \ref{Lemma varphis isotopy on M} can be adapted to the following quantitative statement.
\begin{proposition}\label{p:quantitative}
Let $(D,\omega,\alpha)$ be a convex domain. Fix a primitive $\theta$ on a collar $C$ of $\Sigma=\partial D$, $\theta|_{\Sigma}=\alpha$, making $(M,\omega,\theta)$ convex as in Remark \ref{Remark choice of primitive for convex domain} where $M=D\setminus \partial D$.  
Then there is an $\epsilon>0$ and a $K>0$ such that for all closed $(\mu,\lambda)\in \Omega^2(D,C)$ with $\Vert(\mu,\lambda) \Vert_{C^1(D,C)}<\epsilon$, the triple $(M,\omega+\mu,\theta+\lambda)$ is a convex manifold. Moreover, there is a compactly supported closed $2$-form $\beta$ on $M$ satisfying the bound $\Vert \beta\Vert_{C^0(M)}\leq K\Vert (\mu,\lambda)\Vert_{C^1(D,C)}$ and admitting an isomorphism
\[
\varphi:(M,\omega+\beta,\theta)^{\wedge}\to (M,\omega+\mu,\theta+\lambda)^{\wedge}\quad \textrm{with}\quad \varphi|_{M^\inn}=\mathrm{id}.
\]
In particular, $[\beta]\mapsto [\mu,\lambda]$ via $H^2_c(M)\cong H^2(D,C)$.\hfill\qed
\end{proposition}

\begin{remark} We can ensure $\beta$ is compactly supported in $M$, not just $\hat{M}$, since for small deformations we can ensure $\Sigma_s$ is close to $\partial M^{\inn}=\Sigma\times \{-\epsilon\}$ 
$($see Lemma \ref{Lemma family of alpha gives parametrization}$)$.
\end{remark}

%Let $(M,\omega,\theta)$ be a convex domain. Let $(N,\omega,\theta)$ be a convex submanifold obtained by taking the strict subgraph $$N=M^{\inn}\cup \{(y,r)\in \Sigma\times [0,\sigma(y)): r < c-f(y)\}$$ of some smooth function $f: \Sigma \to (c,\infty)$ with $f(y)> c-\sigma(y)$. Let $N^{\inn}$ denote a subgraph in the interior of $N$, containing $M^{\inn}$. Then there is an $\epsilon>0$ and a $K>0$ depending only on $N^\inn,N,(\omega,\theta)|_N$, such that for all closed $(\mu,\lambda)\in \Omega^2(N,N^\out)$ with $\Vert(\mu,\lambda) \Vert_{C^1(N,N^\out)}<\epsilon$ we obtain a convex manifold $$(N,\omega+\mu,\theta+\lambda).$$ There is a closed 2-form $\beta$ on $M$ supported in $N$, with $\Vert \beta\Vert_{C^0(N)}\leq K\Vert (\mu,\lambda)\Vert_{C^1(N,N^\inn)}$, and an isomorphism $\varphi:( M,\omega+\beta,\theta)^{\wedge}\to (N,\omega+\mu,\theta+\lambda)^{\wedge}$   with $\varphi|_{N^\inn}=\mathrm{id}$.

\subsection{Proof of Theorem \ref{t:invariance1}}
\label{Subsection Constant relative class implies compactly supported deformation}
By Proposition \ref{Lemma varphis isotopy on M}, we have a family of isomorphisms
\[
\varphi_s:(\hat M,\omega_0+\beta_s,\theta_0)\to \hat{M}_s:=(M,\omega_s,\theta_s)^{\wedge}.
\]
% the hat M as a manifold is independent of s since same theta 0 at infinity
By Lemma \ref{Lemma from iso get betas are exact},
there is a family of compactly supported $1$-forms $\lambda_s$, $\lambda_0=0$, with
\begin{equation}\label{Equation betas is exact}
\beta_s=d\lambda_s.
\end{equation}
We now run Moser's argument. Let $V_s$ be the compactly supported vector field on $\hat{M}$ determined by $(\omega_0+\beta_s)(V_s,\,\cdot\,)=-\partial_s\lambda_s$. Let $F_s:\hat M\to\hat M$ be the isotopy generated by $V_s$, so $\partial_s F_s = V_s \circ F_s$, $F_0=\mathrm{id}$. Then $F_s^*(\omega_0+\beta_s) = \omega_0$ on $\hat{M}$, and $F_s^*\theta_0 = \theta_0$ at infinity since $F_s$ is compactly supported. So we obtain the isomorphism
\[\varphi:=\varphi_1\circ F_1:(M,\omega_0,\theta_0)^{\wedge}\to (M,\omega_1,\theta_1)^{\wedge}
\]
claimed in Theorem \ref{t:invariance1}. \hfill\qed
%
	% Sanity check: \partial_s F_s^*\omega_s' = 
	% = F_s^*(\partial_s \omega_s' + Lie derivative_{V_s}(\omega_s')
	% But the Lie derivative is 
	% d i_{V_s} \omega_s' + i_{V_s} d\omega_s'
	% = d (-\gamma_s)
	% which kills the term \partial_s \omega_s'
	% so F_s^*\omega_s' is constant.
% \begin{remark*}
% 	Sec.\ref{Subsection Comparing the two types of variations} shows that
% 	varying $(\omega_s,\theta_s)$ on $M$, keeping $[\omega_s,\theta_s]$ constant, is the same up to isomorphism to keeping $(\omega,\theta)$ fixed but varying the contact hypersurface $\Sigma_s =\{(\psi(y),c -f_s(y)): y\in \Sigma\}$ by varying the conical parametrisations $j_s$ in Sec.\ref{Subsection varying alpha in a family}.
% \end{remark*}

\begin{remark}[Liouville isomorphisms]\label{Remark invariance in Liouville case}
	For Liouville manifolds (so $\theta_s$ extends to $M^{\inn}$ with $\omega_s=d\theta_s$ on all of $M$), the $(\omega_s,\theta_s)=D(\theta_s,0)$ are automatically a constant class. Let $\bar{\theta}_s:=\varphi_s^*\theta_s$. Then $\bar{\theta}_0=\theta_0$, $d\bar{\theta}_s=\omega_0+\beta_s$, and $\bar{\theta}_s=\theta_0$ at infinity. So $\lambda_s:=\bar{\theta}_s-\theta_0$ satisfies \eqref{Equation betas is exact}. Then, for $V_s,F_s$ as above, apply Cartan's formula:
	\[
\partial_s (F_s^*\bar\theta_s) = F_s^*[\partial_s \bar\theta_s + d\iota_{V_s}(\bar\theta_s) + \iota_{V_s}d\bar\theta_s] =
	d(F_s^*\iota_{V_s}(\bar\theta_s)).
	\]
	Now define the compactly supported function $g:=\int_0^1 (F_s^*i_{V_s}(\bar\theta_s))\, ds$. Then
	\begin{equation}\label{Equation Symplectomorphism of contact type at infinity}
	\varphi^*\theta_1 - \theta_0 = dg.
	\end{equation}
	Such maps $\varphi$ are called \textit{Liouville isomorphisms}. This is a proof of \cite[Lemma 2.2]{SeidelBiased}.
\end{remark}

\begin{example}\label{Example Harris}
Let $(\omega_s)_{0\leq s\leq 1}$ be cohomologous symplectic forms making $M^{\inn}$ a convex domain (Remark \ref{Remark Intro Convex domain}). Extending to a completion $W=M^{\inn}\cup (\Sigma\times [0,\infty))$ we may assume $\omega_s=d\theta_s$ on the conical end where $\theta_s = e^r \alpha_s$ on $\Sigma\times [0,\infty)$, and $(\Sigma,\alpha_s)$ is of positive contact type. By assumption, $\frac{d}{ds}\omega_s=d\sigma_s$ for some $1$-forms $\sigma_s$ on $W$.

If $H^1(\Sigma)=0$ (which implies $[\omega_s,\theta_s]$ is constant), one can pick $\sigma_s$ with $\sigma_s=e^r\alpha_s$ on the end. One can construct $\varphi_s$ as a (typically non-compactly supported) flow of a vector field $V_s$ by Moser's argument, so $\sigma_s=-\omega_s(V_s,\cdot)$. Thus, on the end, 
\[
\frac{d\alpha_s}{ds}=-[d\alpha_s + dr\wedge \alpha_s](V_s,\cdot)
\]
after canceling out $e^r$ factors. So $V_s$ is integrable because the radial component $-\frac{d\alpha_s}{ds}(Y_s)\, Z$ (where $Y_s$ is the Reeb vector field for $\alpha_s$) is harmless. This is the argument in Harris \cite[Lemma 6.1]{Harris}.\footnote{The statement \cite[Lemma 6.1]{Harris} is missing the assumption $H^1(\Sigma)=0$, otherwise there may be an obstruction to extending the closed form $\sigma_s - e^r\alpha_s\in H^1(W^{\out})$ to $W$. In Harris' applications, $\Sigma \cong ST^*S^3\cong S^3\times S^2$ has $H^1(\Sigma)=0$ as required.}
% S^3 is Lie gp, so TS^3 trivial, so S^3 times R^3, so sphere bdle is S^3 x S^2 
In Section \ref{Subsection An example of a Quasi-Liouville magnetic TT2} we consider $(T^*T^2,d\theta)$ but we use a different primitive $\alpha$ at infinity making the relative class $[d\theta,\alpha]$ non-trivial. 
% bdry of D*T^2 is a T^3, and we use a non-exact closed 1 form on T^3 on the end
In this case, the Moser argument cannot yield an isomorphism, due to the obstructed relative class. Nevertheless we prove that symplectic cohomology is invariant.
\end{example}

%%%%%%%%%%%%%%%%%%%%%%%%%%%%%%%%%%%%%%%%%%%%%%%%%%%%%%%%%%%%%%%%%%%%%%%%%
%%%%%%%%%%%%%%%%%%%%%%%%%%%%%%%%%%%%%%%%%%%%%%%%%%%%%%%%%%%%%%%%%%%%%%%%%
\section{Symplectic cohomology for convex manifolds}
\label{Section SH}
%%%%%%%%%%%%%%%%%%%%%%%%%%%%%%%%%%%%%%%%%%%%%%%%%%%%%%%%%%%%%%%%%%%%%%%%%
%%%%%%%%%%%%%%%%%%%%%%%%%%%%%%%%%%%%%%%%%%%%%%%%%%%%%%%%%%%%%%%%%%%%%%%%%
\subsection{Symplectic cohomology}
\label{Subsection Symplectic cohomology}
%%%%%%%%%%%%%%%%%%%%%%%%%%%%%%%%%%%%%%%%%%%%%%%%%%%%%%%%%%%%%%%%%%%%%%%%%
%%%%%%%%%%%%%%%%%%%%%%%%%%%%%%%%%%%%%%%%%%%%%%%%%%%%%%%%%%%%%%%%%%%%%%%%%

Symplectic cohomology for Liouville manifolds was constructed by Viterbo \cite{Viterbo}, see also the surveys \cite{OanceaEnsaios,Ritter3,SeidelBiased}.  Symplectic cohomology for complete convex manifolds was constructed by the second author \cite{Ritter2}, so we will only make some remarks here. We mention some finer points in Section \ref{Subsection Novikov field, Action $1$-form, Energy}.

Let $(M,\omega,\theta)$ be convex. From now on, we use the \textit{radial coordinate} $R=e^r\in [1,\infty)$, so $j^*\theta=R\alpha$. Recall the \textit{Reeb vector field} $Y$ on $\Sigma$ is determined by $\alpha(Y)=1$, $d\alpha(Y,\cdot)=0$. By \textit{Reeb periods} we mean the periods of closed orbits of $Y$. The contact form $\alpha$ is always assumed to have been perturbed generically (using Remark \ref{Remark can vary alpha arbitrarily}), so that the Reeb orbits are transversally non-degenerate and the Reeb periods form a discrete subset of $\R^+$. The choice of perturbation does not affect $SH^*(M,\omega,\theta)$ up to isomorphism, by Theorems \ref{Theorem invariance under iso of SH} and \ref{t:invariance1}.

Recall $SH^*(M,\omega,\theta)$ is the direct limit in \eqref{Equation Intro SH is direct lim}, and we now describe the class of Hamiltonians $H:M \to \R$ more precisely. Recall we identify $M^{\out}$ with the image of $j$ in \eqref{Equation Intro j}. We always assume that $H$ is \textit{radial} at infinity, meaning $H=h(R)$ only depends on the radial coordinate $R$, thus $X_H=h'(R)Y$. This yields a one-to-one correspondence between $1$-periodic Hamiltonian orbits $x: S^1 \to M$ lying in a slice $\Sigma\times \{R\}$ with $h'(R)=T\neq 0$ and Reeb orbits $y:[0,T]\to M$ of period $T$, via $y(t) = x(t/T)$. The $1$-orbits of $X_H$ will be transversally non-degenerate, so non-degeneracy will be ensured by a generic one-periodic time-dependent perturbation of $H$ (which we suppress from the notation -- the choice of perturbations will not affect the Floer cohomology groups up to isomorphism). 

Pick some $R_{\infty}$ such that $j(\Sigma \times \{R_{\infty}\})\subset M^{\out}$ (for example $R_{\infty}$ close to $1$). Call $M_{\infty}\subset M$ the region $R\geq R_{\infty}$. Then assume that the Hamiltonian $H=h(R)$ is linear in $R$ on $M_{\infty}$, with slope $m=h'(R)$ different from the Reeb periods. By the previous two paragraphs, this implies that there are no $1$-orbits of $X_H$ in $M_{\infty}$.

The Floer complex is generated by the $1$-orbits of $X_H$, but the differential depends on a choice of almost complex structure $J$ on $M$ \textit{compatible} with $\omega$. This means:
\[
\omega(Ju,Jv)=\omega(u,v),\qquad \omega(v,Jv)>0,\qquad \forall\,u,v\in TM,\ \ u,v\neq 0.
\]
This yields a Riemannian metric $g=\omega(\cdot,J\cdot)$ on $M$. The data $(H,J)$ yields the differential, which counts \textit{Floer trajectories}, i.e.~solutions $u=u(s,t): \R \times S^1 \to M$ of the equation $\partial_s u + J(\partial_t u - X_H)=0$ that are isolated up to $\R$-translation.

We must ensure that these trajectories do not escape to infinity so that moduli spaces of Floer trajectories have well-behaved compactifications by broken trajectories. A maximum principle will hold, i.e. $R\circ u$ cannot attain a local maximum in $M_{\infty}$, if we choose $J$ to be of \textit{contact type} on $M_{\infty}$, meaning
\begin{equation}\label{Eqn J is of contact type}
JZ=Y \qquad \qquad (\textrm{equivalently } J^*\theta=dR).
\end{equation}
The possible structures $J$ as above form a non-empty contractible space, which is used to show that the Floer cohomology groups $HF^*(H)$ in \eqref{Equation Intro SH is direct lim} are independent up to isomorphism on the choice of $J$.
We recall that a generic time-dependent perturbation of $J$ is needed on $M\setminus M_{\infty}$ to ensure that moduli spaces of Floer trajectories are smooth manifolds (we suppress the perturbation from the notation -- the choice of perturbation will not affect the Floer cohomology groups up to isomorphism). 

The Floer cohomology group $HF^*(H)=HF^*(M;H,J)$ will be independent of the choices of $H,J,R_{\infty}$, in fact it is isomorphic to $HF^*(\hat{M};H,J)$ computed for the completion $\hat{M}$ for any generic time-dependent $(H,J)$ subject to $H$ having eventually slope $m$ at infinity and $J$ being of contact type at infinity. This is because continuation isomorphisms can be constructed provided the slope $m$ at infinity is constant. 

\textit{Continuation homomorphisms} $HF^*(H_+,J_+)\to HF^*(H_-,J_-)$ for different choices of the data $(H,J)$ can only be constructed if the slopes satisfy $m_+\leq m_-$ (only then a maximum principle holds). These maps count isolated solutions $u:\R \times S^1 \to M$ of $\partial_s u + J_s(\partial_t u - X_{H_s})=0$ where $(H_s,J_s)=(H_{\pm},J_{\pm})$ for $s$ close to $\pm \infty$.

As $HF^*(H)$ only depends on $m$ up to isomorphism, the direct limit in \eqref{Equation Intro SH is direct lim} can therefore be taken for any sequence $H_k$ with increasing slopes $m_k\to \infty$, using the continuation homomorphisms. The direct limit will be independent up to isomorphism on the choices. 
The above discussion implies that
\[
SH^*(M,\omega,\theta)\cong SH^*\big((M,\omega,\theta)^{\wedge}\big),
\]
in particular Theorem \ref{Theorem invariance under iso of SH} implies that this group up to isomorphism only depends on the isomorphism class of $(M,\omega,\theta)$.
By the same arguments as in \cite{Ritter3}, $SH^*(M,\omega,\theta)$ admits a pair-of-pants product and a unit, and the unital algebra $SH^*(M,\omega,\theta)$ only depends on the isomorphism class of $(M,\omega,\theta)$.

One typically chooses $H$ to be Morse and $C^2$-small in the compact region where $H$ is not radial (in the above sense), and one perturbs $J$ time-independently on this region so that $(H,g)$ is a Morse-Smale pair. This ensures that the Floer complex on this region reduces to the Morse complex for $H,g$ (the $1$-orbits of $X_H$ become non-degenerate constant orbits, and the Floer trajectories become time-independent $-\nabla H$ flow lines). If $m>0$ is smaller than all Reeb periods, one can ensure that one globally obtains a Morse complex for $M$. This implies that there is a canonical map
\[
c^*: QH^*(M,\omega)\to SH^*(M,\omega,\theta),
\]
where $QH^*(M,\omega)$ is the quantum cohomology (the quantum product is constructed using $J$ as above, but the unital algebra $QH^*(M,\omega)$ only depends on $\omega$ up to isomorphism). In particular, in the quantum product, holomorphic spheres $u:\C P^1 \to M$ are counted with weight $t^{\int u^*\omega}$.
 The same argument as in \cite{Ritter3} shows that $c^*$ is a unital algebra homomorphism. 

\begin{remark}[Viterbo's trick]\label{Remark Viterbo trick}
We remark that a generalisation of a key idea due to Viterbo \cite{Viterbo} still applies here:
if $c^*$ is not a unital algebra isomorphism, then there must exist a closed Reeb orbit in $\Sigma$. Indeed, if for all choices of $H$ there never existed a non-constant $1$-orbit of $X_H$ in the region where $H$ is radial, then one could easily construct a family $H_k$ that forces $c^*$ to be an isomorphism. 
\end{remark}

%%%%%%%%%%%%%%%%%%%%%%%%%%%%%%%%%%%%%%%%%%%%%%%%%%%%%%%%%%%%%%%%%%%%%%%%%
\subsection{Novikov field, Action $1$-form, Energy}
\label{Subsection Novikov field, Action $1$-form, Energy}
%%%%%%%%%%%%%%%%%%%%%%%%%%%%%%%%%%%%%%%%%%%%%%%%%%%%%%%%%%%%%%%%%%%%%%%%%
The groups $HF^*,SH^*,QH^*$ above are all defined over the \textit{Novikov field} $\Lambda$ in a formal variable $t$ over a base field $\K$,
\begin{equation}\label{Equation Novikov ring}
\Lambda = \Big\{ \sum_{i=1}^{\infty} a_i t^{n_i}\ : \ a_i\in \K,\ n_i\in \R,\ n_i \to \infty \Big\}.
\end{equation}

For any $H$ as above, there is a (typically non-exact) action 1-form $d\mathcal A_H$ on the space of free loops $\mathcal L M= C^{\infty}(S^1,M)$,
\[
d\mathcal A_H:=d\mathcal H -\tau_\omega
\]
where we define the function $\mathcal{H}:\mathcal{L}M\to \R$ by
\[\mathcal H(x):=\int_0^1H(x(t))\ dt,\]
and $\tau_\omega$ is the transgression $1$-form on $\mathcal{L}M$ defined by
\[
\tau_\omega (\xi) := \int_0^1 \omega(\xi(t),\partial_t x)\, dt
\]
for $\xi\in T_x \mathcal{L}M = C^{\infty}(x^*TM)$. Thus
$dA_H(\xi) = -\int_0^1 \omega(\xi(t),F(x))\, dt$
\;where
\begin{equation}\label{Equation definition of F(x) = partialt x - XH}
F=F_H: \mathcal{L}M\to \bigcup_{x\in \mathcal{L}M}\, x^*TM, \qquad F(x)=\partial_t x - X_H.
\end{equation}
Thus $1$-orbits of $X_H$ are the zeros of $F$, equivalently the zeros of $dA_H$, and
Floer trajectories are maps $u:\R\to \mathcal{L}M$ satisfying \textit{Floer's equation}: $F(u)=J\partial_su$.

Let $\mathcal M(x,y;H,J)$ be the space of rigid Floer trajectories $u: \R \times S^1 \to M$ from $x$ to $y$, modulo shift in the $s$-variable.
Then the energy is
\begin{equation}\label{e:energy}
E(u): = \int_{\R \times S^1} |\partial_s u|^2\, ds\wedge dt=\int_{\R \times S^1} |F(u)|^2\, ds\wedge dt.
\end{equation}
The differential $\partial$ on the Floer complex
\[
CF^*(H)=\oplus \{\Lambda x: F(x)=0\}
\]
is explicitly
\begin{equation}\label{Equation Floer differential}
\partial y= \sum_{u \in \mathcal{M}(x,y;H,J)}\epsilon(u) t^{-d\mathcal{A}_H(u)} x
= \sum_{u \in \mathcal{M}(x,y;H,J)}\epsilon(u) t^{\tau_{\omega}(u)}\, t^{\mathcal H(x)-\mathcal{H}(y)} x
\end{equation}
where $\epsilon(u)\in\{\pm1\}$ are orientation signs (which we will not discuss), 
and $d\mathcal{A}_H(u)$ and $\tau_{\omega}(u)$ are evaluations of these $1$-forms on the $1$-chain $u$ in $\mathcal{L}M$. In particular,
\begin{equation}\label{Equation nu definition in terms of variation of action}
\tau_\omega(u)=\int_{\R\times S^1}u^*\omega.
\end{equation}
The maximum principle and Gromov compactness imply that the coefficient of $x$ in \eqref{Equation Floer differential} belongs to $\Lambda$ if for every $C>0$ there is an $E_C(x,y;H,J)>0$, such that
\begin{equation}\label{e:weights}
\forall\, u\in\mathcal M(x,y;H,J),\quad \tau_\omega(u)\leq C\quad \Longrightarrow\quad E(u)\leq E_C(x,y;H,J).
\end{equation}

From Floer's equation, we see that this condition is satisfied since 
\begin{align}\label{Equation Energy of Floer traj}
 E(u)=
\tau_\omega(u)+\mathcal H(x)-\mathcal H(y).
\end{align}
From $(CF^*(H,J),\partial)$ one obtains the $HF^*(H)$ mentioned in Section \ref{Subsection Symplectic cohomology}. The continuation maps $\phi:HF^*(H_+,J_+)\to HF^*(H_-,J_-)$ are explicitly 
\begin{equation}\label{Equation Continuation map twist}
\phi y = \sum_{u \in \mathcal{M}(x,y;H_s,J_s)}\epsilon(u) t^{\tau_\omega(u)\,}t^{\mathcal H_-(x)-\mathcal H_+(y)} \,x
\end{equation}
where $\mathcal{M}(x,y;H_s,J_s)$ is the moduli space of rigid maps $u:\R \to \mathcal L M$ satisfying $F_{H_s}(u)=J_s\partial_s u$ for $(H_s,J_s)$ as in Section \ref{Subsection Symplectic cohomology}. In this case,
\[
E(u)=\tau_\omega(u)+\mathcal H_-(x)-\mathcal H_+(y)+\int_{\R\times S^1} (\partial_s H_s)(u(s,t))\, ds \wedge dt.
\]
If we require $\partial_s H_s \leq 0$ (which forces the slopes $m_s$ at infinity to decrease), then the maximum principle holds and we have the estimate
\[
E(u)\leq \tau_\omega(u)+\mathcal H_-(x)-\mathcal H_+(y),
\]
which implies \eqref{e:weights} for the new moduli space and so $\phi$ is well-defined. More generally, $\phi$ is well-defined if we just require $m_s$ to decrease, since on the right above we get a harmless additional term $+\,c\cdot \max \,\{ |\partial_s H_s(p)|: p \in M\setminus M_{\infty}\}$, since there is an $M_{\infty}$ independent of $s$ that works for all $H_s$ in the notation of Section \ref{Subsection Symplectic cohomology}. The constant $c$ is the measure of the bounded set of $s\in \R$ for which $H_s$ is $s$-dependent.

\subsection{The BV-operator}
\label{Subsection BV}

We will apply symplectic cohomology to prove the existence of more than one closed magnetic geodesic with given energy in Section \ref{Subsection An example of a Quasi-Liouville magnetic TT2} and \ref{ss:convexorientable} (see Theorem \ref{t:nonexact} and \ref{t:cmp}). To this purpose, we need to define an additional piece of structure, the \textit{BV-operator}
\[\Delta:CF^*(H)\to CF^{*-1}(H)\]
constructed by Seidel \cite{SeidelBiased}. Following \cite{Abouzaid-survey,Bourgeois-Oancea-S1}, we can describe $\Delta$ as follows.

Let $\nu\mapsto(H^\nu_s,J_s^\nu)$ be a family of continuation pairs depending on a parameter $\nu\in S^1$ in such a way that $(H^\nu_s,J^\nu_s)\equiv (H,J)$ for large $r$, $H^\nu_{-}=H(\cdot,\cdot+\nu)$ and $H^\nu_+=H$, where, as before, we identify $H$ with a small one-periodic time-dependent perturbation of an autonomous Hamiltonian, so that $H^\nu_+$ is obtained from $H$ by shifting the time by $\nu$. The BV-operator is given by
\[
\Delta y=\sum_{u \in \mathcal{M}(x,y;H^\nu_s,J^\nu_s)}\epsilon(u)\, t^{\tau_{\omega}(u)} t^{\mathcal H(x)-\mathcal{H}(y)} x,
\]
where $\mathcal M(x,y,H^\nu_s,J^\nu_s)$ is the moduli space of rigid cylinders solving $F_{H^\nu_s}(u)=J^{\nu}_s\partial_s u$ and going from $x(\cdot+\nu)$ to $y$. The BV-operator is a chain map, namely,
\begin{equation}\label{Equation BV properties}
\partial \Delta+\Delta\partial=0.
\end{equation}
In particular, it preserves the set of cocycles and of coboundaries.
\subsection{Filtration by the radial coordinate} \label{Subsection period}

As the transgression form $\tau_{\omega}$ might not be exact, the Hamiltonian action is multivalued and can not be used to filter the symplectic cohomology. However, as first observed  by Bourgeois-Oancea in \cite[p.654]{Bourgeois-Oancea-Inventiones} and refined later by McLean-Ritter in \cite[Appendix D]{McLean-Ritter}, if the Hamiltonian $H:M\to\R$ is radial and convex on $M^{\mathrm{in}}$, we still get a geometric filtration of $CF^*(H)$ by the radial coordinate, or equivalently by the period of the associated Reeb orbit. This filtration is preserved under $\partial$ and $\Delta$, since both operators count solutions of a small perturbation of Floer's equation $F_H(u)=J\partial_s u$. More precisely, a one-periodic orbit $x$ of $H$ on $M^{\mathrm{out}}$, yields two generators $x_+$ and $x_+$ of $CF^*(H)$ after perturbation. Then,
\begin{equation}\label{e:deltaDelta}
\left\{\begin{aligned}
\partial x_-&=ax_++y_1,\\
\partial x_+&=y_2,
\end{aligned}\right.\qquad \left\{\begin{aligned}
\Delta x_-&=y_3,\\
\Delta x_+&=bx_-+y_4,
\end{aligned}\right.
\end{equation}
where $a,b\in\Z$ and $y_1,y_2,y_3,y_4$ are generated by orbits having $r$-component strictly less than that of $x$. The values of $a$ and $b$ depend on whether $x$ is a good or bad orbit. Let $z$ be the primitive Reeb orbit from which $x$ is obtained by iteration and denote by $k$ the order of iteration. Recall that an orbit is \textit{good} if $\bar{\mu}(x)\equiv\bar{\mu}(z)\ (\!\!\!\!\mod 2)$ and \textit{bad} otherwise. Here $\bar{\mu}$ is the transverse Conley-Zehnder index of an orbit. Results of Bourgeois-Oancea \cite[Proposition 3.9]{Bourgeois-Oancea-MorseBott} (see also \cite[Proposition 2.2]{CFHW})), respectively of Zhao \cite[Equation (6.1)]{Zhao}, gives us the values of $a$ and $b$:
\begin{equation}\label{e:ab}
a=\begin{cases}
0&\text{if $x$ is good},\\
\pm 2&\text{if $x$ is bad},
\end{cases}\qquad b=\begin{cases}
k&\text{if $x$ is good},\\
0&\text{if $x$ is bad}.
\end{cases}
\end{equation}
%%%%%%%%%%%%%%%%%%%%%%%%%%%%%%%%%%%%%%%%%%%%%%%%%%%%%%%%%%%%%%%%%%%%%%%%%
%%%%%%%%%%%%%%%%%%%%%%%%%%%%%%%%%%%%%%%%%%%%%%%%%%%%%%%%%%%%%%%%%%%%%%%%%
\subsection{Twisted symplectic cohomology}
\label{Subsection twisted symplectic cohomology}
%%%%%%%%%%%%%%%%%%%%%%%%%%%%%%%%%%%%%%%%%%%%%%%%%%%%%%%%%%%%%%%%%%%%%%%%%
%%%%%%%%%%%%%%%%%%%%%%%%%%%%%%%%%%%%%%%%%%%%%%%%%%%%%%%%%%%%%%%%%%%%%%%%%
Twisted symplectic cohomology was first constructed in \cite{Ritter1,Ritter2} for Liouville manifolds. We now adapt this to convex manifolds.
Let ${\zeta}\in H^1(\mathcal L(M))$ be a cohomology class represented by a closed 1-form $\eta$ on $\mathcal L(M)$. The twisted group $HF^*(H,J)_\eta$ (and respectively $SH^*(M,\omega,\theta)_{\eta}$) is defined by replacing $\tau_{\omega}$ by $\tau_{\omega}+\eta$ in \eqref{Equation Floer differential} (resp.\eqref{Equation Continuation map twist}). Notice this changes the weights in the count, but not the moduli spaces.
That the twisted differential and twisted continuation maps are well-defined requires the analogues of \eqref{e:weights} with $\tau_\omega+\eta$ in place of $\tau_\omega$. Explicitly, 
abbreviating $\mathcal{M}(x,y)=\mathcal{M}(x,y;H,J)$ (resp.\,$\mathcal{M}(x,y;H_s,J_s)$), for every $C>0$ we need a constant $E_C(x,y;H,J,\eta)>0$ such that
\begin{equation}\label{e:weights-twisted}
\forall\, u\in\mathcal M(x,y),\quad \tau_\omega(u)+\eta(u)\leq C\quad \Longrightarrow\quad E(u)\leq E_C(x,y;H,J,\eta).
\end{equation}
Giving an upper bound $C$ on $\tau_\omega(u)+\eta(u)$ is the same as assuming an upper bound on 
the total exponent $\int u^*\omega+\int u^*\eta+\int H_-(x)\, dt - \int H_+(y)\, dt$ appearing in the twisted differential (for $H_{\pm}=H$) resp.\,twisted continuation map.
By \eqref{e:weights}, the implication \eqref{e:weights-twisted} is equivalent to saying that given $x,y$, the following holds:
\begin{equation}\label{e:implication}
\forall\, C>0,\ \exists\, C'>0,\ \forall\,u\in \mathcal{M}(x,y),\quad\tau_\omega(u)+\eta(u)\leq C\ \ \Rightarrow\ \ \tau_\omega(u)\leq C'.
\end{equation}
Note \eqref{e:implication} can fail in general, e.g. if $|\mathcal{M}(x,y)|=\infty$ and $\eta=-\tau_\omega$.

\begin{definition}\label{Definition SH well defined}
Call $SH^*(M,\omega,\theta)_{\eta}$ \textit{well-defined} if \eqref{e:implication} holds.
\end{definition}

\begin{lemma}\label{Lemma have a c<1 bound}
If there is a constant $c<1$ such that
\[
-\eta(u)\leq c \,E(u) \quad \textrm{ for all}\; \,u\in \mathcal{M}(x,y),
\]
then $SH^*(M,\omega,\theta)_{\eta}$ is well-defined.
\end{lemma}
\begin{proof}
By \eqref{Equation Energy of Floer traj}, $-\eta(u)\leq c \tau_{\omega}(u) + p_{x,y}$ where $p_{x,y}=
c(\mathcal H(x)-\mathcal H(y))$ only depends on the asymptotics (a similar argument holds for continuation solutions). Then,
\[
\tau_{\omega}(u)+\eta(u)\leq C\quad\Rightarrow\quad\tau_{\omega}(u)\leq c\tau_{\omega}(u)+p_{x,y}+C\quad\Rightarrow\quad\tau_{\omega}(u)\leq \frac{p_{x,y}+C}{1-c}.\qedhere\]
\end{proof}

\begin{lemma}\label{Lemma SH twist if change by exact form}
$SH^*(M,\omega,\theta)_{\eta}$ is well-defined if $\eta$ is transgression-compatible (Def.\ref{Definition Intro transgression compatible twist}). In particular, if $\tau(\omega)$ is exact (e.g.~ when $\omega$ is exact), then any twist $\eta$ is allowed.
\end{lemma}
\begin{proof}
If $\tau_{\omega}$ is exact then $C'$ in \eqref{e:implication} is determined a priori by the data $x,y$. For the transgression-compatible case, \eqref{Equation Intro transgression compatible} implies $C'=cC+\mathcal K(y)-\mathcal K(x)$ works.
%%
%%
%$$
%\tau_\omega(u)+\eta(u)\leq C\quad\Longrightarrow\quad \tau_\omega(u)\leq cC+\mathcal K(y)-\mathcal K(x).\qedhere
%$$
%%
\end{proof}
\begin{lemma}\label{Lemma SH twisted change of basis trick}
If $SH^*(M,\omega,\theta)_{\eta}$ is well-defined then for any function $\mathcal{K}: \mathcal{L}M\to \R$, $SH^*(M,\omega,\theta)_{\eta+d\mathcal{K}}$ is well-defined and there is a natural isomorphism $SH^*(M,\omega,\theta)_{\eta} \cong SH^*(M,\omega,\theta)_{\eta+d\mathcal{K}}$ induced by the chain-level change of basis isomorphism sending a $1$-orbit $x$ to $t^{-\mathcal{K}(x)}x$. So we may write $SH^*(M,\omega,\theta)_{{\zeta}}$ for a class ${\zeta}\in H^1(\mathcal{L}M)$. $\qed$
\end{lemma}
%\begin{proof}
%The new weight appearing in the Floer differential is $t^{\mathcal{K}(y)-\mathcal{K}(x)}$.
%\end{proof}

\begin{remark}\label{Remark Twisting variable s}
An alternative approach, is to distinguish two formal variables $t,b$,
\[
\partial y = \sum_{u \in \mathcal{M}(x,y;H,J)}\epsilon(u) b^{\eta(u)}t^{\tau_\omega(u)}t^{\mathcal H(x)-\mathcal H(y)} x.
\]
One can work for example over $\Lambda \hat{\otimes} B$, where $B$ consists of finite sums $\sum c_j b^{m_j}$ where $c_j\in \K$, $m_j\in \R$. The tensor product is completed, meaning $\Lambda \hat{\otimes} B=\{\sum a_i  t^{n_i}:a_i\in B$, $n_i \in \R,$ with $n_i\to \infty\}$. 
%We only need finite sums in $B$ as Gromov compactness applies to moduli spaces of solutions involved with powers of $t$ below a given exponent.
Then $SH^*(M,\omega,\theta)_{{\zeta}}$ always exists. However, for the purposes of proving a deformation theorem like \eqref{Eqn nonexact is twisted}, one would need to \textit{specialise} the twisted group by evaluating $b\mapsto t$, leading again to convergence issues in the Novikov field.
\end{remark}
%%%%%%%%%%%%%%%%%%%%%%%%%%%%%%%%%%%%%%%%%%%%%%%%%%%%%%%%%%%%%%%%%%%%%%%%%
\subsection{Proof of Theorem \ref{t:invariance2}.(1): existence of twisted symplectic cohomology}
\label{Subsection Existence of twisted symplectic cohomology}
%%%%%%%%%%%%%%%%%%%%%%%%%%%%%%%%%%%%%%%%%%%%%%%%%%%%%%%%%%%%%%%%%%%%%%%%%
%
%
%%%%%%%%%%%%%%%%%%%%%%%%%%%%%%%%%%%%%%%%%%%%%%%%%%%%%%%%%%%%%%%%
%%%%%%%%%%%%%%%%%%%%%%%%%%%%%%%%%%%%%%%%%%%%%%%%%%%%%%%%%%%%%%%%
We now restrict ourselves to twisting by classes in $H^1(\mathcal{L}M)$ arising by transgression from classes in $H^2(M,M^{\out})\cong H^2_c(M)$ (if $H^1(\Sigma;\R)=0$, this means any $H^2(M)$ class). Explicitly, we twist by any closed two-form $\beta$ on $M$ exact at infinity. By Lemma \ref{Lemma SH twisted change of basis trick} we may assume $\beta$ is supported in a compact region $N \subset \mathrm{int}(M^{\inn})$.
We will prove that $SH^*(M,\omega,\theta)_{\tau(\beta)}$ is well-defined when $\|\beta\|$ is sufficiently small.

Consider the pairs $(H,J)$ in the construction of $SH^*(M,\omega,\theta)$ satisfying:
\begin{align*}
\mathbf{(H1)}\quad &\forall\,x\in\mathcal{L}(M),\ \ x(S^1)\cap N\neq \emptyset\ \ \Rightarrow\ \ \textstyle \|F_H(x)\| := (\int_0^1 |\partial_t x  - X_H|^2 \, dt)^{1/2}> \delta,\\
\mathbf{(H2)}\quad &\Vert H\Vert_{C^1(N)}\leq C.
\end{align*}

\begin{lemma}\label{Lemma cofinal family of hams}
There exist $\delta,C>0$ admitting a cofinal family $(H_k,J_k)$ of such pairs $(H,J)$ and monotone interpolating homotopies $H_{k,s},J_{k,s}$ belonging to such pairs. 
%We call this data \textit{twisting datum}.
\end{lemma}
\begin{proof}
For small $r\geq 0$, abbreviate $M_r = M^{\inn} \cup (\Sigma \times [0,r])\subset M$ and $M_r^{\out}=M\setminus M_r$, so $M_0=M^{\inn}$. We will use the regions \[
N\subset M_0\subset M_{\epsilon} \subset M_{2\epsilon}\] for a small $\epsilon>0$. Fix a $C^2$-small Morse function $H_0: M \to \R$, such that the only $1$-periodic orbits of $X_{H_0}$ in $M_0$ are critical points of $H_0$, and $H_0=h_0(R)=h_0(e^r)$ is radial and strictly convex on $M_0^{\out}$ of slope $h'(R)$ less than all Reeb periods. 
By composing $H_0$ with an isotopy of $M$ supported in $M_0$, we may assume the critical points $\mathrm{Crit}(H_0)$ lie in $M_0\setminus N$. By construction, the zeros of $F_{H_0}$ on $\mathcal{L}M$ are precisely $\mathrm{Crit}(H_0)$. \red{An Arzel\`{a}-Ascoli argument}\footnote{This relies on the Sobolev embedding $W^{1,2}(S^1,\R^m)\hookrightarrow C^0(S^1,\R^m)$ \cite[Exercise 1.22]{Salamon}.} \red{implies that for $x\in \mathcal{L}M_{\epsilon}$ the value $F_{H_0}(x)$ is small only if $x$ is $C^0$-close to $\mathrm{Crit}(H_0)$.} Thus there is a $\delta_0>0$ such that if $x\in \mathcal{L}M_{\epsilon}$ and $\Vert F_{H_0}(x)\Vert\leq \delta_0$ then $x(S^1)\cap N= \emptyset$.
Define
\[
c:=h_0(e^{\epsilon})-h_0(e^0)>0,\;\quad C:=\|H_0\|_{C^1(M_{\epsilon})}, \;\quad \delta:=\min\{\delta_0,\tfrac{c}{C}\}.
\]
Let $H_k:M\to\R$ equal $H_0$ on $M_{\epsilon}$, and let $H_k$ be radial on $M^{\out}$, with fixed slope $m_k$ on $M_{2\epsilon}^{\out}$ not equal to a Reeb period, such that the slopes $m_k$ strictly increase to infinity as $k\to\infty$ and the linear interpolations $H_{k,s}$ from $H_{k+1}$ to $H_k$ are monotone: $\partial_s H_{k,s}\leq 0$. The $J_k$ can be chosen to be small generic perturbations of a fixed $J$. 

These functions satisfy \textbf{(H2)}. To establish \textbf{(H1)}, it suffices to show $\Vert F_H(x) \Vert > c/C$ for $x:S^1 \to M$ with $x(a)\in N$ and $x(b)\in M_{\epsilon}^{\out}$, for some $a,b\in S^1=\R/\Z$, where $H=H_k$ or $H=H_{k,s}$. We may assume $b>a$ in $\R$ with $|b-a|\leq 1$. By shrinking $[a,b]$ we may assume $x([a,b])\subset M_{\epsilon}$, so the path lies in the region where $H_k=H_{k,s}=H_0$. Abbreviate $H=H_0$, $F=F_{H}$, $X=X_H$ and the restriction $y=x|_{[a,b]}$. By Cauchy-Schwarz:
\[
%\begin{array}{rcl}
%\textstyle
\|F(x)\|\sqrt{b-a} 
%& \geq & 
\geq
\int_{a}^{b} |F(x)(t)|\, dt \geq
\int_{a}^b \frac{|d_{x(t)}{H}|}{\|dH\circ y\|} |\partial_t x - X|\, dt 
%\\[2mm] & \geq & 
\geq
\frac{|\int_a^b\partial_t (H\circ x)\, dt|}{\|dH \circ y\|} \geq \frac{c}{C},
%\end{array}
\]
where we used that $dH(X)=0$. As $|b-a|\leq 1$, we deduce $\|F_H(x)\|\geq c/C$.
\end{proof}

\begin{theorem}[Energy Estimate]\label{Theorem Energy estimate 2 for Floer traj}
For $\beta,N,\delta,C$ as above,
\[
|\tau_\beta(u)|\leq \Vert\beta\Vert_{C^0(N)}\cdot (1+\tfrac{C}{\delta})\cdot E(u),
\]
where $u$ is any Floer trajectory or continuation solution for the data from Lemma \ref{Lemma cofinal family of hams}.
\end{theorem}
\begin{proof}
Let $u$ be a Floer trajectory for the given data $(H,J)$ (the proof for continuation maps is analogous). Denote $u_s = u|_{\{s\}\times S^1}$ for $s\in \R$, then $u_s\in \mathcal{L}M_{2\epsilon}$ as the maximum principle applies on $M_{2\epsilon}^{\out}$ by the construction of the data in Lemma \ref{Lemma cofinal family of hams}.
Let
\begin{equation}\label{e:mathcalS}
\mathcal{S}_u:=\Big\{s\in\R\ \Big|\ u_s \notin \mathcal{L}(M_{2\epsilon}\setminus N)\Big\}\subset \R
\end{equation}
be the values $s\in\R$ for which $u(s,t)\in N$ for some $t\in S^1$. Using \textbf{(H1)} and the definition $ E(u)=\int_{\R}\Vert F_H(u_s)\Vert^2\, ds$, Chebyshev's inequality implies
\begin{equation}\label{Equation estimate of mathcalSu in terms of energy}
|\mathcal{S}_u|\leq \frac{E(u)}{\delta^2},
\end{equation}
where $|\mathcal{S}_u|$ is the Lebesgue measure of $\mathcal{S}_u$. We note that
\[
\forall\,(s,t)\in\R\times S^1,\quad\beta|_{u(s,t)}\neq 0\;\Longrightarrow\; s\in \mathcal S_u, \; u(s,t)\in N \,\textrm{ and }\, X_H|_{u(s,t)}=X_{H_0}|_{u(s,t)},
\]
since $H=H_0$ on $M_{\epsilon}$ by construction.
Using \eqref{Equation estimate of mathcalSu in terms of energy} and Cauchy-Schwarz:
\begin{align*}
|\tau_\beta(u)|= \Big|\int_{\mathcal \R\times S^1} u^*\beta\Big|&=\Big|\int_{\mathcal S_u\times S^1} \beta (\partial_s u, J\partial_s u + X_H )\, ds\, dt\Big|
\\
&\leq
\Vert\beta\Vert\Big( \int_{\mathcal S_u\times S^1} |\partial_s u|^2\, ds\, dt 
+ \Vert dH_0\Vert_{C^1(N)}\, \int_{\mathcal S_u\times S^1}  |\partial_s u|\, ds\, dt\Big)\\
&\leq \Vert\beta\Vert \Big(E(u)+C\sqrt{E(u)}\cdot \sqrt{|\mathcal{S}_u|}\Big)\\
&\leq \Vert\beta\Vert(1+\tfrac{C}{\delta})E(u).\qedhere
\end{align*}
\end{proof}

\begin{corollary}\label{c:SH}
If a class in $H^2_c(M)\cong H^2(M,M^{\out})$ has a representative $\beta$ compactly supported in $N\subset \mathrm{int}(M^{\inn})$ with
\[
\Vert\beta\Vert_{C^0(N)}<(1+\tfrac{C}{\delta})^{-1},
\]
then twisted symplectic cohomology $SH^*(M,\omega,\theta)_{\tau(\beta)} = \varinjlim HF^*(H_k,J_k)$ is well-defined using the data from Lemma \ref{Lemma cofinal family of hams}. So for any class $\beta\in H^2_c(M)\cong H^2(M,M^{\out})$, the group $SH^*(M,\omega,\theta)_{\tau(s\beta)}$ is defined for all sufficiently small $s\geq 0$.
\end{corollary}
\begin{proof}
Let $c:=\Vert\beta\Vert (1+\tfrac{C}{\delta})<1$. By Theorem \ref{Theorem Energy estimate 2 for Floer traj}, the claim follows by Lemma \ref{Lemma have a c<1 bound}. 
% 
%By Theorem \ref{Theorem Energy estimate 2 for Floer traj} and equation \eqref{Equation Energy of Floer traj}, 
%$$
%E(u)=(\tau_\omega+\tau_\beta)(u)-\tau_\beta(u)+\mathcal H(y)-\mathcal H(x)
%\leq (\tau_\omega+\tau_\beta)(u)+c E(u)+\mathcal H(y)-\mathcal H(x).
%$$
%Therefore, twisted differentials and the continuation maps are well-defined, thanks to \eqref{e:weights}, as the above inequality can be rewritten in the form
%\[
%E(u)\leq  \frac{1}{1-c}\Big((\tau_\omega+\tau_\beta)(u)+\mathcal H(x)-\mathcal H(y)\Big).\qedhere
%\]
\end{proof}
\begin{remark}\label{Remark dependent on cofinal family}
It is not yet clear whether $SH^*(M,\omega,\theta)_{\tau(\beta)}$ is independent of the chosen cofinal family $H_k$, as we cannot a priori control $\delta$ for general monotone homotopies between two given cofinal families.
Nevertheless independence on this choice follows a posteriori from the isomorphism with $SH^*(M,\omega+\beta,\theta)$ in Theorem \ref{t:invariance2}$.(2)$.
\end{remark}
As part of the direct limit, we have the canonical $\Lambda$-linear homomorphism \eqref{Equation c* map twisted}
since $HF^*(H_0,J_0;\omega,\theta)\cong QH^*(M,\omega)$ (as a vector space, $QH^*(M,\omega)=H^*(M)\otimes \Lambda$).
%
%%%%%%%%%%%%%%%%%%%%%%%%%%%%%%%%%%%%%
\subsection{Product structure}\label{Subsection product structure}
%%%%%%%%%%%%%%%%%%%%%%%%%%%%%%%%%%%%%
To conclude the proof of Theorem \ref{t:invariance2}.(1) we need to explain why $SH^*(M,\omega,\theta)_{\tau(\beta)}$ admits a unital ring structure given by the pair-of-pants product, under the assumptions in Corollary \ref{c:SH}. We refer to \cite{Ritter3} for the detailed construction of the product. This uses an auxiliary $1$-form $\gamma$ defined on the pair-of-pants $P$, satisfying $d\gamma \leq 0$ (this ensures the maximum principle), and $\gamma$ is equal to a positive constant multiple of $dt$ near each end. We may choose $dt$ at the two positive ends, and $2\,dt$ at the negative end. 
% By Stokes' theorem get 0 \leq \int u*gamma = \int dt + \int dt - \int 2dt = 1+1-2
% so need at least factor 2 at the negative end
The product involves the moduli space $\mathcal{M}(x;y,z;H,J)$ of rigid solutions $u:P\to M$ of the equation $(du-X\otimes \gamma)^{0,1}=0$ asymptotic to $x$ at the negative end and $y,z$ at the positive ends, where $X=X_H$. We obtain a $\Lambda$-linear map $HF^*(H)\otimes HF^*(H) \to HF^*(2H)$ which on $1$-orbits is: 
\begin{equation}\label{Equation pop product}
y\otimes z\quad \longmapsto \sum_{u\in\mathcal M(x;y,z;H,J,\gamma)}t^{\tau_{\omega}(u)}t^{2\mathcal H(x)-\mathcal H(y)-\mathcal H(z)}x.
\end{equation}
Here we abuse notation slightly, $\tau_{\omega}(u)=\int_P u^*\omega=\tau_{\omega}(P_-)+\tau_{\omega}(P_{+,1}) + \tau_{\omega}(P_{+,2})$ where we decompose $P=P_-\cup P_{+,1} \cup P_{+,2}$ as the union of three cylinders (whose images via $u$ yield three $1$-chains in $\mathcal{L}M$) asymptotic to the three ends, such that the positive boundary of $P_-$ is the figure eight-loop consisting of the two negative boundaries of $P_{+,1}$, $P_{+,2}$. The choice of decomposition will not affect the weights (here it is crucial that we are twisting by a class in $H^1(\mathcal{L}M)$ that arises as the transgression of a class in $H^2(M)$). The exponent of $t$ in \eqref{Equation pop product} is precisely the \textit{topological energy}
\[
E_{top}(u) := \int_P u^*\omega - d(u^*H \wedge \gamma),
\]
which is a homotopy invariant that bounds from above the (geometric) energy 
\[
E(u):=\tfrac{1}{2}\int_P \|du-X\otimes \gamma\|^2\, \mathrm{vol}_P=\int_P u^*\omega - d(u^*H ) \wedge \gamma,
\]
since $d\gamma\leq 0$. Explicitly: 
$
E(u)\leq \tau_\omega(u)+2\mathcal H(x)-\mathcal H(y)-\mathcal H(z).
$
This ensures that the above map is well-defined. By considering continuation maps as in \cite{Ritter3} a direct limit of these maps defines a $\Lambda$-bilinear homomorphism $SH^*(M;\omega,\theta)^{\otimes 2}\to SH^*(M;\omega,\theta)$ called the pair-of-pants product. The same argument as in \cite{Ritter3} shows that the element $c^*(1)\in SH^*(M;\omega,\theta)$ is a unit, so the map in \eqref{Equation c* map twisted} is a unital $\Lambda$-algebra homomorphism, using the quantum product on $QH^*(M,\omega)$. 
 
Fix a compact subregion $P'\subset P$ independent of $u$ such that $P\setminus P'$ is the disjoint union of the three cylindrical ends. Abbreviate $A:=\mathrm{Area}(P')$. We claim that
\[
E(u)\geq A\delta^2\quad\Longrightarrow\quad |\tau_\beta(u)|\leq \Vert\beta\Vert_{C^0(N)}C'E(u),
\]
for some constant $C'>0$ independent of $u$. Let $u_1,u_2,u_3$ be the restriction of $u$ to the three cylindrical ends with corresponding sets $\mathcal S_{u_1},\mathcal S_{u_2},\mathcal S_{u_3}$, as in \eqref{e:mathcalS}. Let
\[
P_u:=P'\cup (\mathcal S_{u_1}\times S^1)\cup (S_{u_2}\times S^1)\cup (S_{u_3}\times S^1).
\]
Assuming $E(u)\geq A \delta^2$, we obtain the following generalisation of \eqref{Equation estimate of mathcalSu in terms of energy}:
\[
\mathrm{Area}(P_u)\leq A+\tfrac{1}{\delta^2}E(u)\leq \tfrac{2}{\delta^2}E(u).
\]
We estimate $-\tau_\beta(u)$ from above as in Theorem \ref{Theorem Energy estimate 2 for Floer traj} substituting $\mathcal S_u\times S^1$ with $P_u$: 
\begin{align*}
\Big|\int_{P_u} u^*\beta\Big| &= \Big|\int_{P_u} \beta\circ (du-X\otimes\gamma)^{\otimes 2} + \beta(du-X\otimes \gamma,X)\,\gamma+\beta(X,du-X\otimes \gamma)\,\gamma\Big|\\
&\leq \Vert\beta\Vert_{C^0(N)} \Big(E(u)+2C\Vert \gamma\Vert_{C^0(P)}\sqrt{E(u)}\cdot 
\tfrac{1}{\delta}\sqrt{2E(u)}\Big),
%
% looks odd to have two such similar terms at the end,
% but you need a 2-form, in particular antisymmetric.
% The form is function(s,t) ds wedge dt
% the get the function you plug in (d_s,d_t) in that order.
% By anti-symmetry, get same answer if you want
% function(s,t) dt wedge ds
% if plug in (d_t,d_s).
% So strictly speaking one cannot write 2 beta(du-Xgamma,X)
% because that would not be a 2-form on the surface,
% you get two different answers by the above plugging-in methods.
%
\end{align*}
where we used that $\beta(X,X)=0$. This proves the claim. 
Corollary \ref{c:SH} together with the construction of the product conclude the proof of Theorem \ref{t:invariance2}.(1).\hfill\qed
%%%%%%%%%%%%%%%%%%%%%%%%%%%%%%%%%%%%%%%%%%%%%%%%%%%%%%%%%%%%%%%%%%%%%%%%%
%%%%%%%%%%%%%%%%%%%%%%%%%%%%%%%%%%%%%%%%%%%%%%%%%%%%%%%%%%%%%%%%%%%%%%%%%
%%%%%%%%%%%%%%%%%%%%%%%%%%%%%%%%%%%%%%%%%%%%%%%%%%%%%%%%%%%%%%%%%%%%%%%%%
%%%%%%%%%%%%%%%%%%%%%%%%%%%%%%%%%%%%%%%%%%%%%%%%%%%%%%%%%%%%%%%%%%%%%%%%%
%%%%%%%%%%%%%%%%%%%%%%%%%%%%%%%%%%%%%%%%%%%%%%%%%%%%%%%%%%%%%%%%%%%%%%%%%
\section{Invariance of Symplectic Cohomology}
\label{Section Compact deformation invariance}
%%%%%%%%%%%%%%%%%%%%%%%%%%%%%%%%%%%%%%%%%%%%%%%%%%%%%%%%%%%%%%%%%%%%%%%%%
%%%%%%%%%%%%%%%%%%%%%%%%%%%%%%%%%%%%%%%%%%%%%%%%%%%%%%%%%%%%%%%%%%%%%%%%%

%%%%%%%%%%%%%%%%%%%%%%%%%%%%%%%%%%%%%%%%%%%%%%%%%%%%%%%%%%%%%%%%%%%%%%%%%
\subsection{Small perturbations: proof of Theorem \ref{t:invariance2}.(2)}
\label{Subsection Invariance under small deformations}
%%%%%%%%%%%%%%%%%%%%%%%%%%%%%%%%%%%%%%%%%%%%%%%%%%%%%%%%%%%%%%%%%%%%%%%%%
%%%%%%%%%%%%%%%%%%%%%%%%%%%%%%%%%%%%%%%%%%%%%%%%%%%%%%%%%%%%%%%%%%%%%%%%%
To simplify notation, we may assume that $M=M_a=D\setminus\partial D$, where $D$ is some convex domain with a collar $C$ and $M^\inn=\overline{D\setminus C}$.
By Proposition \ref{p:quantitative}, for $(\mu,\lambda)\in \Omega^2(D,C)$ which is closed and $C^1$-small, we obtain a convex domain $(D,\omega+\mu,\theta+\lambda)$ and a closed form $\beta\in\Omega^2_c(M)$ (in a class corresponding to $[\mu,\lambda]$) with $\|\beta\|_{C^0(M)}\leq K \|(\mu,\lambda)\|_{C^1(D,C)}$, and admitting an isomorphism
\[
(M,\omega+\beta,\theta)^{\wedge}\stackrel{_{\cong}}{\longrightarrow} (M,\omega+\mu,\theta+\lambda)^{\wedge}.
\]
By Theorem \ref{Theorem invariance under iso of SH}, 
\[
SH^*(M,\omega+\beta,\theta)\cong SH^*(M, \omega+\mu,\theta+\lambda).
\]
To prove Theorem \ref{t:invariance2}.(2) it remains to show $SH^*(M,\omega+\beta,\theta)\cong SH^*(M,\omega,\theta)_{\tau(\beta)}$, for $\Vert \beta\Vert_{C^0(M)}$ small enough, where the right-hand side is defined by Theorem \ref{t:invariance2}.(1). The twisting is needed so that the groups on both sides have the same system of local coefficients. To build the isomorphism, it suffices to construct a sequence of commutative diagrams for the twisted Floer cohomologies of $M$:
\begin{equation}\label{Equation commutative diagram for continuations}
\xymatrix@R=14pt{ 
HF^*(\omega,\theta;H_{k+1},J_{0,k+1})_{\tau(\beta)}
\ar@{<-}^-{\psi_{k+1}}[r] 
&  
HF^*(\omega+\beta,\theta;H_{k+1},J_{1,k+1})
\\
HF^*(\omega,\theta;H_{k},J_{0,k})_{\tau(\beta)}
\ar@{<-}[r]^-{\psi_{k}} 
\ar@{->}^-{\varphi_{0,k}}[u]
&
HF^*(\omega+\beta,\theta;H_{k},J_{1,k})
\ar@{->}_{\varphi_{1,k}}[u]
},
\end{equation}
where $(H_k,J_{0,k})$ and the continuation maps $\varphi_{0,k}$ are defined by data as in Lemma \ref{Lemma cofinal family of hams}, whilst $(H_k,J_{1,k})$ and the continuation maps $\varphi_{1,k}$ are defined by data used in the construction of $SH^*(M,\omega+\beta,\theta)$.
Thus, the direct limits over the vertical maps respectively define $SH^*(M,\omega,\theta)_{\tau(\beta)}$ and $SH^*(M,\omega+\beta,\theta)$.
We now construct the horizontal maps $\psi_k$. To simplify the notation, we will drop all subscripts $k$.
The map $\psi$ is a continuation map, where for $s\in \R$ we vary the pair
\[
(\omega_s:=\omega+\rho(s)\beta,\,J_s)
\]
but keep the Hamiltonian fixed. Here $\rho:\R\to[0,1]$ is a function and $J_s$ is an $\omega_s$-compatible almost complex structure of contact type at infinity, satisfying $\rho_s=0$ and $J_s=J_0$ for $s\leq 0$; and $\rho_s = 1$, $J_s=J_1$ for $s\geq 1$. 

Let $\mathcal{M}(x,y;\omega_s,H,J_s)$ be the set of rigid solutions $u: \R \times S^1 \to M$ from $x$ to $y$ of the equation $F^s_{H}(u)=J_s\partial_s u$, where $F^s_H(x):=\partial_t x - X_s$ and $X_s$ is the Hamiltonian vector field of $H$ with respect to $\omega_s$, so $\omega_s(\cdot,X_s)=dH$. Transversality is standard, since we allow $J_s$ to vary and we assumed $(M,\omega)$ to be weakly monotone.
At the chain level, $\psi: CF^*(\omega+\beta,H,J_{1})
\longrightarrow CF(\omega,H,J_{0})_{\tau(\beta)} $ is defined on $1$-orbits as follows:
\begin{equation}\label{Equation psi map definition}
\psi(y)  \;= \sum_{u \in \mathcal M(x,y;\omega_s,H,J_s)} \epsilon(u)t^{(\tau_\omega+\tau_\beta)(u)}t^{\mathcal H(x)-\mathcal H(y)} x.
\end{equation}
%
%
%Here $\tau_{\omega}(u)=\int u^*\omega$ depends on $\omega$, not $\omega_s$: the $1$-form for the local system we use on $\mathcal{L}M$ is $\tau(\omega_s)+\tau((1-\rho(s))\beta)-d\mathcal{H} = \tau(\omega)+\tau(\beta) - d\mathcal{H}$, which is independent of $s$.

Provided the $\psi$ maps are well-defined, standard Floer theory arguments imply:
\begin{enumerate}
	\item Diagram \eqref{Equation commutative diagram for continuations} commutes at the chain level up to chain homotopy.
	\item The $\psi$ maps are isomorphisms, their inverse being the continuation maps $\bar\psi$ obtained from the reverse deformation $(\bar{\omega}_s,\bar J_s):=(\omega_{-s},J_{-s})$. Indeed, $\bar{\psi}\psi$ and $\psi\bar{\psi}$ correspond (up to chain homotopy) to a deformation where the symplectic form and the almost complex structure are both held constant, and thus the map is the identity for dimension reasons (moduli spaces are never rigid, due to an $s$-translation symmetry, unless Floer continuation solutions are constant).
	\item The isomorphism $\psi_\infty=\varinjlim \psi_k:SH^*(M,\omega,\theta)_{\tau(\beta)}\to SH^*(M,\omega+\beta,\theta)$ is compatible with the unital product structure (i.e.\,the $\psi_k$ maps fit into commutative diagrams similar to those used in \cite{Ritter3} to construct the product). This requires an energy estimate for pairs-of-pants, but just as in Section \ref{Subsection product structure} this estimate will follow once one has the energy estimate for Floer cylinders (Theorem \ref{Theorem Energy estimate 1}). This yields Theorem \ref{t:invariance2}.(2).
\end{enumerate} 
We now prove that $\psi$ is well-defined, if $\Vert\beta\Vert_{C^0(N)}$ is small. As $\omega_s$ only varies on $M^{\inn}$, we can keep $J_s$ independent of $s$ on $M^{\out}$ and the maximum principle applies on $M^{\out}$. So we only need to bound the energy $
E(u)=\int_{\R\times S^1}|\partial_s u|^2_sdsdt,
$
where $|\cdot|_s$ is the norm associated to the Riemannian metric $\omega_s(\cdot,J_s\cdot)$.
Thus
\[
E(u)=\int_{\R\times S^1}\!\!\omega_s(\partial_su,\partial_tu-X_s)\, ds\, dt=(\tau_\omega+\tau_\beta)(u)+\int_{\R\times S^1}\!\!(\rho(s)-1)\, u^*\beta+\mathcal H(x)-\mathcal H(y).
\]
By Lemma \ref{Lemma have a c<1 bound}, the following theorem implies that $\psi$ is well-defined.
\begin{theorem}\label{Theorem Energy estimate 1}
For $\delta,C$ as above, there is a constant $c''>0$ such that for any sufficiently small closed $2$-form $\beta$ compactly supported on $M$,
\[
\Big|\int_{\R\times S^1}(\rho(s)-1)\,u^*\beta\Big|\ \leq \ \int_{\R\times S^1}|u^*\beta| \ \leq \ \Vert\beta\Vert_{C^0(M)}\cdot c''\Big(1+\frac{C}{\delta}\Big)\cdot E(u),
\]
for all $u\in \mathcal{M}(x,y;\omega_s,H,J_s)$.
\end{theorem}
\begin{proof}
To bound $\int |u^*\beta|$, we argue as in the proof of Theorem \ref{Theorem Energy estimate 2 for Floer traj} except we now work with norms $|\cdot|_s$ depending on $s$. But since $\omega$ and $\omega_s$ differ only on a compact set, these norms are equivalent. So there is a constant $c''>0$ such that $\frac{1}{c''}|\cdot|_s\leq|\cdot|\leq c''|\cdot|_s$ for all $s\in\R$. 
With this observation, the argument in Theorem \ref{Theorem Energy estimate 2 for Floer traj} goes through.
\end{proof}
%Putting together all the results of this subsection, we get that if $\Vert(\beta,\lambda)\Vert_{C^1(N,N^\out)}$ is small enough, then
%\[
%SH^*(N, \omega+\beta,\theta+\lambda)\cong SH^*(M,\omega+\beta,\theta)\cong SH^*(M,\omega,\theta)_{\tau(\beta)}= SH^*(M,\omega,\theta)_{\tau(\beta)}.
%\]
%This proves the second part of Theorem \ref{t:invariance2}.\hfill\qed
\begin{remark}[Technical Remark about Gromov Compactness]
 The energy estimate of Theorem \ref{Theorem Energy estimate 1} is sufficient for the standard arguments of Gromov compactness to go through \cite[Thm 3.3]{Hofer-Salamon}. Indeed, the standard removal of singularities argument (e.g. see McDuff-Salamon \cite{McDuff-Salamon2}) involves considering bubbling that occurs at a specific value of $s$ when energy concentrates, and that argument applies in our setup because our form $\omega_s$ is closed. Moreover, in our argument we need a uniform $\hbar$-bound (the minimal energy represented by a non-constant $J$-holomorphic sphere) that works for $\omega_s$ for all $s\in \R$, which is crucial for Hofer-Salamon's argument \cite[Theorem 3.3]{Hofer-Salamon} to apply. A clean approach would be to separately show that $\hbar$ varies continuously in the metric $\omega_s(\cdot,J_s\cdot)$. A simpler but weaker argument goes as follows. We need to rule out the possibility of the vanishing of the infimum of $\hbar_s$ (the optimal $\hbar$-value for $J_s$), taking the infimum over the compact interval $C$ of values of $s$ for which $\omega_s$ is $s$-dependent. If this infimum were zero, it would imply the existence of a sequence $u_n$ of non-constant $J_{s_n}$-holomorphic spheres such that the energy $E(u_n)$ converges to zero. By passing to a subsequence we may assume $s_n$ converges to some value $s^*\in C$. The usual Gromov compactness argument then says that $u_n$ will converge to a $J_{s^*}$-holomorphic sphere, possibly with a bunch of bubbles, if the energy concentrates at certain points. Part of the proof is that the energy of this limit curve is the limit of the energies $E(u_n)$, if one remembers to take into account all the bubbles arising in the limit curve. In our case, this would imply that the limit curve has zero energy, so the limit is a point with no bubbles. By continuity this would imply that the $u_n$ eventually lie inside a contractible neighbourhood of that point, and therefore these spheres $u_n$ are homologically trivial, which in turn implies that their energy is zero, and thus the $u_n$ are constant for large $n$. Contradiction.
\end{remark}
%%%%%%%%%%%%%%%%%%%%%%%%%%%%%%%%%%%%%%%%%%%%%%%%%%%%%%%%%%%%%%%%%%%%%%%%%
\subsection{Long deformations: proof of Corollary \ref{t:invariance3}}
\label{Subsection Transgression invariant compact deformations}
%%%%%%%%%%%%%%%%%%%%%%%%%%%%%%%%%%%%%%%%%%%%%%%%%%%%%%%%%%%%%%%%%%%%%%%%%
%%%%%%%%%%%%%%%%%%%%%%%%%%%%%%%%%%%%%%%%%%%%%%%%%%%%%%%%%%%%%%%%%%%%%%%%%
Let $(M,\omega_s,\theta_s;{\zeta}_s)$ be a transgression invariant family of convex manifolds (Definition \ref{Definition Intro transgression compatible twist}). Thus,
\[
{\zeta}_s-{\zeta}_{s'}=\tau(\omega_{s'})-\tau(\omega_s)\in H^1(\mathcal{L}M),\qquad\forall\,s,s'\in [0,1].
\]
By Theorem \ref{t:invariance2}.(2), given any $s\in[0,1]$, there is a relatively open interval $I_s$ such that $s\in I_s\subset [0,1]$ and for all $s'\in I_s$ we have
\[
SH^*(M,\omega_s,\theta_s)_{{\zeta}_s-{\zeta}_{s'}}\cong SH^*(M,\omega_s,\theta_s)_{\tau(\omega_{s}'-\omega_s)}\cong SH^*(M,\omega_{s'},\theta_{s'}).
\]
By Lemma \ref{Lemma SH twist if change by exact form}, we can twist the above isomorphism by ${\zeta}_{s'}$:
\begin{equation}\label{e:isomo}
SH^*(M,\omega_s,\theta_s)_{{\zeta}_s}\cong SH^*(M,\omega_{s'},\theta_{s'})_{{\zeta}_{s'}},\qquad \forall\, s'\in I_s.
\end{equation}
As $[0,1]$ is compact, there is a sequence 
$0=s_0$, $\ldots$, $s_m=1$ with $s_{i+1}\in I_{s_i}$. Corollary \ref{t:invariance3} follows by composing the isomorphisms in \eqref{e:isomo} for $s=s_i$ and $s'=s_{i+1}$.\hfill\qed
%
%%%%%%%%%%%%%%%%%%%%%%%%%%%%%%%%%
\section{Twisted cotangent bundles of surfaces}
\label{Section Application: Twisted cotangent bundles of surfaces}
%%%%%%%%%%%%%%%%%%%%%%%%%%%%%%%%%

%%%%%%%%%%%%%%%%%%%%%%%%%%%%%%%%%
\subsection{Basic notation}
\label{Subsection Preliminary notation for cotangent bundles}
%%%%%%%%%%%%%%%%%%%%%%%%%%%%%%%%%
We review some background in the following two sections, but for the sake of brevity we refer the reader to \cite{CMP,GinzburgGurel,Benedetti} for a more extensive survey and for references on the topic of twisted cotangent bundles. 

Let $(N,g)$ be a closed Riemannian manifold. 
Let $\pi:T^*N\to N$ be the footpoint projection and let $\theta=p\,dq$ be the canonical $1$-form (so $\theta_{(q,p)}=p\circ d\pi$).
%
% Added this to avoid confusion later.
%
We identify $T^*N$ and $TN$ via the musical isomorphism 
\[
\flat:T_q N \rightarrow T_q^*N,\quad v \mapsto p=g_q(v,\cdot).
\]
%
%
%the isomorphism induced by $g$,
%%
%%
%\[
%\flat: T_q N \rightarrow T_q^*N, \quad v \mapsto p=g_q(v,\cdot),
%\]
%%
%%
For example, we have $\theta_{(q,v)}=g_q(v,d\pi\cdot)$. Write $g$ also for the dual metric on $T^*N$ and denote all norms by $|\cdot|$. 
The disc and sphere bundle of radius $r$ are
\[D_{r}^g=\{(q,v)\in TN: |v|\leq r\} \quad\qquad \Sigma_{r}^g=\partial D_{r}^g = \{(q,v)\in TN: |v|= r\}.
\]
The use of the letter $r$ here is for notational convenience and is not to be understood as a radial coordinate in a conical parametrization of a convex manifold as in \eqref{Equation Intro j}.

The connection determines a splitting:
\begin{equation}\label{Equation splitting TT*N}
T_{(q,p)}(TN)\cong T_q N \oplus T_q N, \quad \partial_t (q,v) \mapsto (\partial_t q,\nabla_{\!t}\, v)
%
% don't confuse with \nabla_{\!\partial_tq}\, p)
% consider q(t)=constant, but let p vary, would get (0,0)
% Would need to put both: \nabla_{\!(\partial_tq,\partial_t p)}\, p
%
\end{equation}
so that the first component is the map $\xi\mapsto d\pi\cdot\xi$ and $\nabla$ denotes the Levi-Civita connection.
Via \eqref{Equation splitting TT*N}, on $TN$ we get a Riemannian metric $g\oplus g$ and an almost complex structure $J_g$ compatible with the symplectic form $d\theta$, where 
\begin{equation}\label{e:dtheta}
J_g=\begin{pmatrix} 0 & -\mathrm{id} \\ \mathrm{id} & 0\end{pmatrix},\qquad d\theta=\begin{pmatrix} 0 & -g \\ g& 0\end{pmatrix}.
\end{equation}
Applying the inverse of the map in \eqref{Equation splitting TT*N} to $T_qN\oplus 0$ and $0\oplus T_qN$, we get the horizontal distribution $T^{\mathrm{hor}}_{(q,v)}TN$ and the vertical distribution $T^{\mathrm{vert}}_{(q,v)}TN$ on $T(TN)$:
\[
w\mapsto w^h\in T^{\mathrm{hor}}_{(q,v)}TN,\qquad w\mapsto w^\nu\in T^{\mathrm{vert}}_{(q,v)}TN,\qquad \forall\, q\in N,\ \ v,w\in T_qN.
\]
In particular, $T^{\mathrm{vert}}TN=\ker d\pi$. The tautological horizontal and vertical lifts yield two vector fields on $TN$:
\[
X_{(q,v)}=v^h \qquad \qquad Y_{(q,v)} = v^{\nu}.
\]
We recall that $X$ is the geodesic vector field of $g$, namely the Hamiltonian vector field for $(q,v)\mapsto\frac{1}{2}|v|^2$ using $d\theta$, and $Y$ is the Liouville vector field for $(d\theta,\theta)$, so $Y=v\,\partial_v$ in local coordinates. We will later use that for 1-forms $\beta$ on $N$, $\pi^*\beta(X)_{(q,v)}=\beta(v)$.

%%%%%%%%%%%%%%%%%%%%%%%%%%%%%%%%%
\subsection{Twisted cotangent bundles}
\label{Subsection Basic notation for twisted cotangent bundles}
%%%%%%%%%%%%%%%%%%%%%%%%%%%%%%%%%

Let $(N,g)$ be a closed Riemannian manifold of dimension $n>1$. Let $\sigma\in\Omega^2(N)$ be a closed $2$-form, called \textit{magnetic form}. The Lorentz force $\mathbb{Y}:TN\to TN$ is the bundle map determined by
\[
g_q(\mathbb{Y}_q(u),v)=\sigma_q(u,v),\qquad \forall\,q\in N,\ u,v\in T_qN.
\]

A smooth curve $\gamma:I\to N$ is a \textit{magnetic geodesic}, if it satisfies
\begin{equation}\label{e:mg}
\nabla_{\gamma'}\gamma'=\mathbb{Y}_\gamma(\gamma'),
\end{equation}
where $\nabla$ is the Levi-Civita connection for $g$. From the equation, it follows that $\gamma$ has constant speed $r:=|\gamma'|$. If we reparametrise $\gamma$ by arc-length and denote by $\dot{\gamma}$ the derivative of $\gamma$ with respect to this new parameter, \eqref{e:mg} becomes
\begin{equation}\label{e:mg2}
\nabla_{\dot\gamma}\dot\gamma=\frac{1}{r}\mathbb{Y}_\gamma(\dot\gamma).
\end{equation}
Closed magnetic geodesics $\gamma:\R/T\Z\to N$ with speed $r$ are exactly the critical points of the possibly multi-valued free-period action functional $\mathbb S_r$ defined on the space of all free loops of any period:
\[
\mathbb S_r(\gamma)=\int_0^T\frac{1}{2}\Big(|\gamma'(t)|^2+r^2\Big)dt-\int_{[0,1]\times\R/T\Z}\hat\gamma^*\sigma,
\]
where $\hat{\gamma}:[0,1]\times \R/T\Z$ is a connecting cylinder to a fixed reference loop in the free-homotopy class of $\gamma$.

Consider the twisted tangent bundle $(TN,\omega)$, where
\begin{equation}\label{Equation Magnetic Ct Bdle form omega}
\omega:=d\theta-\pi^*\sigma,
\end{equation}
which is weakly monotone as $c_1(TN,\omega)=0$. We now interpret magnetic geodesics as flow lines (up to reparametrization) for the Hamiltonian given by
\begin{equation}\label{Equation Ham used for magn geod}
\rho:TN\to\R,\qquad  \rho(q,v):= |v|.
\end{equation}
Let $W$ be the vertical vector field determined by the Lorentz force:
\[
W_{(q,v)}=(\mathbb{Y}(v))^{\nu},\qquad  \forall\,(q,v)\in TN.
\]

\begin{lemma}\label{Lemma Magnetic geodesic vector field}
The Hamiltonian vector field of $\rho$ with respect to $\omega$ is
\begin{equation}\label{Equation Xrho magn geod}
X_{\rho}=\tfrac{1}{\rho}(X+ W).
\end{equation}
Its flow lines in $\Sigma_{r}^g$ are the curves $(\gamma,r\dot{\gamma})$, where $\gamma$ is any solution of \eqref{e:mg2} parametrised by arc-length, and these are integral curves for the distribution $\ker \omega|_{\Sigma_{r}}$.
\end{lemma}
\begin{proof}
As $d\theta(\cdot,X)=d(\frac{1}{2}{\rho}^2)={\rho}\,d{\rho}$ and $\pi^*\sigma(\cdot,W)=0$ (as $W$ is vertical),
% X is ordinary geod field
\eqref{Equation Xrho magn geod} is equivalent to $d\theta(\cdot,W)=\pi^*\sigma(\cdot,X)$. 
Using \eqref{e:dtheta},
$
d\theta(T^{\mathrm{vert}}TN,W) = 0
$
since $W$ is also vertical and, for $w^h\in T_q^{\mathrm{hor}}N$, we deduce the required equality:
\[
d\theta(w^h,W)=-g(w,\mathbb{Y}(v)) = -\sigma(v,w)=\pi^*\sigma(w^{h},X).
\]
Abbreviate $x=(\gamma(s),r\dot{\gamma}(s))\in TN$.
Using \eqref{e:mg2}, $\frac{1}{r}W_x=\frac{1}{r}\mathbb{Y}(r\dot{\gamma})^{\nu}=r(\nabla_{\dot{\gamma}}\dot{\gamma})^\nu$. Thus, $\partial_s (\gamma,r\dot{\gamma})=\dot{\gamma}^h+ r(\nabla_{\dot{\gamma}}  \dot{\gamma})^\nu=\frac{1}{r}X_x+\frac{1}{r} W_x$.
The final claim is immediate since $\omega(\cdot,X_{\rho})=d\rho=0$ as $\rho$ is constant on $\Sigma_{\rho}$.
\end{proof}

Let us now assume that $N$ is an oriented surface. Let $\jmath:TN\to TN$ be fibrewise rotation by $\frac{\pi}{2}$, and $\mu$ the Riemannian area form. 
Then $(X,Y,H,V)$ is a positively oriented frame with respect to $d\theta\wedge d\theta$, where
\[
H_{(q,v)}=(\jmath\, v)^{h}, \qquad\qquad V_{(q,v)}=(\jmath\, v)^{\nu}.
\]
Here $V$ is generated by the fibrewise rotation $e^{it}:(q,v)\mapsto (q,e^{it}v)$. Following \cite{Gudmundsson-Kappos}, one verifies the Lie bracket relations
\begin{equation}\label{e:bracket}
\begin{aligned}
&[Y,X]=X,\quad &[Y,H]=H,\qquad &[Y,V]=0,\\
&[V,X]=H, \quad &[H,V]=X, \qquad &[X,H]=\rho^2K\, V,
\end{aligned}
\end{equation}
where $K:M\to \R$ is the Gaussian curvature for $g$.
The linear algebra dual coframe is $(\tfrac{1}{\rho^2} \theta,  \tfrac{1}{\rho} d\rho,  \tfrac{1}{\rho^2}\eta, \tau)$, where $\eta = \jmath^*\theta$, and $\tau= \frac{1}{\rho^2}\,g(\,(\,\cdot\,)^{\nu},\jmath\, v)$ is an $S^1$-connection form on every $\Sigma_{r}^g$ with curvature $K\mu$:
\[
\tau(V)=1,\qquad\qquad d\tau = - K\, \pi^*\mu.
\]
The Lorentz force has the expression $\mathbb{Y}(v)=f\jmath\, v$, where $f:N\to \R$ is the unique function satisfying $\sigma=f\mu$. Then, $W=(f\circ \pi) V$ and \eqref{e:mg2} becomes
\begin{equation}\label{e:kappaf}
\kappa_\gamma=\tfrac{1}{r} f(\gamma),
\end{equation}
where $\kappa_\gamma$ is the geodesic curvature of $\gamma$, as follows from the identity $\nabla_{\dot{\gamma}} \dot{\gamma} =\kappa_{\gamma} \, \jmath \,\dot{\gamma}$.
%%%%%%%%%%%%%%%%%%%%%%%%%%%%%%%%%
\subsection{Convexity for twisted cotangent bundles}
\label{Subsection Convexity for twisted cotangent bundles}
%%%%%%%%%%%%%%%%%%%%%%%%%%%%%%%%%
We now investigate when $(D_r^g,\omega)$, for $\omega$ as in \eqref{Equation Magnetic Ct Bdle form omega}, has boundary of positive contact-type, i.e. there is a positive contact form $\alpha_r\in \Omega^1(\Sigma_r^g)$ with $d\alpha_r=\omega|_{\Sigma^g_r}$. To this purpose, we recall the Gysin sequence
\begin{equation}\label{e:gysin}
H^1(N)\longrightarrow H^1(\Sigma_r^g)\longrightarrow H^{2-n}(N;o(TN)) \stackrel{\wedge e}{\longrightarrow} H^2(N)
\end{equation}
where $e\in H^n(N;o(TN))$ is the Euler class of $N$ and $o(TN)$ is the orientation line bundle. The last map is conjugated via the Thom isomorphism to the map
\begin{equation}\label{e:pair}
H^2(D_r^g,\Sigma_r^g)\to H^2(D_r^g),\qquad \red{[x,y]\mapsto [x]}.
\end{equation}
in the long exact sequence of the pair $(D_r^g,\Sigma_r^g)$.

We first investigate when $(D_r^g,\omega)$ can be a Liouville domain. By Lemma \ref{Lemma Liouville condition} this is equivalent to requiring that $[\omega,\alpha_r]=0\in H^2(D_r^g,\Sigma_r^g)$. Thus, \red{the map \eqref{e:pair} shows that a necessary condition is to have $[\omega]=0$, which is equivalent to asking that $\sigma$ is an exact form}. In this case, we define
\begin{equation}\label{Equation r0 definition cotangent bundle}
r_0:=\inf_{d\beta=\sigma}\Vert\beta\Vert,\qquad \Vert\beta\Vert:=\max_{q\in N}|\beta_q|,
\end{equation}
where we run over all primitives $\beta$ for $\sigma$. Note that $r_0$ depends only on $g,\sigma$. For $r>r_0$ and any dimension $n=\dim N$, $D_{r}^g$ admits the deformation \red{through the Liouville domains $(D_r^g,d\theta-s\pi^*\sigma,(\theta-s\pi^*\beta)|_{\Sigma^g_r})$, where $\beta$ is a primitive of $\sigma$ such that $\Vert\beta\Vert<r$. For $s=0$ we get the standard domain $(D_r^g,d\theta,\theta|_{\Sigma_r^g})$}, and thus $SH^*(D_r^g,\omega)\cong SH^*(D_r^g,d\theta)$ (thus, it recovers the ordinary homology of the free loop space of $N$). \red{For $r\leq r_0$, we have two cases according to the dimension of $N$:
\begin{itemize}
\item If $\dim N\geq 3$ determining if $D_r^g$ is Liouville for $r\leq r_0$ is a hard question. We know that $D_r^g$ is not Liouville for $r\in (r_u,r_0]$ where $\tfrac12r_u^2$ is the Ma\~ n\' e critical value of the universal cover of $N$ \cite{CFP}, due to the existence of \red{null-homologous} magnetic geodesics with negative action \cite[Appendix B, Theorem B.1]{Con}. 
\item If $\dim N=2$, $D_r^g$ is not a Liouville domain for $r\leq r_0$, thanks to \cite{CMP}.
\end{itemize}
}
\red{Next, we assume that $(D_r^g,\omega)$ is a non-Liouville convex domain for some contact form $\alpha$ on $\Sigma^g_r$. This means that $[\omega,\alpha]\neq0$. We have two cases:
\begin{itemize}
\item If $\dim N\geq 3$, we claim that this can happen only if $\sigma$ is exact and hence $(D^g_r,\omega)$ would be Quasi-Liouville. To prove the claim, notice that if $n=\dim N\geq3$, then the last map in \eqref{e:gysin} vanishes. Hence, the map in \eqref{e:pair} vanishes as well. This implies that if $(D^g_r,\omega,\alpha)$ is convex, then $[\omega,\alpha]\in H^2(D^g_r,\Sigma^g_r)$ is sent to $[\omega]=0$, which is equivalent to $\sigma$ being exact. All known examples of Quasi-Liouville disc bundles for $\dim N\geq 3$ have $r>r_0$ and contact form at the boundary given by $\alpha:=(\theta-s\pi^*\beta)|_{\Sigma^g_r}+\epsilon\eta$ with $\Vert\beta\Vert<r_0$, $\eta$ a closed 1-form on $\Sigma^g_r$ such that $[\eta]$ is not in the image of the map $H^1(N)\cong H^1(D^g_r)\to H^1(\Sigma^g_r)$ (this ensures that $[\omega,\alpha]\neq0$) and $\epsilon$ small enough. In particular, sending $\epsilon$ to $0$, we see that these are deformations of the Liouville domains described above. It is an open problem to find examples such that $(D^g_r,\omega)$ is Quasi-Liouville for some $r\in(r_u,r_0]$.
\item The situation looks more promising when $\dim N=2$. We have two sub-cases:
\begin{itemize}
\item Assume in addition that $(D_r^g,\omega)$ is a Quasi-Liouville manifold (see Example \ref{ex:QL}). This means that $[\omega]=0$ (namely $\sigma$ is exact). This can happen only if the map in \eqref{e:pair} and hence the last map in \eqref{e:gysin} is not injective, since $[\omega,\alpha]\neq0$. In this case $N$ has to be the two-torus (recall that $H^{0}(N;o(TN))=0$ if $N$ is not orientable and that for orientable surfaces different from $\T^2$ the Euler class does not vanish). Contreras, Macarini and Paternain gave examples in \cite{CMP} of a pair $(g,\sigma)$ on the two-torus, discussed in the next subsection, for which $D_{r_0}^g$ is Quasi-Liouville.
\item Finally, consider the case in which $(D_r^g,\omega)$ is not Quasi-Liouville, meaning that $\omega$, or equivalently $\sigma$, is not exact. This forces $N$ to be an orientable surface. Moreover, $N$ is different from the two-torus since on the two-torus the map in \eqref{e:pair} vanishes as it is represented by multiplication by the Euler class in \eqref{e:gysin}. This yields $[\omega,\alpha]\mapsto[\omega]=0$ contrary to what we assumed. On the other hand, on orientable surfaces different from the two-torus there are examples of convex domains, which are not Liouville. These are discussed in Section \ref{ss:convexorientable}.
\end{itemize}
\end{itemize}
}We now simplify the notation: we will use the dilation $\delta_{r}(q,v):=(q,rv)$ to bring $(D_{r}^g,\Sigma_{r}^g)$ to $(D^g,\Sigma^g):=(D_1^g,\Sigma_1^g)$. The pull-back symplectic form on $D^g$ is 
\[
\delta_{r}^*\omega=r\omega_s,\qquad \omega_s:=d\theta-s\pi^*\sigma,\qquad s:=1/r.
\]
Therefore, we will consider below the symplectic manifold with boundary $(D^g,\omega_s)$ (as $r\omega_s$ and $\omega_s$ have the same Hamiltonian vector fields up to reparametrisation and the same almost complex structures). From Lemma \ref{Lemma Magnetic geodesic vector field}, we get
\[
X_{\rho}=X+s W\qquad \text{on}\ \ \Sigma^g.
\]
Its flow lines $(\gamma,\dot{\gamma})$ can either be interpreted as magnetic geodesics of $(g,s\sigma)$ with speed $1$ via \eqref{e:mg} or as magnetic geodesics of $(g,\sigma)$ with speed $\rho=1/s$ via \eqref{e:mg2}.
%%%%%%%%%%%%%%%%%%%%%%%%%%%%%%%%%%
\subsection{Quasi-Liouville examples using $\T^2$: proof of Theorem \ref{t:cmp}} 
\label{Subsection An example of a Quasi-Liouville magnetic TT2} 
%%%%%%%%%%%%%%%%%%%%%%%%%%%%%%%%%%
We will construct an exact $\sigma$ such that $\omega_s$ has a primitive $\alpha_{s}$ on $\Sigma$ and $(D^g,\omega_s,\alpha_{s})$ is a Quasi-Liouville domain. This is an explicit construction of the contact form for the kind of
systems considered in Contreras-Macarini-Paternain \cite[Section 5.1]{CMP}. 
\red{We first build an angular form $\psi\in T^*(T\T^2\setminus \T^2)$ as follows. Pick a global non-vanishing section $u$ of $\Sigma^g\to \T^2$. Then let $\varphi(q,v)$ denote the angle between $v\neq 0 \in T_q \T^2$ and $u(q)$. Finally, set $\psi:=d\varphi$. One can check that $\psi(V)=1$.}
Explicitly, using properties of the Levi-Civita connection, one can verify \cite[Lemma 2.4]{Benedetti} that $\psi(X)_{(q,v)}=-\nu_q(v)$ where $\nu\in\Omega^1(\T^2)$ is the curvature of the section $u$,
\[
\nu_q(v)=g_q(\nabla_v u,\jmath u).
\]
\red{We now construct $\sigma$.
First, we fix a simple contractible loop of period $T$,} 
\begin{equation}\label{Equation contractible gamma in torus}
\delta:[0,T]\to \T^2,
\end{equation}
\red{where $\delta$ is parametrised by arc-length.} Suppose that its geodesic curvature satisfies
\begin{equation}\label{Equation kappa gamma - nu}
\kappa_\delta-|\nu_{\delta}|>\varepsilon,
\end{equation}
for some $\varepsilon>0$. Then, choose a vector field $B$ on $\T^2$ such that
\begin{enumerate}[\itshape(i)]
 \item $\delta$ is an integral curve for $B$;
 \item $|B_{q}|\leq 1$ for all $q\in \T^2$, with equality precisely on the image of $\delta$.
\end{enumerate}

\red{From this vector field $B$ on $\T^2$ we obtain a one-form
\[
\beta=\flat B.
\]
Finally, define
\[
\sigma=d\beta.
\]
We will call the magnetic system $(D^g,\omega_s,\alpha_{s})$ on $\T^2$ a \textit{QL-magnetic-torus}.}

In this case, the free-period action functional $\mathbb S_r$ is obtained by integrating the Lagrangian function $L+\tfrac12r^2$, where
\[
L(q,v)=\tfrac{1}{2}|v|^2-\beta(v)=\tfrac{1}{2}|v-B|^2-\tfrac{1}{2}|B|^2.
\]
It follows that $L+\tfrac{1}{2}\geq 0$ with equality exactly for $(q,v)=(\delta,\dot\delta)$. Therefore, $\delta$ and its iterates represent the set of global minimizers for $\mathbb S_1$ on the set of contractible closed curves. In particular, $\delta$ is a closed magnetic geodesic with speed $1$.

\begin{lemma}
The value $r_0$ defined in \eqref{Equation r0 definition cotangent bundle} equals $1$ for $(g,\sigma)$ as above.
\end{lemma}
\begin{proof}
We have $r_0\leq \max |\beta|=1$. Let $\beta'$ be any primitive of $\sigma$, and $\hat\delta$ any disc bounding $\delta$. Then, since $\beta(\dot\delta)=\beta(B)=|B|^2=1$,
\[
T=\int_\delta\beta=\int_{\hat\delta}\sigma=\int_\delta\beta'\leq T\Vert\beta'\Vert \Vert\dot{\delta}\Vert= T\Vert\beta'\Vert. \qedhere
\]
\end{proof}

Define $\alpha_{s,a}\in \Omega^1(\Sigma^g)$ by
\[
\alpha_{s,a}=\big(\theta-s\pi^*\beta\big)|_{\Sigma^g}+a\psi.
\]
The relative class $[\omega_s,\alpha_{s,a}]\in H^2(D^g,\Sigma^g)$ is non-trivial for $a \neq 0$, as the form $a\psi$ is a non-exact closed 1-form on $\Sigma^g\cong\T^3$ which does not extend to $D^g$.
%
% H^1(M) -> H^1(end) -> H^2(relative) -> H^2(M)
% ??-theta -> [0,alpha-theta]=[0,\alpha-theta]+D[\theta,0]=[d\theta,alpha]-> [dtheta]=0
% cannot extend alpha globally so rel class nontrivial
%
%bdry of D*T^2 = T^2 x Disc is a T^3, and we use a non-exact closed 1 form on T^3 on the end
%
%

\begin{remark}
In the exact setup, it is possible to study geodesics in a closed manifold $N$ by applying Morse theory for appropriate Lagrangian functionals $L$ to the free loop space $\mathcal{L}N$, see \cite{AbbondandoloSchwarz}. This can also be carried out replacing $\theta$ by $\theta-\pi^*\beta$ if it is a contact form for the sphere bundle $($one then changes $L$ to $L-b$, where $b(q,v)=\beta_q(v))$. However, this fails to be a contact form in the case $N=\T^2$, and that trick does not apply to $\alpha_{s,a}$ because the non-trivial $a\psi$ term cannot be reabsorbed into $L$. 
%
% the psi term essentially calculates the winding number (degree) of dot(gamma) viewed as a map S^1->S^1
%  
\end{remark}

\begin{theorem}\label{Theorem Set A of r,a when is convex}
 The set $A=\{(s,a)\in [0,\infty)\times [0,\infty): \alpha_{s,a}(X+s W)>0\}$ is an open set such that $[0,1)\times\{0\}\subset A$ and the connected component $A_*$ of $(0,0)$ contains a non-empty interval $\{1\} \times (0,a_0)$. For any $(s,a)\in A_*$ in this connected component, $(D^g,\omega_s,\alpha_s:=\alpha_{s,a})$ is a convex domain which can be deformed to the standard $(D^g,d\theta,\theta|_{\Sigma})$ and
\[
SH^*_c(D^g,\omega_s,\alpha_s) \cong H_{2-*}(\mathcal{L}_c\T^2) \cong H_{2-*}(\T^2),
\]
where $c$ is any free homotopy class of loops in $\T^2$ and the latter isomorphism uses the homotopy equivalence $\mathcal{L}_c\T^2\to \T^2$, $\gamma\mapsto \gamma(0)$ 
$($whose fibres $\Omega_c\T^2$ are contractible$)$.
%
% the fiber is \Omega_c T^2
% which is iso by concatenation with \Omega_0 T^2
% which has the same hpy groups as \R^2 (univ.cover of T^2)
% So fibres of hpy equiv are contractible
%
\end{theorem}
\begin{proof}
We compute
\begin{equation}\label{ineq}
\begin{aligned}
\alpha_{s,a}(X+s W) &= 1-s \beta_q(v)+a\big(s f(q)-\nu_q(v)\big)
\\
& \geq 1-s|\beta_q|+a\big(s f(q)-|\nu_q|\big).
\end{aligned}
\end{equation} 
Thus $[0,1)\times\{0\}\subset A$. We now show that $(1-b_0,1+b_0)\times (0,a_0)\subset A$ for some small $a_0,b_0>0$. The right-hand side of \eqref{ineq} is the sum of:
\begin{enumerate}[\itshape (i)]
\item $a(s f(q)-|\nu_q|)$. This is larger than $a\varepsilon$ for $s=1$ and $q$ belonging to the image of $\delta$ by \eqref{Equation kappa gamma - nu} and the identity $f=\kappa_{\delta}$ in \eqref{e:kappaf}. Thus, $a(s f(q)-|\nu_q|)$ is larger than $\frac{1}{2}a\varepsilon$ on a neighbourhood $U$ of the image of $\delta$, if $|s-1|$ is small enough;
\item $1-s|\beta_q|$. This is strictly positive everywhere for $s<1$. For $s=1$ it only vanishes on the image of $\delta$. So $1-s|\beta_q|\geq\epsilon'>0$ on $\T^2\setminus U$, if $|s-1|$ is small.
\end{enumerate}
Thus, the sum is positive in $U$ if $a>0$, and it is positive on $\T^2\setminus U$ for $s\in (1-b_0,1+b_0)$ and $a<a_0$, where $b_0>0$ is sufficiently small and
\[
a_0:=\frac{\epsilon'}{\max\{0,c_0\}},\qquad c_0:=\sup_{q\notin U,\ s\in [1-b_0,1+b_0]}s f(q)-|\nu_q|
\]
So, $(1-b_0,1+b_0)\times (0,a_0)\subset A$. Thus, the sets $[0,1)\times\{0\}$ and $\{1\}\times (0,a_0)$ belong to the same path-connected component $A_*$. By Lemma \ref{Lemma Trick to get convexity}, $(D^g,\omega_s,\alpha_{s,a})$ is a convex domain for all $(s,a)\in A_*$. Therefore, the deformation in the claim arises from a path connecting $(s,a)$ to $(0,0)$ within $A_*$. Applying Corollary \ref{t:invariance3} and Viterbo's theorem \cite{Viterbo}, we deduce the isomorphisms in the statement.
\end{proof}

We will now use Theorem \ref{Theorem Set A of r,a when is convex} to infer existence results about magnetic geodesics. We clarify that $\rho$ is not the radial coordinate $R$ determined by $\Sigma^g$ for the convex domain $(D^g,\omega_s,\alpha_{s,a})$ and, more generally, $\rho$ is not a radial Hamiltonian. However, to prove our results we do not need to find $R$, it suffices to exploit the fact that chain level generators $x$ for $SH^*(D^g,\omega_s,\alpha_{s,a})$ at infinity correspond under projection to $\Sigma^g$ to closed Reeb orbits, which in turn correspond to closed magnetic geodesics $\gamma$ of speed $1$.

Observe that after a time-dependent perturbation of a radial Hamiltonian $h$, the Floer chain complex $CF^*(h)$ (where we suppress $D^g,\omega_s,\alpha_{s,a}$ from the notation) is generated by elements that can be labeled $x^k_-$ and $x^k_+$, where $x^k$ is the $k$-th iterate of a prime magnetic geodesic in $\Sigma^g$ for $k\in \N$ (the labeling uses the above comments about projection to $\Sigma^g$). Following Appendix \ref{Appendix1}, if $x$ has transverse Conley-Zehnder index $\bar{\mu}(x)$, then $x_-$ and $x_+$ have degrees $|x_-|=1-\bar{\mu}(x)$ and $|x_+|=2-\bar{\mu}(x)$, respectively (using that $n=\dim_{\C}D^g=2$).
\begin{lemma}\label{l:duist}
The transverse Conley-Zehnder index of $x$ equals the Morse index of the corresponding magnetic geodesic $\gamma$ for the free-period action functional $\mathbb S_1$.
\end{lemma}
\begin{proof}
Let $(s_0,a_0)$ be a pair in $A_*$ and consider a path $(s,a)$ in $A_*$ joining $(s_0,a_0)$ to $(0,0)$. Let $Y_{s,a}$ be the Liouville vector field of $(D^g,\omega_{s},\alpha_{s,a})$ defined at $\Sigma^g$, so that $Y_{0,0}=Y$. Let $V_{s,a}$ be a nowhere vanishing vector field contained in $\ker \alpha_{s,a}$ such that $V_{0,0}=V$. It is not difficult to see that $Y_{s,a}$ and $V_{s,a}$ can be chosen to depend continuously on $(s,a)$. Let $L_{s,a}$ be the Lagrangian distribution for $\omega_s$ generated by $Y_{s,a}$ and $V_{s,a}$ and observe that $L_{0,0}=T^{\mathrm{vert}}(T\T^2)$ is also a Lagrangian distribution for $\omega_s$. It follows that the relative Maslov index of $L_{s,a}$ with respect to $L_{0,0}$ vanishes. Since $h''>0$ we have that the full Conley-Zehnder index $\mu(x)$ computed with respect to the distribution $L_{s,a}$ is equal to $\mu(x)=\bar\mu(x)+\frac{1}{2}$. Since the relative Maslov index vanishes, $\mu(x)$ is also the Conley-Zehnder index computed with respect to the vertical distribution $L_{0,0}$. By a classical result of Duistermaat \cite{Duistermaat,Weber}, $\mu(x)-\frac{1}{2}$ is the Morse index of $\gamma$ for the \textit{fixed}-period action functional, and as $h''>0$ this is equal to the Morse index $m(\gamma)$ for the \textit{free}-period action functional, so $m(\gamma)=\bar{\mu}(x)$ (see Merry-Paternain \cite{Merry-Paternain} for details).
\end{proof}
\begin{corollary}\label{c:duist}
For all closed Reeb orbits $x$, we have $\bar{\mu}(x)\geq 0$, and thus $|x_-|\leq 1$ and $|x_+|\leq 2$. If the corresponding $\gamma$ is a local minimizer of the free-period action functional $\mathbb S_1$, then $\bar{\mu}(x)=0$.
\end{corollary}

\begin{corollary}\label{Corollary magnetic case on T2}
\red{The following hold for every \textit{QL-magnetic-torus} $(D^g,\omega_s,\alpha_{s})$:}
\begin{enumerate}
\item\label{Item T2 geod} There exists at least one periodic magnetic geodesic with speed $1$ in every non-trivial free homotopy class;
\item If the magnetic geodesic in \eqref{Item T2 geod} is non-degenerate, then there exists at least $2$ such magnetic geodesics;
\item \red{If the contractible periodic magnetic geodesics of unit speed are non-degenerate, then there are infinitely many contractible magnetic geodesics of Morse index one.}
% Since need those to kill off the index 0 orbits you get from gamma and iterates (need finite dimensional homology in the end)
% However, the index 0 orbits split under perturbation into two, index 1 and index 0
% If do Morse Bott model, then cannot have Floer traj from index 0 to index 0.
%
\end{enumerate}
\end{corollary}
\begin{proof}
(1) For free homotopy classes $c\neq 0$, if there are no such geodesics then $SH^*_c(D^g,\omega_1,\alpha_1)=0$, contradicting Theorem \ref{Theorem Set A of r,a when is convex}.

(2) If there is only one non-degenerate geodesic in the class $c\neq 0$ in \eqref{Item T2 geod}, then $2\geq \mathrm{rank}\, SH^*_c(D^g,\omega_1,\alpha_1)$ (recall that each such geodesic contributes two generators to the chain complex after time-perturbation, and we remark that iterates lie in different free homotopy classes).
% Iterates don't arise because the hpy class is non-trivial, so they have different hpy class
This contradicts Theorem \ref{Theorem Set A of r,a when is convex} (we expect rank $4$).

(3) Suppose by contradiction that there are only finitely many prime magnetic geodesics with index $1$. By the iteration formula \eqref{Equation iteration formula1}, if $\bar{\mu}(x)=1$, then $\bar{\mu}(x^k)$ eventually grows for large $k$. So there is a minimal Reeb period $T>0$ such that all prime and non-prime magnetic geodesics with index $1$ have Reeb period $\leq T$ and we denote by \red{$A$} the number of such orbits. \\ 
\textit{Sub-claim}: $\Delta:CF^2(h) \to CF^1(h)$ is injective.\\
\textit{Proof:} Let $0\neq w\in CF^2(h)$, we want $\Delta w\neq 0$.
Let $x_+\in CF^2(h)$ be a generator appearing in $w$ with maximal Reeb period. 
%(there cannot exist an $x_-\in CF^2(h)$ as this would imply $\bar{\mu}(x)=-1$ contracting $\bar{\mu}(x)\geq 0$). 
After rescaling if necessary, we may assume $w=w'+x_+$ where $w'$ does not involve $x_+$. From \eqref{e:deltaDelta}, it follows that $\langle \Delta w',x_-\rangle=0$.
% Only x_+ stuff can interact via Delta with x_-, due to Reeb period filtration and x_+ being in max Reeb period by construction
Thus $\langle \Delta w,x_-\rangle=\langle \Delta x_+,x_-\rangle\neq 0$ by \eqref{e:deltaDelta} and \eqref{e:ab} ($x$ is a good orbit: if $x=x_*^k$ for a prime Reeb orbit $x_*$, then \eqref{Equation iteration formula1} implies $\bar{\mu}(x_*)=0$ since $\bar{\mu}(x)=0$, so the $\bar{\mu}$-values of $x,x_*$ have the same parity). So $\Delta w\neq 0$. $\checkmark$

\red{We take a Hamiltonian function $h$ on the symplectisation of $D^g$ so that \begin{itemize}
\item its Morse complex has generators in degrees $2,1,1,0$, computing $H^*(\T^2)$,
\item the slope $\tau$ of $h$ at infinity is bigger than $T$, so that all generators of the Floer chain complexes involve Reeb periods $\leq \tau$, as will their images under $\partial$ and $\Delta$ (see \ref{Subsection period}).
\end{itemize}}
We use the abbreviation $C_d = CF^d_0(h)$ (notice we restricted to contractible orbits), $\partial_d=\partial|_{C_d}: C_d \to C_{d+1}$ and $\Delta_d = \Delta|_{C_d}:C_d \to C_{d-1}$. \red{There are no generators in degree $3$ by Corollary \ref{c:duist}. In particular, $\partial _2=0$. Observe that for each geodesic $x_0$ with $\bar{\mu}(x_0)=0$ we have $|x_{0-}|=1$, $|x_{0+}|=2$, and for each geodesic $x_1$ with $\bar{\mu}(x_1)=1$ we have $|x_{1-}|=0$, $|x_{1+}|=1$. Since $h$ has two Morse generators in degree $1$ and one in degree $2$, we see that $c(\tau):=\dim C_1 - \dim C_2-1$ equals the number of geodesics $x_1$ with $\bar{\mu}$-index $1$ and period less than $\tau$, since the counts of the $x_{0\pm}$ cancel out in $c(\tau)$. Therefore, by assumption $c(\tau)\leq A$ is bounded independently of $\tau$. Since $\Delta_2$ is injective by the sub-claim, the dimension of $\coker\Delta_2$ is equal to $\dim C_1 - \dim C_2=c(\tau)+1$ and, hence, is bounded by $A+1$. Since $\Delta$ is a chain map (see \ref{Equation BV properties}) and $\partial_2=0$, the map $\Delta_2$ sends $C_2$ into $\ker\partial_1$, namely $\mathrm{Im}\,\Delta_2\subset \ker\partial_1$. Therefore, we obtain $\dim\mathrm{Im}\,\partial_1= \dim C_1-\dim\ker\partial_1 \leq \dim\coker\Delta_2\leq A+1$.}

When we increase the slope $\tau$, we modify $h$ to $h_1$ by only increasing $h$ in the region at infinity where $h'=\tau$. By the maximum principle, this implies that the continuation map $CF^*(h)\to CF^*(h_1)$ for the linear interpolation will be an inclusion of a sub-complex (non-constant continuation solutions lying in the region where $h=h_1$ cannot be rigid as they would admit an $\R$-reparametrization action).
Thus $\mathrm{Im}\,\partial_1$ computed for $CF^*(h)$ is contained in the $\mathrm{Im}\,\partial_1$ computed for $CF^*(h_1)$. Using such Hamiltonians, it follows from the bound \red{$\dim\mathrm{Im}\, \partial_1\leq A+1$} that $\mathrm{Im}\,\partial_1$ eventually stabilises as a vector subspace, independently of $\tau$. Finally, observe that $\dim C_2\to \infty$ as $\tau\to \infty$, because if $y$ is the closed Reeb orbit corresponding to $\delta$, then all the iterates $y_+^k$ have degree $2$, as $\bar{\mu}(y^k)=0$ by Corollary \ref{c:duist}.
Since $\partial_2=0$, it now follows that $SH^2_0(D^g,\omega_1,\alpha_1)$ is infinite dimensional, contradicting Theorem \ref{Theorem Set A of r,a when is convex}.
\end{proof}

\begin{remark}[Alternative Proof]\label{Remark Equivariant SH proof}
Corollary \ref{Corollary magnetic case on T2}.(3) can also be proved using the more elaborated machinery of $S^1$-equivariant symplectic cohomology $ESH^*(D^g,\omega_1,\alpha_1)$ $($we use the conventions from \cite{McLean-Ritter}$)$. Using the Morse-Bott spectral sequence from McLean-Ritter \cite[Cor.7.2]{McLean-Ritter}, aside from the Morse complex of $h$, the $E_1$-page has generators labeled by the unperturbed magnetic geodesics with grading $2-\bar{\mu}$.
% Need to shift 1-\bar{\mu} up by one, since do H^{*-1}(BottMfd/S^1)
There are infinitely many orbits $\delta^k$ in degree $2$, which are cycles as there are no generators in degree $3$. The Morse-Bott spectral sequence converges in degree $2$ to the finite dimensional group $ESH_0^2(D^g,\omega_1,\alpha_1)$, so for dimension reasons there cannot be only finitely many orbits in degree $1$. Here we used that the analogue of Theorem \ref{Theorem Set A of r,a when is convex} yields $ESH^*_c(D^g,\omega_s,\alpha_s)\cong ESH^*_c(D^g,\omega_0,\alpha_0)\cong H_{2-*}^{S^1}(\mathcal{L}_c \T^2)$ where the latter is the $S^1$-equivariant Viterbo theorem \cite{Viterbo}, and we used that $\dim H_{0}^{S^1}(\mathcal{L}_c \T^2)=1<\infty$.
%
% that is the classical dtheta case, so can
% write down Morse Bott sequence
% and see that in degree 2 have finite dim'l chain complex.
% There are probably even easier arguments.
% In fact, since only care about H_0, 
% it is H_0(LcT^2 x_{S^1} S^{infty})
% and that space is connected, so get 1-dimensional H_0
%
\end{remark}

\subsection{Convex domains for $N\neq \T^2$: proof of Theorem \ref{t:nonexact}}\label{ss:convexorientable}

Let $N=S^2$ or a surface of genus\;$\geq 2$. We now work with non-exact magnetic forms. Define
\begin{equation}\label{e:nor}
\mathcal N:=\Big\{(g,\sigma)\ \Big|\ \int_N\sigma=2\pi\chi(N)\Big\}.
\end{equation}
This is not restrictive, since, up to changing orientation of $N$, and rescaling $\sigma$ to $c\sigma$ and $s$ to $s/c$, we can assume that the normalisation above holds. By the Gauss-Bonnet theorem, the form $\sigma':=\sigma-K\mu$ is exact, and for every primitive $\beta$  we get a primitive $\theta_{s,\beta}$ of $\omega_s$ outside of the zero section:
\begin{equation}\label{e:theta}
\theta_{s,\beta}:=\theta-s\pi^*\beta+s\tau,
\end{equation}
where $\tau$ is the $S^1$-connection form. We let
\[
\alpha_{s,\beta}:=\theta_{s,\beta}|_{\Sigma^g}.  
\]
\red{Recall the definition of the function $f:N\to\R$ as the density of $\sigma$ with respect to the area form $\mu$: $\sigma=f\mu$.} We define for every $(g,\sigma)\in\mathcal N$,
\[
s_-(g,\sigma)=\sup_{d\beta=\sigma'}\Big\{ s_*\geq 0 \ \Big|\ \red{\forall\, s\in[0,s_*]},\ \ 1-\Vert\beta\Vert s+(\min f)s^2>0\Big\}.
\]
More explicitly, let $s_-(g,\sigma,\beta)$ be the smallest positive real root of the polynomial $1- \Vert\beta\Vert x + (\min f)x^2$ if it exists, and otherwise let $s_-(g,\sigma,\beta)=+\infty$. Then,
\[
s_-(g,\sigma)=\sup_{d\beta=\sigma'} s_-(g,\sigma,\beta).
\]
Finally, let $A$ be the set of triples $(g,\sigma,s)$ such that $(g,\sigma)\in\mathcal N$ and $s<s_-(g,\sigma)$.
\begin{lemma}
The set $A$ is connected, and $(D^g,\omega_s,\alpha_{s,\beta})$ is a convex domain
for any $(g,\sigma,s)\in A$, where $\beta$ is any primitive of $\sigma'$ with $s<s_-(g,\sigma,\beta)$.
\end{lemma}
\begin{proof}
To see that $A$ is connected, we just observe that if we have an interpolation $(g_u,\sigma_u)$ with $u\in[0,1]$, then we can take a small $s$ such that $1-\Vert \beta_u\Vert_us +(\min f_u)s^2$ is positive for all $u$. To prove that $\alpha_{s,\beta}$ is a positive contact form, we use Lemma \ref{Lemma Trick to get convexity}. Indeed, the formulae in Section \ref{Subsection Basic notation for twisted cotangent bundles} yield for $(q,v)\in\Sigma^g$
\[
	\alpha_{\beta,s}(X+s W)_{(q,v)}=1-\beta_q(v)s+s^2f(q)\geq 1-\Vert \beta\Vert s+s^2\min_{q\in N} f(q). \qedhere
	\]
\end{proof}
When $N=S^2$, we can prove that $(D^g,\omega_s,\alpha_{s,\beta})$ is also convex when $\sigma$ is symplectic and $s$ is large enough. Indeed, let $\mathcal{N}^+\subset \mathcal{N}$ be the subset of those $(g,\sigma)$ for which $\sigma$ is symplectic, equivalently $f>0$. Let 
\[
s_+(g,\sigma)=\inf_{d\beta=\sigma'}\Big\{ s_*\geq 0 \ \Big|\ \forall\, s\geq s_*,\ \ 1-\Vert\beta\Vert s+(\min f)s^2>0\Big\}.
\]
More explicitly, let $s_+(g,\sigma,\beta)$ be the largest positive real root of the polynomial $1- \Vert\beta\Vert x + (\min f)x^2$ if it exists, and otherwise let $s_+(g,\sigma,\beta)=0$. Then,
\[
s_+(g,\sigma)=\inf_{d\beta=\sigma'} s_+(g,\sigma,\beta).
\]
Finally, let $A^+$ be the set of triples $(g,\sigma,s)$ such that $(g,\sigma)\in\mathcal N^+$ and $s>s_+(g,\sigma)$.
\begin{theorem}
Let $N=S^2$. The set $A^+$ is connected, and $(D^g,\omega_s,\alpha_{s,\beta})$ is a convex domain for all $(g,\sigma,s)\in A^+$, where $\beta$ is any primitive of $\sigma'$ with $s>s_+(g,\sigma,\beta)$.\hfill\qed 
\end{theorem}
Having found large sets $A$ and $A^+$ for which the domain is convex, we proceed to find in this class some symmetric examples, for which, we can compute the symplectic cohomology. For this purpose, we pick a metric $\bar g$ on $N$ with $|K|=1$ and let $\bar{\sigma}=K\mu$. We consider the \textit{symmetric twisted symplectic form}
\[
\bar\omega_s := d\theta - s\pi^*\bar\sigma
\]
and the speed Hamiltonian $\rho$ associated to the metric $\bar g$.

For $N=S^2$, we have that $(\bar g,\bar{\sigma})\in \mathcal N^+$.  Moreover, $s_-(\bar g,\bar{\sigma})=s_-(\bar g,\bar{\sigma},0)=+\infty$ and $s_+(\bar g,\bar\sigma)=s_+(\bar g,\bar{\sigma},0)=0$. Thus, $(\bar g,\bar \sigma,s)\in A\cap A^+$, for all $s>0$.

For $N$ a surface of genus\;$\geq 2$, we have $(\bar g,\bar{\sigma})\in\mathcal N$ and $s_-(\bar g,\bar{\sigma})=s_-(\bar g,\bar{\sigma},0)=1$. Thus, $(\bar g,\bar \sigma,s)\in A$, for all $s\in(0,1)$.
\begin{corollary}\label{c:sh}
Let $N=S^2$. For every $(g,\sigma,s)\in A\cup A^+$, the domain $(D^g,\omega_s,\alpha_{s,\beta})$ can be deformed through convex domains to $(D^{\bar g},\bar{\omega}_{\bar s},\alpha_{\bar s,0})$ for any $\bar s>0$.

Let $N$ be a surface of genus\;$\geq 2$. For every $(g,\sigma,s)\in A$, the domain $(D^g,\omega_s,\alpha_{s,\beta})$ can be deformed through convex domains to $(D^{\bar g},\bar{\omega}_{\bar s},\alpha_{\bar s,0})$ for any $\bar s\in(0,1)$.

In both cases, the relative class \eqref{Equation intro relative class}
is constant during the deformation up to a positive factor and up to identifying domains by a fibrewise rescaling. So,
\[
SH^*(D^g,\omega_s,\alpha_{s,\beta})\cong SH^*(D^{\bar g},\bar{\omega}_{\bar s},\bar\alpha_{\bar s,0}).
\]
\end{corollary}
\begin{proof}
A deformation $(D^s,\omega^s,\alpha^s)$ from $(D^{\bar g},\bar{\omega}_{s'},\bar\alpha_{s',0})$ to $(D^g,\omega_s,\alpha_{s,\beta})$ exists since $A\cup A^+$ (respectively $A$) is connected. Performing a rescaling $\psi_s: D^{\bar g}\to D^s$ of the form $\psi_s(q,p)=(q,\lambda_s(q,p)p)$, for some suitable function $\lambda_s:D^{\bar g}\to [0,+\infty)$ we can pull-back all the objects to $D^{\bar g}$. Since $H^2(D^{\bar g},\Sigma^{\bar g})\cong H^2(N)\cong\R$ and $\omega_s$ is non-exact, it follows that there exists also some constant $c_s>0$ such that $(D^{\bar g},c_s\psi_s^*\omega_s,c_s\psi_s^*\alpha_{s,\beta})$ is a deformation with constant relative class, from $(D^{\bar g},\bar{\omega}_{\bar s},\bar\alpha_{\bar s,0})$ to $(D^{\bar g},c_1\psi_1^*\omega_s,c_1\psi_1^*\alpha_{s,\beta})$. Since the factor $c_1$ does not affect the symplectic cohomology, the isomorphism in the statement follows from Theorem \ref{Theorem invariance under iso of SH} and Theorem \ref{t:invariance1}.
\end{proof}
\red{We compute the symplectic cohomology of the symmetric cases in Lemma \ref{l:sh1} and \ref{l:sh2} by a direct and geometric approach which uses that Lemma \ref{l:duist} and Corollary \ref{c:duist} hold also in the present setting, as their proof can be readily adapted. Alternatively, the computation can be done using the deformation invariance of symplectic cohomology proved in Corollary \ref{t:invariance3} as follows. We decrease $s$ to $0$ to deform $(D^{\bar g},\bar\omega_s,\bar{\alpha}_{s,0})$ to $(D^{\bar g},d\theta,\theta|_{\Sigma^{\bar g}})$. Corollary \ref{t:invariance3} with ${\zeta}_0=0$ and ${\zeta}_1=-s\tau(\pi^*\bar\sigma)$ yields
	\[
	SH^*(D^{\bar g},\bar\omega_s,\bar\alpha_{s,0})\cong SH^*(D^{\bar g},d\theta,\theta|_{\Sigma^{\bar g}})_{-s\tau(\pi^*\sigma)}\cong H_{2-*}(\mathcal{L} N)_{-s\tau(\sigma)},
	\]
	where we used the twisted Viterbo isomorphism \cite{Ritter1}. The twisted homology of $\mathcal{L} N$ vanishes for $N=S^2$ since $\sigma$ is not exact \cite{Ritter1}. For a surface of genus\;$\geq 2$, we recover the untwisted homology of $\mathcal{L} N$ since $\sigma$ is atoroidal, meaning $\tau(\sigma)=0\in H^1(\mathcal{L}N)$.}
\begin{lemma}\label{l:sh1}
Consider the symmetric twisted tangent bundle $(TN,\bar{\omega}_s)$, where $N$ is a surface of genus\;$\geq 2$ and $s<1$. The periodic Reeb orbits on $\Sigma^{\bar g}$ are as follows.
\begin{itemize}
\item There is no periodic orbit in the trivial free homotopy class of $D^{\bar g}$.
\item In every non-trivial free homotopy class, there is exactly one periodic orbit, it is transversally non-degenerate, and its transverse Conley-Zehnder index is $0$.
\end{itemize}
It follows that $SH^*_c(D^{\bar g},\bar\omega_{s},\bar{\alpha}_{s,0})\cong H_{2-*}(N)$ if $c=0$ is the trivial free homotopy class, and $SH^*_c(D^{\bar g},\bar\omega_{s},\bar{\alpha}_{s,0})\cong H_{2-*}(S^1)$ if $c\neq 0$.
\end{lemma}
\begin{proof}
The magnetic geodesics with speed $1$ correspond to curves in $N$ with geodesic curvature $s$. Following Hedlund \cite{Hedlund}, such curves have an explicit description, when lifted to the universal cover $\mathbb H$ of $N$ (where $\mathbb H=\{(x,y)\in \R^2\ |\ y>0\}$ is the hyperbolic upper half-plane). The lifted curves are oriented segments of circles that form an angle $\theta\in(0,\frac\pi2)$ between the exit direction and the boundary at infinity $\{y=0\}$ oriented by $\partial_x$, given by $\cos\theta=s$. As for standard geodesics, we know that there are no contractible trajectories and exactly one trajectory in every non-trivial free homotopy class. After reparametrisation, the lifted curves are genuine geodesics for a Finsler metric on $\mathbb H^2$ with negative flag curvature. In particular, each non-contractible periodic orbit is transversally non-degenerate and length-minimizing in its class. Therefore, by Corollary \ref{c:duist}, the transverse Conley-Zehnder index of the associated Reeb orbit is zero.
%
% it's a min of length function so min of energy so has morse index 0
%
% and agrees with CZ index of associated Hamiltonian orbits for Ham obtained via Legendre transform (Lagrangian function is energy in our case), when do projection to contact hypersurface
In particular, each closed orbit is \textit{good} (as the primitive orbits have even index). Thus, after a small time-dependent perturbation of the Hamiltonian it yields a Floer subcomplex with the homology of $S^1$.
\end{proof}

\begin{lemma}\label{l:sh2}
Consider the symmetric twisted tangent bundle $(TS^2,\bar{\omega}_{s})$ and the primitive $\bar\theta_{s,0}$ as in \eqref{e:theta}, where $s>0$. The radial coordinate induced by integrating the Liouville flow of $\theta_{s,0}$ starting from $\Sigma^{\bar g}$ is defined globally on $T^*S^2$ via
\[
R_{s}(q,v) = \sqrt{\frac{|v|^2+s^2}{1+s^2}}\ ,
\]
$(TS^2,\bar\omega_{s},\bar\theta_{s,0})$ is the completion of $(D^{\bar g},\bar\omega_{s},\bar\alpha_{s,0})$ and $SH^*(D^{\bar g},\bar\omega_s,\bar\alpha_{s,0})=0$.
\end{lemma}
\begin{proof}
Let $Z_s$ denote the Liouville vector field of $\bar\theta_{s,0}=\theta+s\tau$, which means that $\iota_{Z_s}\omega_s=\theta+s\tau$. Denote $r_s$ the coordinate defined on the complement of the zero section by the flow of $Z_s$ with $r_s=0$ along $\Sigma^{\bar g}$. By definition $R_s=e^{r_s}$. We differentiate the function $\rho$ along a flow line of $Z_s$, using Lemma \ref{Lemma Magnetic geodesic vector field},
\begin{equation*}
\frac{d\rho}{dr_s}=d\rho(Z_s)=-\omega_s\Big(\frac{1}{\rho}(X+s W),Z_s\Big)=\frac{1}{\rho}(\theta+s\tau)(X+s W)=\frac{1}{\rho}(\rho^2+s^2).
\end{equation*}
Multiplying both sides by $\tfrac{\rho}{\rho^2+s^2}$ and integrating from $0$ to $r_s$ yields the claimed formula for $r_s=\log R_s$.
Note $r_s\to \infty$ as $\rho\to \infty$, so the flow of $Z_s$ is positively complete.
We now compute the symplectic cohomology. The closed Reeb orbits on $\Sigma^{\bar g}$ correspond to curves on $S^2$ with geodesic curvature $s$. An explicit computation in geodesic polar coordinates shows that all such trajectories are periodic with common minimal period 
$
T=2\pi/\sqrt{1+s^2}.
$
We consider the sequence of Hamiltonians, for $k\in1+2\pi\N$,
\[
h_k:TS^2\to\R,\qquad h_k(q,p)=k\sqrt{1+s^2}\cdot R_s(q,v) = k \sqrt{\rho^2+s^2}.
\]
The Hamiltonian vector field is
$
X_{h_k}=k(\rho^2+s^2)^{-1/2}(X+s W\big).
$
The associated flow defines a Hamiltonian $S^1$-action on $(T^*S^2,\omega_s)$ with minimal period $2\pi/k$. Hence, the only $1$-periodic orbits of the flow are the constant orbits, which lie in the zero section. One could now compute the Conley-Zehnder indices explicitly. One can bypass this, by mimicking the argument in \cite{Ritter2} (compare also \cite[Section 2.6]{McLean-Ritter}): changing the slope $k$ to $k+2\pi$ will decrease the indices by $2$ (one looks at how the linearized flow for $h_k$ acts on a trivialisation of the anti-canonical bundle, and one notices that it has winding number one). Finally by considering the direct limit, one concludes that symplectic cohomology vanishes in each degree.
\end{proof}
\begin{proof}[\textbf{Proof of Theorem \ref{t:nonexact}}]
Let $(g,\sigma,s)\in A\cup A^+$ for $N=S^2$, or $(g,\sigma,s)\in A$ for $N$ of genus\;$\geq 2$. The computation of $SH^*(D^g,\omega_s,\alpha_{s,\beta})$ in the statement follows from Corollary \ref{c:sh} and Lemmas \ref{l:sh1}-\ref{l:sh2}. We now prove the statements about the existence of closed magnetic geodesics with speed $1/s$ for the pair $(g,\sigma)$.

For $N$ of genus\;$\geq 2$, if there were no such curve in a free homotopy class $\nu\neq 0$, we would obtain the contradiction $H_{2-*}(S^1)\cong SH^*_\nu(D^g,\omega_s,\alpha_{s,\beta})=0$.

Now let $N=S^2$. If there were no closed magnetic geodesics of speed $1/s$, we would obtain the contradiction $H^*(S^2)\cong SH^*(D^g,\omega_s,\alpha_{s,\beta})=0$.
We now prove that there are at least two prime periodic magnetic geodesics with speed $1/s$, assuming all the periodic orbits are transversally non-degenerate. 
Suppose by contradiction that $x$ is the only such geodesic. We do a case-by-case analysis of indices, using \eqref{Equation iteration formula1}:
\begin{enumerate}[\itshape a)]
	\item $\overline{\mu}(x)\leq 0$. Then $\overline{\mu}(x^k)\leq 0$, for all $k\in \N$. Thus, the non-constant orbits in the Floer chain complexes have grading $|x_{\pm}^k|\geq 1$, which would imply that $SH^0(T^*S^2,\omega_s,\alpha_{s,\beta})\cong H^0(T^*S^2)$, contradicting Theorem \ref{t:nonexact}.
	%%%%
	\item $\overline{\mu}(x)\geq 3$ with $x$ hyperbolic. Then $\overline{\mu}(x^{k+1})-\overline{\mu}(x^k)\geq 3$. It follows for grading reasons that $x^k_+$ is a cycle, and it is not a boundary unless it arises from the Floer differential applied to $x^k_-$. But in the local Floer complex for $x^k$, we have $\partial x^k_-=0$ whenever $x^k$ is a good orbit by \eqref{e:ab}, and we can always ensure that $x^k$ is good (if $x^k$ is a bad hyperbolic orbit, we replace $k$ by $k+1$). 
	% only an issue in hyperbolic case, when the parity of CZ(gamma) and CZ(gamma^k) are different parity, where gamma is the primitive Reeb orbit. In that bad case, the boundary of min is 2 times the max.
	%%%%
	\item $\overline{\mu}(x)\geq 3$ with $x$ elliptic. Here $\overline{\mu}(x)\geq 3$ forces $\widetilde{\Delta}\geq 1$ and non-degeneracy implies $\widetilde{\Delta}\not\in \Q$, so $\widetilde{\Delta}>1$. So, for some sufficiently large $k$, $\overline{\mu}(x^{k+1})-\overline{\mu}(x^k)\geq 4$. The proof follows as in the previous case (using that elliptic orbits are always good). 
	%%%%
	\item $\overline{\mu}(x)=2$. Then all iterates $x^k$ are good hyperbolic orbits. 
	% since iterates have same parity as primitive orbit x 
	The $x^k_{\pm}$ and the two generators of the Morse complex for $S^2$, give generators in gradings
	%
	% 1-2k = -1,-3,-5,-7,-9,...
	%
	% perturb
	% 0,-1,-2,-3,-4,-5,-6,-7,-8,-9,...
	%
	$2,0,0,-1,-2,-3,...$
	which cannot be acyclic in degrees $2$ and $0$, contradicting Theorem \ref{t:nonexact}. 
	% 
	%$2,0,0$ and $-m$ for each $m\geq 1$, in particular $|x_-^2|=-3$, $|x_+^2|=-2$ and $\partial x_-^2 = 0$ since it is a good orbit. Thus the complex cannot be acyclic, contradicting Theorem \ref{t:nonexact}.  
	
	\item $\overline{\mu}(x)=1$ with $x$ hyperbolic. 
	%
	% 1-k = 0,-1,-2,-3,...
	%
	% perturb
	% 1,0,0,-1,-1,-2,-2,-3...
	%
	Generators' gradings: $2,1,0,0,0,-1,-1,-2,...$, which by rank-nullity cannot be acyclic either in degree $2$ or $0$ (or both).
	% look at degree zero
	
	\item $\overline{\mu}(x)=1$ with $x$ elliptic. Then $0<\widetilde{\Delta}<1$. Suppose $\widetilde \Delta<\tfrac{1}{2}$. Then the $\bar{\mu}=1$ orbits are $x^1,x^2,\ldots,x^a$ for some $a\geq 2$, thus $\bar{\mu}(x^{a+1})=3$, and recall iterates of an elliptic orbit are good. Let $m_2,m_0$ denote the Morse critical points in degrees $2,0$. The restriction of the differential to $\Lambda x_+^1 \oplus \Lambda x_+^a \to \Lambda m_2$ must have non-trivial kernel by rank-nullity, thus we obtain a cycle $y=\lambda_1x_+^1 + \lambda_2 x^a_+\neq 0$, for some $\lambda_1,\lambda_2\in\Lambda$. As symplectic cohomology vanishes, there is a chain $z$ with $\partial z=y$. By \eqref{e:deltaDelta} and \eqref{e:ab}, this can happen only if $\lambda_2=0$, as $x^a$ is good and has maximal period among the orbits with $\bar{\mu}=1$. Moreover, $\Delta z=0$ by \eqref{e:deltaDelta} as all orbits in grading $-1$ have Reeb period strictly larger than those with grading $0$.
	% (we may discard the cycle $m_0$ if it appears in $z$).
	By \eqref{Equation BV properties}, $\Delta y=\Delta\partial z=-\partial\Delta z=0$. However, $y=\lambda_1x_+^1$ and, therefore, $\langle \Delta y,x_-^1\rangle = \lambda_1 \neq 0$ by \eqref{e:ab}, contradiction.
	Now suppose $\widetilde \Delta>\tfrac{1}{2}$. There are $m,k\in\Z$ such that $x^{k-1}$ is the only orbit with $\overline{\mu}=m$ and $x^k$ and $x^{k+1}$ are the only orbits with $\overline{\mu}=m+2$. The previous argument applies with $x^{k-1}_-$ in place of $m_2$, $x^k_+$ in place of $x_+^1$, and $x^{k+1}_+$ in place of $x^a_+$. \qedhere
\end{enumerate}
\end{proof}
%\begin{remark*}
%It is unknown whether there are simple orbits homotopic to a fibre of $\Sigma_r$.
%\end{remark*}
%
\begin{remark}[Alternative Proof]
The last case above can be proved using $ESH^*$ as in Remark \ref{Remark Equivariant SH proof}. Suppose $x^1,x^2,\ldots,x^{a_1}$ have $\bar{\mu}=1$, and $x^{1+a_{k-1}},\ldots,x^{a_k}$ have $\bar{\mu}=2k-1$ for $k\geq 1$. Recalling the two Morse critical points, the number of generators in degrees $(2,1,0,-1,\ldots)$ after perturbation is $(1,a_1,a_1+1,a_2,a_2,a_3,a_3,\ldots)$. 
Consider the $E_1$-page of the Morse-Bott spectral sequence for $ESH^*$ \cite[Cor.7.2]{McLean-Ritter}. The Morse complex for $S^2$ contributes generators in degrees $2+2\Z_{\leq 0},0+2\Z_{\leq 0}$ due to the formal variables $u^{-m}$ in degree $-2m$. Each non-constant $S^1$-orbit with index $1-\overline{\mu}$ contributes one copy of $H^{*-1}(S^1/S^1)=H^{*-1}(\mathrm{pt})$ in grading $*+1-\overline{\mu}$ $($using \cite[Thm.4.1]{McLean-Ritter}$)$. 
The total number of generators in degrees $(2,1,0,-1,\ldots)$ is $(1,a_1,2,a_2,2,a_3,2,\ldots)$. We use two facts explained in \cite{McLean-Ritter}: the vanishing of symplectic cohomology implies the vanishing of the $S^1$-equivariant symplectic cohomology; and the equivariant Morse complex for $S^2$ constitutes a subcomplex.
% could also say the floer differential must decrease the radial coordinate
Thus, by Theorem \ref{t:nonexact}, the spectral sequence converges to zero, and the $E_1$-page considered with total gradings must satisfy the same rank-nullity conditions as an acyclic complex. This implies $a_1=1$ (the degree $0$ generators in the subcomplex cannot kill a non-constant orbit, and the degree $1$ orbit will eventually kill the Morse index $2$ critical point) and thus $a_2=a_3=\cdots=2$. Using \eqref{Equation iteration formula1},
 $2m\widetilde{\Delta}$, $(2m+1)\widetilde{\Delta}$ must lie in the open interval $(m,m+1)$ for all $m\geq 1$. So $\widetilde{\Delta}\in (\frac{m}{2m},\frac{m+1}{2m+1})$. Letting $m\to \infty$ yields the contradiction $\frac{1}{2}=\widetilde{\Delta}\notin \Q$.
$\qedhere$
%2  
%   11111111   
%0  00000000
%            -1-1-1-1-1
%            -2-2-2-2-2
%                           ...
% 
\end{remark}

\appendix
%%%%%%%%%%%%%%%%%%%%%%
\section{From the magnetic $T^*S^2$ to the Hyperk\"{a}hler $T^*\C P^1$}
\label{Subsection From the magnetic TS2 to the HK TCP1}
%%%%%%%%%%%%%%%%%%%%%%

The tangent bundle $T^*\C P^1 \to \C P^1$ is isomorphic as a complex line bundle to $\OO(-2)\to \C P^1$. After picking a Hermitian metric on $\OO(-2)$, we can ensure that this identification is $S^1$-equivariant (where $S^1\subset \C^*$ acts naturally by rotation in the complex fibres) and preserves the norm $\rho=|p|$.
The curvature form $\sigma$ on $T^*\C P^1$ then satisfies $\frac{1}{2\pi}[\sigma]=c_1(\mathcal{O}(-2))=-2\omega_{FS}\in H^2(\C P^1)$ where $\int_{\C P^1}\omega_{FS}=1$, and let $\tau$ be the associated angular form for $\OO(-2)$.

Fix a metric $g$ on $S^2\cong \C P^1$ of constant Gaussian curvature one and identify the real vector bundle $T^*S^2$ with $TS^2$ as in Section \ref{Section Application: Twisted cotangent bundles of surfaces}. Note however that, since $\int_{S^2} \sigma=-4\pi$, the induced rotation $\jmath:TS^2 \to TS^2$ is rotation by $-\frac{\pi}{2}$ compared to the usual orientation for $\C P^1$, and $\sigma=\mu$ where $\mu=g(\cdot,\jmath \,\cdot)$ is the area form of $g$ with respect to $\jmath$ in the notation of Section \ref{Section Application: Twisted cotangent bundles of surfaces}.

Following the conventions in \cite[Section 7.3]{Ritter4}, we can construct a symplectic form $\omega=d\tau+\varepsilon d(\rho^2\tau)$ on $TS^2$ for $\varepsilon>0$, where $d(\rho^2\tau)$ is fiberwise the area form and we have $d\tau=-\pi^*\sigma$. On the zero section, $\omega$ restricts to $-\pi^*\sigma$, therefore $[\omega]=-\pi^*\sigma\in H^2(TS^2)$.
% n pi*omegaB = -pi^*sigma = dtau
%
% dtheta = n pi^*omegaB = 2 pi*omegaB = "-c_1" = -pi^*sigma = dtau
% Omega = d(r^2 theta) = 2r dr dtheta
Thus, away from the zero section, $\omega=d((1+\varepsilon \rho^2)\tau)$. By replacing $\varepsilon = 1/2s$ for $s>0$, and rescaling the symplectic form by $s$, we redefine the symplectic form by
\[
\widetilde{\omega}_s = d\big((\tfrac{\rho^2}{2}+s)\tau\big).
\]
 Thus $\widetilde{\omega}_s$ restricts to $-s\pi^*\sigma$ on $S^2$, just like the magnetic symplectic form $\omega_s = d\theta-s\pi^*\sigma$, and so $[\widetilde{\omega}_s]=[\omega_s]\in H^2(TS^2)$.

The form $\widetilde{\omega}_s$ can be identified with the Hyperk\"{a}hler form $\omega_I$ for $TS^2\cong T^*\C P^1$ viewed as an asymptotically locally Euclidean manifold \cite{Ritter2}, for which the zero section and the fibres are holomorphic submanifolds. So $\widetilde{\omega}_s$ makes the zero section and the fibres of $TS^2$ both symplectic submanifolds; $d\theta$ makes them both Lagrangian; and $\omega_s$ makes the zero section symplectic but keeps the fibres Lagrangian. 

\begin{theorem}
	There exists a diffeomorphism $F_s:TS^2\to TS^2$ preserving the zero section (but not the fibres) such that 
	\begin{equation}\label{eq:psis}
	F_s^*\big((\tfrac{\rho^2}{2}+s)\tau\big) =\theta+s\tau.
	\end{equation}
	In particular, we can identify the magnetic $(TS^2,\omega_s=d\theta-s\pi^*\sigma)$ with the negative line bundle $(\mathcal{O}_{\C P^1}(-2),\widetilde{\omega}_s)$ and $SH^*(TS^2,\omega_s)\cong SH^*(\mathcal{O}_{\C P^1}(-2),\widetilde{\omega}_s)$. The latter is known to vanish by \cite{Ritter4}, consistently with Theorem \ref{t:nonexact}.
\end{theorem}
\begin{proof}
	We follow the ideas in \cite[Section 2]{Benedetti2} and refer to Section \ref{Subsection Preliminary notation for cotangent bundles} and \ref{Subsection Basic notation for twisted cotangent bundles} for the notation.
	 We denote by $\Sigma_{\rho}$ an arbitrary level set of $\rho$. From \eqref{e:bracket} and the general formula $d\alpha(U_1,U_2) = U_1\cdot \alpha(U_2)-U_2\cdot \alpha(U_1) - \alpha([U_1,U_2])$, for a $1$-form $\alpha$ and vector fields $U_1,U_2$, we get
	 \begin{equation}\label{e:bracket2}
	 d\theta=\tau\wedge \eta,\qquad d\eta=\theta\wedge \tau,\qquad d\tau=\tfrac{1}{\rho^2}\eta\wedge\theta\qquad \text{on}\ \ T\Sigma_{\rho},
	 \end{equation}
	 where the last equality follows since $K=1$.
	 For $a>0$, consider the rescaling
	\[m_a:TS^2\to TS^2, \quad m_a(q,v)=(q,av),\]
	which satisfies
	\begin{equation}\label{e:ma}
	m_a^*\tau=\tau,\qquad \partial_a m_a=\tfrac1a Y,\qquad dm_a\cdot H=aH.
	\end{equation}
	Denote by $\Phi_t$ the flow at time $t$ of the vector field $-H=(-\jmath\, v)^h$. The integral curves of $\Phi_t$ are $t\mapsto (\gamma(t),\jmath\, \gamma'(t))$, where $t\mapsto \gamma(t)$ is a geodesic in $S^2$ for $g$. Indeed, $\jmath\,\gamma'$ is a parallel field along $\gamma$ and $\gamma'=-\jmath\, v$ with $v=\jmath\, \gamma'$. We claim that 
	\[
	\Phi_b^*\tau = -\frac{\sin(b\rho)}{\rho}\,\theta + \cos(b\rho)\,\tau. 
	\]
	\textit{Proof of claim.} Say $\Phi_b^*\tau = x \theta + y \eta + z \tau$ for smooth $x,y,z$ depending on $b$ (and write $x'$ etc.~to denote derivatives in $b$). Here we used that $\Phi$ preserves the tangent bundle of $\Sigma_\rho$ , which is spanned by $X,H,V$ and the dual space is spanned by $\theta,\eta,\tau$. Multiplying that equation by $(\Phi_b^*)^{-1}=\Phi_{-b}^*$, differentiating in $b$ yields
	\[
	x' \theta + y'\eta + z' \tau=x\mathcal L_{H}\theta+y\mathcal L_{H}\eta+z\mathcal L_{H}\tau=-\rho^2 x\tau+ z\theta,
	\] where we used Cartan's formula together with \eqref{e:bracket2}.
 Thus $y$ is constant in $b$, and $x'=z$, $z'=-\rho^2 x$. This implies $z''=-\rho^2z$ and $x=z'$. Since $\Phi_1^*=\mathrm{Id}$, we have $x_1=y_1=0$, $z_1=1$, and therefore $x=\sin(b\rho)/\rho$, $y=0$ and $z=\cos(b\rho)$. $\checkmark$
	
	We now claim that the following diffeomorphism satisfies \eqref{eq:psis} for some smooth functions $a_s:[0,\infty)\to [0,\infty)$, $b_s:[0,\infty)\to \R$ that we must determine:
	\[
	F_s = m_{a_s(\rho)}\circ \Phi_{b_s(\rho)}.
	\]
	\textit{Proof of claim.} Abbreviate $c_s(\rho)=a_s(\rho)^2\cdot\frac{\rho^2}{2}+s$. The pull-back
	$F_s^*((\tfrac{\rho^2}{2}+s)\tau)$ equals 
	\[
	\Phi_{b_s(\rho)}^*m_{a_s(\rho)}^*\big((\tfrac{\rho^2}{2} +s)\tau\big)
	= c_s(r)
	\Big(-\frac{\sin(b_s(\rho)\rho)}{\rho}\,\theta + \cos(b_s(\rho)\rho)\,\tau\Big).
	\]
	Observe indeed that
	\[
	dF_s=\tfrac{1}{b_s(\rho)}Y\otimes d(a_s(\rho))-a_s(\rho)d m_{a_s(\rho)}\cdot H\otimes d(b_s(\rho))+dm_{a_s(\rho)}d\Phi_{b_s(\rho)}
	\]
	and $\tau$ vanishes on the first two terms. 
	To satisfy \eqref{eq:psis} we want $c_s(\rho)\cos(b_s(\rho)\rho)=s$ and $c_s(\rho)\sin (b_s(\rho)\rho)=-\rho$. Squaring and adding gives $c_s(\rho)=\sqrt{\rho^2+s^2}$, whereas taking the quotient implies $\tan (b_s(\rho)\rho)=-\frac{\rho}{s}$. Therefore, we find 
	\[
	a_s(\rho)=\frac{1}{\rho}\sqrt{2(\sqrt{\rho^2+s^2}-s)}=\sqrt{\frac{2}{\sqrt{\rho^2+s^2}+s}},\quad b_s(\rho)=-\frac{1}{\rho}\tan^{-1}\Big(\frac{\rho}{s}\Big)=-\frac{1}{s}u(\tfrac{\rho^2}{s^2}),
	\]
where $u:[0,+\infty)\to (0,1]$ 
% consider limit rho -> 0
is the unique strictly decreasing smooth function such that $u(x^2)=\frac{\tan^{-1}(x)}{x}$. Thus, the functions $a_s:[0,+\infty) \to (0,\sqrt{2/s}]$ and $b_s:[0,+\infty)\to [-1/s,0)$ are strictly monotone and $a_s\circ \rho:TS^2\to (0,\sqrt{2/s}]$ and $b_s\circ \rho:TS^2\to [-1/s,0)$ are globally smooth. Therefore, the map $F_s$ is a diffeomorphism satisfying \eqref{eq:psis} and the claim is established.
\end{proof}
%%%%%%%%%%%%%%%%%%%%%%%%%%%%%%%%%%%%%%%%%%%%%%%%%%%%%%%%%%%
%%%%%%%%%%%%%%%%%%%%%%%%%%%%%%%%%%%%%%%%%%%%%%%%%%%%%%%%%%%
\section{Iteration formula for CZ-indices in dimension $4$}
\label{Appendix1}
%%%%%%%%%%%%%%%%%%%%%%%%%%%%%%%%%%%%%%%%%%%%%%%%%%%%%%%%%%%
%%%%%%%%%%%%%%%%%%%%%%%%%%%%%%%%%%%%%%%%%%%%%%%%%%%%%%%%%%%
Let $(M,\omega)$ be convex and let $n=\tfrac{1}{2}\dim M$.
For autonomous (i.e. time-independent) Hamiltonians $H:M\to \R$, any non-constant $1$-periodic orbit $x$ will be degenerate as there is at least an $S^1$-family of such orbits obtained by time-translation $x(\cdot+\textrm{constant})$. Such an orbit $x$ is called \textit{transversally non-degenerate} if the 1-eigenspace of $d_{x(0)}\varphi_H^1$ is one-dimensional, i.e. equals $\R\cdot X_H$. Assume now that $x$ is transversally non-degenerate. Then the family of $1$-orbits near $x$ is parametrised by $S^1$. Following the conventions\footnote{Except, we declare $\omega(\cdot,X_H)=dH$, which means that our grading by $n-\mu_{\mathrm{CZ}}$ will agree with the Morse index of $H$ when $H$ is a time-independent, $C^2$-small Morse function.}
of Salamon's notes \cite{Salamon}
%
% Mention these notes, since there CZ and RS are consistent, whereas in CFHW the convention is to
% have CZ and RS opposite in the non-deg case.
%
 one can define a Conley-Zehnder index\footnote{more precisely, the Robbin-Salamon index \cite{Robbin-Salamon}.} $\mu(x)$ associated to the linearisation $d_{x(0)} \varphi_H^t$ of the Hamiltonian flow, written as a path of symplectic matrices $\psi_t$ by symplectically trivializing the symplectic vector bundle $x^*TM$ compatibly with a trivialisation of the canonical bundle.\footnote{We work under the assumption that $c_1(M)=0$. Thus the canonical bundle $\mathcal{K}=\Lambda_{\C}^\mathrm{top} T^*M$ is trivial, using $c_1(\mathcal{K})=-c_1(M)$. A trivialisation of $\mathcal{K}$ is specified by a choice of global non-vanishing section $\Omega$ of $\mathcal{K}$, whereas a trivialisation of $x^*TM$ is specified by $n$ linearly independent sections $v_1(t),\ldots,v_n(t)\in T_{x(t)}M$. Let $f(t)=\Omega(v_1(t),\ldots,v_n(t))$. The obstruction for the two trivializations to be compatible, is the homotopy class of the phase map $f/|f|:S^1\to S^1$.}
 
Assume $x(t)=(y(t),R)\in M^{\out}=\Sigma \times [1,\infty)$ lies in the conical end, and that $H=h(R)$ is radial there with $h''\geq 0$ (with equality for large $R$, where $h$ is linear). Note $X_H=h'(R)Y$ is a multiple of the Reeb field $Y$, so $\overline{y}(t)=y(t/T)$ is a Reeb orbit of period $T=h'(R)$. Decompose $TM^{\out}\cong \xi \oplus \R Z \oplus \R Y$, where $\xi$ is the contact structure and $Z=R\partial_R$ is the Liouville field (noting that the oriented basis for $\xi^{\perp}$ is $Z,Y$, not $Y,Z$). We assume $n\geq 2$ (i.e. $\xi \neq 0$), 
so one can pick a basis of independent sections for $x^*\xi$ which together with $Y,Z$ yield a trivialisation of $x^*TM^{\out}$ compatible with the trivialisation of $x^*\mathcal{K}$. In that basis,
\[
\psi_t =  \overline{\psi}_t \oplus \left(\begin{smallmatrix} 1\;\; & 0 \\ t c\cdot h''(R)\;\; & 1 \end{smallmatrix} \right)
\]
where $c>0$ is a positive constant, and where $\overline{\psi}_t$ is the path of symplectic matrices obtained by trivializing the contact distribution $\xi$ along the orbit. As $h''>0$, the latter symplectic shear contributes $+\tfrac{1}{2}$ to the CZ-index \cite{Robbin-Salamon}. 
%
% the contribution of shear is -0.5 sign(tch'' at end) + 0.5 sign(tch'' at start) but in our case at the start we get the identity which contributes 2/2=1 (half the number of evals=1) by convention.
%
% notice when concatenate, tch'' for t in [1,2] then [2,3] etc will always have same sign, so will not contribute more
%
% Careful: iteration is not a concatenation of symplectic paths associated to each iterate, because you start from dphi^1 when you do next iterate
% 
%
Thus 
\[
\mu(x) = \overline{\mu}(x) + \tfrac{1}{2},
\]
where $\overline{\mu}$ denotes the CZ-index of $\overline{\psi}_t$.

By \cite{CFHW}, a suitable time-dependent perturbation of the Hamiltonian, supported near the orbits, will break the $S^1$-family of orbits into two orbits $x_-,x_+$ corresponding to the minimum and maximum of a Morse function on the $S^1$-parameter space. In our convention, this affects our grading by $-\tfrac{1}{2}\dim S^1 + \textrm{(corresponding Morse index)}$. So 
\[
|x_-| = n - \mu(x)-\frac{1}{2}\dim S^1=n-\overline{\mu}(x)-1, \qquad \qquad |x_+|=|x_-|+1.
\]
\indent Denote $x^k:S^1 \to M$ the Hamiltonian orbit corresponding to the $k$-th iterate of $x$ (which appears when $h'=kT$, and corresponds under projection to the $k$-th iterate $y^k:[0,kT]\to \Sigma$). 
For $x^k$ the shear part still only contributes $+\tfrac{1}{2}$, whereas we need an iteration formula for the $\overline{\psi}_t$-summand. 

Now assume $\dim M=4$, so $\dim \xi=2$, $\overline{\psi}_t\in \mathrm{Sp}(2)$, and assume all iterates of $x$ are transversally non-degenerate. The following is a consequence of \cite[Appendix 8.1]{HWZ-Index}\footnote{In our case, all iterates of $x$ are transversally non-degenerate, so the elliptic case involves a rotation matrix by an irrational angle $\widetilde{\Delta}$ (times $2\pi$) and the first iterate has index $2k+1$ where $k=\lfloor \widetilde{\Delta}\rfloor$, whereas the $k$-th iterate involves a rotation matrix with angle $k \widetilde{\Delta}$ and so has index $2 \lfloor k \widetilde{\Delta}\rfloor+1$.% 
%If the angle had been an integer, the ``$+1$'' would be dropped.
}
%
%
%L'idea di perchè funziona è la seguente. Usiamo la definizione dell'indice di HWZ: se S_t è il path di matrici simplettiche, allora per ogni u in R² diverso da zero guardo il sollevamento dell'angolo che c'è tra u e S_1u dato dall'omotopia S_tu. Chiamo Int l'insieme di tutti i possibili angoli/2\pi al variare di u. Si sa che Int è un intervallo di lunghezza minore di 1/2. Se Int è contenuto nell'intervallo (k,k+1) allora l'indice del path S è 2k+1. Ora l'idea è che se S_1 è una matrice simplettica con autovalori sul cerchio unitario, allora posso trovare un prodotto scalare su R² compatibile con la struttura simplettica standard in modo che S_1 è una matrice in SO(2) corrispondente alla rotazione di un angolo irrazionale theta. Quindi Int={theta} e l'indice è 2\lfloor theta\rfloor +1. Il fatto che S_1 sia esattamente una matrice di rotazione implica che S_1^k è anche una matrice di rotazione con angolo k theta. Applicando la formula si trova che l'indice in questo caso è 2\lfloor k theta\rfloor +1.
%
\begin{equation}\label{Equation iteration formula1}
\begin{array}{lcl}
\overline{\mu}(x^k) & = &
{
\begin{cases}
                         2\lfloor k\widetilde{\Delta}\rfloor+1 \textrm{ for some } \widetilde{\Delta}\in \R\setminus\Q &\ \mbox{if $x$ is elliptic},\\
                         k \overline{\mu}(x)&\ \mbox{if $x$ is hyperbolic}.
                        \end{cases}
}
\\[5mm]
|x^k_-| & = &
{
\begin{cases}
                             %n-1-2\lfloor...
                         -2\lfloor k\widetilde{\Delta}\rfloor &\ \mbox{if $x$ is elliptic},\\
                             %n-k\overline...
                         1- k \overline{\mu}(x),&\ \mbox{if $x$ is hyperbolic}.
                        \end{cases} 
}
\\[5mm]
  |x^k_+|&=&|x^k_-|+1.
\end{array}
\end{equation}

\end{document}